\documentclass[11pt]{amsart}

\usepackage{amsmath}
\usepackage{hyperref}
\usepackage{xy}
\usepackage{amscd}
\newfont{\sheaf}{eusm10 scaled\magstep1}

\newtheorem{definition}{Definition}[section]

\newtheorem{proposition}{Proposition}[section]
\newtheorem{corollary}[proposition]{Corollary}
\newtheorem{lemma}[proposition]{Lemma}
\newtheorem{theorem}[proposition]{Theorem}
\newtheorem{example}{Example}[section]
\newtheorem{remark}{Remark}[section]

\DeclareMathOperator{\Supp}{Supp}
\DeclareMathOperator{\Sing}{Sing}

\DeclareMathOperator{\Pic}{Pic}

\DeclareMathOperator{\rank}{rank}
\title {Rational Parametrizations of Moduli Spaces of Curves}
 \author{Alessandro Verra}
  \address{Universit\'a Roma Tre, Dipartimento di Matematica, Largo San Leonardo Murialdo \hfill
\indent 1-00146 Roma, Italy}
 \email{{ verra@mat.uniroma3.it}}
\thanks {Supported by Ministero dell' Istruzione, Universit\'a e Ricerca of Italy:  PRIN project `Geometria delle variet\'a algebriche e dei loro spazi di moduli'  cofin-2008}
\begin{document}
\maketitle

\tableofcontents

\section{Introduction}
 \par 
Rational parametrizations of  an algebraic variety are early  themes in the history of Algebraic Geometry, present since its  remote origins.
As is well known a rational parametrization of an algebraic variety $X$, defined over a field $k$, is a dominant rational map
$$
f: k^n \dashrightarrow X.
$$
The variety $X$ is said to be unirational if such a rational map exists and rational if $f$ is invertible. Unirational varieties are of course of a very special type, since they satisfy a very special requirement. The requirement is elementary,
but the search for rational parametrizations, or for detecting the rationality of a given $X$, is often a rather difficult and deep question. For this reason these questions repeatedly played  a crucial role  in the evolution of Algebraic Geometry and still this cyclic history does not seem finished at all. \par Nowadays unirational varieties are appropriately inserted in a wider, more approachable class of algebraic varieties, namely rationally connected varieties. In particular,  the recent notion of rational connectedness is responsible for an important change of perspective on this subject, see e.g. \cite{K}. \par
On the other hand,  the knowledge on how to afford the (uni)rationality  problem for a given $X$ appears, nowadays too, as very much indebted to a geometric, quite ad hoc study of the main series of classical examples. Among them
certainly we have moduli spaces of algebraic varieties of various types and features. In particular curves and their related  moduli spaces became, in some sense,  special actors of this subject. \par
This paper aims to give an account, both historical and geometric, on the diverse geography of rational parametrizations of moduli spaces related to curves.  We have recollected, as well as reconnected, several unirationality or rationality constructions for some of these moduli spaces, trying to respect chronology and geometry. A historical report on the different attempts to realize these parametrizations is unwound along the way. It has been also natural to partially extend the picture to related questions, concerning for instance the Kodaira dimension or the families of  uniruled subvarieties of the moduli spaces to be considered. 
It is clear that an exhaustive account of this type is beyond the size of this work. This implies that we had to make choices omitting many more subjects.  \par
Elementary examples of algebraic varieties having a unirational moduli space are represented by hypersurfaces and complete intersections in $\mathbf P^n$ of fixed degree or type.  The coefficients of their defining equations  provide indeed a rational parametrization of the corresponding moduli space. \par In particular,  this simple fact singles out the side of unirational moduli spaces, overshadowing the other  possibilities. In the case of curves too, this situation is a leit motiv of the history we are going to outline. \par A further leit motiv along the paper is the presence of a superabundant family of classical, or less classical, examples. We were fascinated by their beauty and this was a reason, certainly not unique, for insisting on them. \par \par  We work over the complex field.  Just  to fix notations we introduce the main moduli spaces to be considered. The starting point is the \it moduli spaces
$
\mathcal M_g
$
of curves \rm of genus $g \geq 2$. Over $\mathcal M_g$ we have the \it universal Picard variety \rm
$$
p:  \Pic_{d,g} \to \mathcal M_g
$$
which is  the moduli space of pairs $(C,L)$ such that $C$ is a smooth, irreducible genus $g$ curve and $L \in  \Pic^d(C)$. Besides $\mathcal M_g$  we will deal with the \it moduli spaces $\mathcal S_g$ of spin curves \rm and with the \it Prym moduli spaces $\mathcal R_g$. \rm \par All these spaces embed in $ \Pic_{d,g}$ as multisections of $p$, for $d = g-1$ and for $d = 0$ respectively. We recall that $\mathcal R_g$ is the moduli space of pairs $(C,L)$ such that $L$ is non trivial and $L^{\otimes 2} \cong \mathcal O_C$. The space $\mathcal R_g$ is irreducible. On the other hand $\mathcal S_g$ is the moduli space of pairs $(C,L)$ such that $L \in  \Pic^{g-1}(C)$ is a theta characteristic, that is,  $L^{\otimes 2} \cong \omega_C$. The theta-characteristic $L$ is said to be \it even \rm (respectively  \it odd \rm) if $h^0(L)$ is even (odd). The space $\mathcal S_g$  splits as the disjoint union 
$$
\mathcal S^+_g \cup \mathcal S^-_g,
$$
where $\mathcal S^+_g$ and $\mathcal S^-_g$ are irreducible. They  parametrize even (respectively odd) theta characteristics. All the previous moduli spaces naturally interact with the moduli  $\mathcal A_g$ of principally polarized abelian varieties of dimension $g$ and with the moduli $\mathcal F_g$ of K3 surfaces endowed with a genus $g$ polarization. The role played by $\mathcal A_g$ and $\mathcal F_g$ is crucial in our survey, though we have not extended to them the discussion about rational parametrizations.  We can consider, in some sense, all the above mentioned moduli spaces as the classical ones. \medskip \par 
The three chapters of this paper are:  \begin{enumerate} \it \item[2.] Moduli of curves, \item[3.] Moduli of spin curves, \item[4.] Prym moduli spaces. \end{enumerate} 
Let us give a summary of them. The starting point is Severi's conjecture that $\mathcal M_g$ is (uni)rational for every $g$, see \cite{S} p.880. The attempts for proving this conjecture strongly influenced the studies on curves and their moduli during the past century. Moreover they constitute very important roots and motivations to contemporary work in this field.  For these reasons we closely follow the history stemming from Severi's conjecture, even if this was finally disproven after sixty years. Therefore we discuss results on families of singular plane curves and Severi varieties. Before Mumford, Harris and Eisenbud disproved the conjecture, a way to approach it was to construct unirational families of singular plane curves of genus $g$ with general moduli.  We will describe the first attempts in this direction, due to
Severi and then to Beniamino Segre. \par Many contemporary studies and authors on moduli of curves are directly linked to Segre's work. We refer in particular to subjects like moduli of $k$-gonal curves for small $k$ and to Segre-Nagata 
conjecture. So these themes are present in our historical account. \par
Severi observed that the unirationality of $\mathcal M_g$ follows, for $g \leq 10$,  from the rationality of the varieties of nodal plane curves of genus $g \leq 10$ and degree $d$, where $d = [\frac 23g] + 2$ is the minimal degree so that a plane curve of genus $g$ and degree $d$ can be general in moduli.  \par The next attempt is due to Segre. Instead of nodal curves, he tried to use unirational families of plane curves with arbitrary singularities. His methods are very interesting and deep, though the results are on the negative side.  They concern questions of the following type. Let $$ \Sigma \subset \vert \mathcal O_{\mathbf P^2}(d) \vert $$ be an integral family of singular plane curves of geometric genus $g$. Assume that the natural map $\Sigma \to \mathcal M_g$ is dominant. Is it possible that \par  \it
$\circ$ $\Sigma$ is a linear system? \par
$\circ$ $\Sigma$ is a scroll? \par
$\circ$ the points of $ \Sing \ \Gamma$ are general if $\Gamma$ is general in  $\Sigma$? \rm\\
Furthermore let $C \subset S$ be an embedding of a general curve of genus $g$ in a smooth, rational surface $S$. A question underlying the previous ones is: \par \it $\circ$ when $\dim \vert C \vert \ > 0$ and $\vert C \vert$ is not  an isotrivial family? \\\rm
We will address this matter in detail along the paper. We point out that,  dropping the request that $S$ be rational out,  the latter question is just equivalent to: \par
$\circ$ \it is $\mathcal M_g$ uniruled? \rm\\
This brings us to a further step in our survey. Following the history, we will describe the advent of K3 surfaces in the study of moduli of curves and Mukai realizations of canonical curves of low genus as linear sections of homogeneous spaces. Beyond K3 surfaces, let $S$ be any smooth surface and let $C \subset S$. Then we will extend our picture to all what is known in the case $C$ has general moduli and $\vert C \vert$ is not isotrivial. \par  The goal of this preparation is clear:  families of pairs $(C,S)$ are uniruled, since $\dim \vert C \vert > 0$, and often unirational. Building on them,  the final part of the chapter is devoted to the promised geometric constructions for parametrizing $\mathcal M_g$ in low genus. We describe the known uniruledness results for $g \leq 16$, the diverse unirationality constructions for $g \leq 14$ and the rational connectedness of $\mathcal M_{15}$. \par Along the way we come to prove the ruledness of various universal Brill-Noether loci and the unirationality of $ \Pic_{d,g}$, where $g \leq 9$. Some non exhaustive discussions on the rationality problem for $\mathcal M_g$ is also included. \par
Chapters 3 and 4 are written from a similar point of view and in the same spirit. In the case of $\mathcal M_g$ the transition from the uniruledness cases to the cases where $\mathcal M_g$ is of general type is still not completed, since the Kodaira dimension is still unknown for $g \in [17, 21]$. Instead, for the moduli spaces of spin curves $\mathcal S^+_g$ and $\mathcal S^-_g$, the situation has been recently settled. \par In chapter 3 we describe all the unirationality and uniruledness constructions for these moduli spaces. Then we describe the transition from the uniruledness to the case where $\mathcal S^{\pm}_g$ is of general type. \par This appears to be specially interesting in the
case of even spin curves. Here the Kodaira dimension is zero in genus 8, while $\mathcal S^+_g$ is of general type for $g \geq 9$ and uniruled for $g \leq 7$. \par The transition does not admit intermediate cases for the moduli of odd
spin curves, where we have the uniruledness for $g \leq 11$ and a variety of general type for $g \geq 12$. Nevertheless we discuss some possible peculiarities of the uniruled variety $\mathcal S^-_{11}$. \par
Behind the previous picture we have again the geometry of K3 surfaces and of their hyperplane sections. The uniruledness results follow, for some moduli of odd spin curves, just because a general curve of geometric genus 
$g \leq 11$ is a smooth, nodal in genus 10, hyperplane section of a K3 surface: see 2.3 for more details.\par  In the case of moduli of even spin curves the situation is more delicate and the uniruledness of $\mathcal S^+_g$ is 
related to a special class of K3 surfaces, namely Nikulin K3-surfaces. One can show that, with the exception of genus 6, a general curve of genus $g \leq 7$ is a smooth hyperplane section of a Nikulin surface. This is then crucial
to deduce the uniruledness for $g \leq 7$ . \par In genus 8, hyperplane sections of Nikulin surfaces form a codimension two proper closed set in $\mathcal S^+_8$. The beautiful geometry of genus 8 curves, in particular the fact that they are linear  sections of the Grassmannian
$G(2,6)$, is very much responsible, in some sense,  for the moduli space $\mathcal S^+_8$ having Kodaira dimension zero. \par
In chapter 4 we come a little bit more on the side of rationality results. We report on the rationality of $\mathcal R_g$,  which is known for $g \leq 4$. Here there are different beautiful proofs, due to different authors and involving a lot of classical geometry.  \par An important role is played here by the study of embeddings $C \subset S$, where $C$ is a curve of genus $g$ with general moduli, $S$ is a surface  endowed with quasi \'etale double cover $\pi: \tilde S \to S$
and $C$ is a very ample Cartier divisor. The line bundle defining $\pi$ restricts to a non trivial $L$ on $C$ such that $L^{\otimes 2}$ is trivial. Therefore $\pi$ induces a natural map $\vert C \vert \to \mathcal R_g$. \par
We discuss different rationality results  for $\mathcal R_g$, $g = 2,3,4$ and the rationality of Prym moduli spaces of hyperelliptic curves. Then we pass to some unirationality results, focusing in particular on $\mathcal R_6$. We describe two different geometric approaches. \par On one side we use the beautiful geometry of Enriques surfaces. An Enriques surface$S$ is endowed with an  \'etale double cover $\pi: \tilde S \to S$ defined by $\omega_S$. We use Fano polarizations on $S$, that is, very ample linear systems $\vert C \vert$ on $S$ of curves of genus 6. We prove the unirationality of the moduli space of pairs $(S, \mathcal O_S(C))$.  Then we prove that the family of pairs $(S,C)$ dominates $\mathcal R_6$ and deduce its unirationality. \par
A second method, which takes the flavor of families of nodal plane curves, is the attempt to parametrize $\mathcal R_g$ via  rational family of nodal conic bundles over $\mathbf P^2$. We consider $\delta$-nodal conic bundles $q: T \to \mathbf P^2$ satisfying the condition that each fibre of $q$ is a conic of rank $\geq 2$ and the general one is smooth. The discriminant curve of $q$ is then
$
\Gamma := \lbrace x \in \mathbf P^2 \ / \ \mbox{rank} \ q^*(x) = 2 \rbrace.
$
Let $C$ be the normalization of $\Gamma$. $C$ is endowed with an \'etale double covering $\pi: \tilde C \to C$, parametrizing the irreducible components of $q^*(x), x \in \Gamma$. \par We consider conic bundles such that $\pi$ is non trivial. Then we construct explicitely linear systems of $\delta$-nodal conic bundles dominating $\mathcal R_g$,  $g \leq 6$. \par 
In all the exposition we rely on the bibliography for complete proofs. Some minor novelties are nevertheless present: examples and suggestions, a proof of  the unirationality of $\mathcal R_g$, $g \leq 6$. \par  \smallskip
\noindent\it Acknowledgements  \rm \par \noindent
I am indebted to Andrea Bruno and Edoardo Sernesi for their helpful reading of the ultimate version of this paper and for many interesting conversations on it.  Let me thank the referee, and also the editors, for their
patient work. \par

\section{Moduli of curves}
 \subsection{Origins and a conjecture of Severi} 
 In May 1915 Severi publishes   `Sulla classificazione delle curve algebriche e sul teorema di esistenza di Riemann' \cite{S}.  This paper, as Severi says,  summarizes his recent results on the birational classification of algebraic curves. Admittedly, a wider and complete publication on the same subject was postponed to better times, since World War I was rapidly approaching Italy in those days. The paper contains a sentence  which is the starting point of a long history: 
\begin{quote}``Ritengo probabile che la variet\'a $H$ sia razionale o quanto meno che sia riferibile ad un' involuzione di gruppi di punti in uno spazio lineare $S_{3p-3}$; o, in altri termini, che \it nell' equazione di una curva piana di genere $p$ (e per esempio dell' ordine $p+1$) i moduli si possano far comparire razionalmente.''
\end{quote}
 \rm \par   
In the above text, $H$ is the moduli space of curves $\mathcal M_g$. Severi conjectures that  $\mathcal M_g$ is probably rational or at least unirational. The idea is that there exists
an irreducible family $\mathcal P$ of plane curves of equation
$$
\sum_{0 \leq i,j \leq d} f_{ij}X^iY^j = 0,
$$
such that: \medskip \par \begin{enumerate} \item The general member of $\mathcal P$ is birational to a smooth and irreducible projective curve of genus $g$, \item   the $f_{ij}$'s are rational functions of $3g-3$ parameters, \item the corresponding natural map $f: \mathcal P \to \mathcal M_g$ is dominant. \end{enumerate} \medskip \par 
In the same paper it is pointed out that such a family exists for $g \leq 10$. More precisely the author remarks that there exists a rational family 
$$
\mathcal P_{_{g,d}} ,
$$
of plane projective curves of degree $d$ and geometric genus $g$, which dominates $\mathcal M_g$ via the natural map into the moduli. This implies that: \medskip \par  
{\bf Theorem} (Severi) \it $\mathcal M_g$ is unirational if $g \leq 10$. \rm \medskip \par 
The proof suggested in the paper relies on the irreducibility of $\mathcal M_g$ and Brill-Noether theory, two established results for the standards generally accepted at that time.  Let us describe it, after some long preliminaries. \par 
For any curve $C$ of genus $g$  we will denote its Brill-Noether loci as follows,
$$ W^r_d(C) := \lbrace L \in  \Pic^d(C) \ / \ h^0(L) \geq r+1 \rbrace. $$ 
By the Brill-Noether theory $W^r_d(C)$ is not empty if $\rho(g,r,d) \geq 0$, where $\rho(g,r,d) := g - (r+1)(g-d+r)$ is the Brill-Noether number. Assume $C$ is a \it curve with general moduli \rm. Then it is well known that
$$
 \dim \ W^r_d(C) \ = \  \max \lbrace -1, \rho(g,r,d) \rbrace
$$
and that  $W^r_d(C)$ is integral if $ \dim \ W^r_d(C) > 0$. Moreover any $L \in W^r_d(C)$ defines a morphism birational onto its image if $r \geq 2$, cfr. \cite{ACGH} ch. V.
For $r = 2$ we have  $$ \rho(g,r,d) \geq 0 \Longleftrightarrow \frac 32 d - 3 \geq g.
$$
Hence a curve $C$ with general moduli is birational to a plane curve $\Gamma$ of degree $d$, with  $ \frac 32 d \geq g + 3$. As is well known such a curve $\Gamma$ is nodal. \par 
\begin{definition} A curve $\Gamma$ is nodal if each $x \in  \Sing \ \Gamma$ is an ordinary double point. $\Gamma$ is $\delta$-nodal if it is
nodal and $\delta$ is the cardinality of $ \Sing \ \Gamma$.
\end{definition} \par 
Counting linear conditions one expects that a $\delta$-nodal curve $\Gamma$ of degree $d$ exists as soon as $ \ \dim \vert \mathcal O_{\mathbf P^2}(d) \vert - 3\delta \geq 0$ that is
$$
\binom{d+2}2 - 3\delta > 0.
$$
Assume the latter inequality and that  $g \leq \frac 32d - 3$. Then the genus formula $$ g = \binom{d-1}2 - \delta $$ implies that
$$   
g \leq \frac 32 d - 3 \    \text {and}  \   \frac {d^2}3 - \frac 72d + 2 \leq 0.
$$
The next proposition then follows.
\begin{proposition} Let $\Gamma$ be a $\delta$-nodal curve of degree $d$ which is general in moduli and such that $ \dim \ \vert \mathcal O_{\mathbf P^2}(d))\vert - 3\delta \geq 0$. Then
the latter inequalities hold true. In particular they imply $g \leq 10$ and $d \leq 9$.
  \end{proposition} \par 
Thus the space for a rational parametrization of $\mathcal M_g$, suggested by counting linear conditions, drops down to $g \leq 10$ because of Brill-Noether theory . \par We want to discuss more of this situation.   Let $\Gamma \subset \mathbf P^2$ be an integral $\delta$-nodal  curve of geometric genus $g$.  We will assume $d \geq 3$,  so that $d$ is uniquely defined from $g$ and $\delta$.  
A first problem in the study of families of nodal curves $\Gamma$, is to understand the configuration of their nodes. To this purpose we consider the Hilbert scheme of $ \Sing \ \Gamma$ and its open subset parametrizing 0-dimensional smooth schemes of length $\delta$. This will be denoted as
$$
Hilb_{\delta}(\mathbf P^2).
$$
In what follows $Z$ denotes an element  of $Hilb_{\delta}(\mathbf P^2)$ and $\mathcal I_Z$ its ideal sheaf.  
 \begin{definition} Assume $\delta \leq \binom{d-1}2$. The family of plane curves
$$
\mathcal N_{g, \delta} := \lbrace \Gamma \in \vert \mathcal O_{\mathbf P^2}(d) \vert \ / \ \Gamma \ \text{ is integral, $\delta$-nodal of geometric genus $g$} \rbrace
$$
is the Severi variety of $\delta$-nodal plane curves of genus $g$ and degree $d$.
\end{definition} \par 
So far $\mathcal N_{g, \delta}$ is a locally closed set endowed with the diagram
$$
\begin{CD}
{\mathcal N_{g, \delta}} @>{m_{g,\delta}}>> {\mathcal M_g} \\
@V{h_{g,\delta}}VV @. \\
{Hilb_{\delta}(\mathbf P^2)} @. \\
\end{CD}
$$
Here $m_{g,\delta}$ is the map sending $\Gamma$ to its moduli point in $\mathcal M_g$ and $h_{g, \delta}$ is the map sending $\Gamma$ to the element $ \Sing \ \Gamma$ of
$Hilb_{\delta}(\mathbf P^2)$.
\begin{definition} $m_{g,\delta}$ and $h_{g,\delta}$ are the natural maps of $\mathcal N_{g, \delta}$. \end{definition} \par 
Consider the universal family $ \mathcal Z \subset \mathbf P^2 \times Hilb_{\delta} (\mathbf P^2) $ of $Hilb_{\delta}(\mathbf P^2)$. Let $\mathcal I_{\mathcal Z}$ be its ideal sheaf and let
$p_1: \mathcal Z \to \mathbf P^2$, $p_2: \mathcal Z \to Hilb_{\delta}(\mathbf P^2)$ be its projection maps. It is clear that we have an open embedding
$$
\mathcal N_{g, \delta} \subset \mathbb P_{g, \delta},
$$
in the projective bundle
$$
\mathbb P_{g, \delta} := \mathbf P p_{2*}(p_1^* \mathcal O_{\mathbf P^2}(d) \otimes \mathcal I_{\mathcal Z}).
$$ 
Moreover the natural projection $\overline h_{g, \delta}: \mathbb P_{g, \delta} \to Hilb_{\delta}(\mathbf P^2)$ restricts to $h_{g, \delta}$ on $\mathcal N_{g, \delta}$. The fibre of $\overline h_{g, \delta}$
at $Z$ is  $\vert \mathcal I^2_Z(d) \vert$, possibly a point. \par It is important to stress that $\mathcal N_{g, \delta}$ could be reducible, or even empty a priori. \rm \par  However  assume that $U$ is an irreducible component of $\mathcal N_{g, \delta}$ and let
 $$ V \subset \mathbb P_{g, \delta}$$ 
be its closure.   Then it follows by semicontinuity that $$ \overline h_{g, \delta}/V: V \to \overline h_{g, \delta}(V) $$ is a projective bundle over  
 an open set of $h_{g, \delta}(V)$.  Since the Hilbert scheme $Hilb_{\delta}(\mathbf P^2)$ is rational, we conclude that: \medskip \par  \begin{itemize} \it
\item[$\circ$] If both $h_{g, \delta}/V$ and $m_{g, \delta}/V$ are dominant then $\mathcal M_g$ is unirational. \end{itemize} \rm \medskip \par 
This principle summarizes the method outlined by Severi for proving the unirationality of $\mathcal M_g$, $g \leq 10$. For $g \leq 10$ the method is effective, the result follows via a case by case proof or, alternatively, via suitable algorithms. \par
Of course good general reasons are known, nowadays, ensuring that $m_{g, \delta}$ and $h_{g, \delta}$ cannot be simultaneously dominant, 
with finitely many exceptions. \par
On the other hand the arguments classically in use put in evidence, for every $g$ and $\delta$, some global properties of the structure of $\mathcal N_{g, \delta}$ which are implicitely expected but not granted a priori. This is the case for the following:
\medskip \par 
\underline {\it Key questions or assumptions} \it \medskip \par
\begin{itemize} \it
\item[$\circ$] $\mathcal N_{g, \delta}$ is irreducible, in particular not empty.
\item[$\circ$] the tangent map of $h_{g, \delta}$ has generically maximal rank. 
 \end{itemize} \medskip \par  \rm
Both these questions appear to be crucial in the study of families of algebraic curves. We do not address here their long and interesting history, if not for quoting some fundamental answers to them.  
We recall that $\mathcal N_{g, \delta}$ has a natural structure of scheme. Let $\Gamma \in \mathcal N_{g, \delta}$ and $Z =  \Sing \ \Gamma$. The projective completion of the tangent space to $\mathcal N_{g,\delta}$ at $\Gamma$ is just the linear system
$$
\vert \mathcal I_Z(d) \vert.
$$
The irreducibility of $\mathcal N_{g, \delta}$ was claimed by Severi in his famous Anhang F of \cite{S1}. The first complete proof of this property appeared much later and it is due to  J. Harris, \cite{H2}. Actually one has:
\begin{theorem} $\mathcal N_{g,\delta}$ is integral of codimension $\delta$ in $\vert \mathcal O_{\mathbf P^2}(d) \vert$. \end{theorem} \par  
Coming to the second question, we can rephrase it as follows. Since the fibre of $\overline h_{g, \delta}$ at $Z$ is $\vert \mathcal I^2_Z(2) \vert$, the question we are speaking about  is whether such a linear system has minimal
dimension. This means
$$
 \dim \ \vert \mathcal I_Z^2(d) \vert  \ = \  max \ \lbrace -1 \ , \  \dim \ \vert \mathcal O_{\mathbf P^2}(d) \vert - 3\delta \rbrace,
$$
the dimension predicted by the so called postulation. The number $3\delta$ is indeed the number of linear conditions, to be imposed on a linear system of curves on a smooth algebraic surface, for constructing the
linear subsystem of singular curves passing through a fixed set $N$ of  $\delta$ points. Even when
$$
 \dim \ \vert \mathcal O_{\mathbf P^2}(d) \vert \geq  3\delta
$$
it is not granted that the previous equality holds. Consider for example the Severi variety $\mathcal N_{1,9}$ of integral 9-nodal plane sextic curves. In this case we have that $ \dim \vert \mathcal O_{\mathbf P^2}(d) \vert - 3\delta = 0$.
Hence it follows that the tangent map of
$$
h_{1,9}: \mathcal N_{1,9} \to Hilb_9(\mathbf P^2)
$$
is generically of maximal rank  if and only if $h_{1,9}$ is dominant. This is false: take a general $N \in Hilb_9(\mathbf P^2)$. Then there is a unique element $E \in \vert \mathcal I_N^2(6) \vert$ as expected. But this is not an element of
$\mathcal N_{1,9}$ because $E$ is twice the unique plane cubic through $N$. Hence it is not nodal. \par
On the other hand let $\Gamma \in h_{1,9}(\mathcal N_{1,9})$ and $Z =  \Sing \ \Gamma$. Then $\vert \mathcal I_Z^2(6) \vert$ is a Halphen pencil of plane elliptic curves. This pencil is generated by $\Gamma$
and by the unique double plane cubic containing $Z$. \par
The answer to our second question is however known: the previous example is in fact the unique exception. In other words we have
\begin{theorem} Let $Z \in Hilb_{\delta}(\mathbf P^2)$ be general. Then
$$
 \dim \ \vert \mathcal I_Z^2(d) \vert  \ = \  \max \ \lbrace -1 \ , \  \dim \ \vert \mathcal O_{\mathbf P^2}(d) \vert - 3\delta \rbrace,
$$
unless $d = 6$ and $\delta = 9$.
\end{theorem} \par 
 A modern proof of the theorem is due to Arbarello and Cornalba,\cite{AC3}, while the same result appears in Terracini \cite{T}, cfr. \cite{CC1} section 1. In particular $\vert \mathcal I_Z^2(d) \vert = \emptyset$ is empty, for a general $Z$,  if and only if $ \ \dim \vert \mathcal O^2_S(d) \vert - 3\delta < 0$ or $d = 6$ and $\delta = 9$. 
\begin{corollary} Let $ \dim \ \vert \mathcal O_{\mathbf P^2}(d) \vert - 3\delta \geq 0$. Then the  tangent map of $h_{g, \delta}: \mathcal N_{g, \delta} \to Hilb_{\delta}(\mathbf P^2)$ is generically of maximal rank
unless $(g, \delta) = (1,9)$.
\end{corollary} \par 
 It follows immediately that:
\begin{proposition} The next conditions are equivalent if $(g, \delta) \neq (1,9)$:
\begin{enumerate}
\item $h_{g,\delta}: \mathcal N_{g, \delta} \to Hilb_{\delta}(\mathbf P^2)$ is dominant,
\item $ \ \dim \ \vert \mathcal O_{\mathbf P^2}(d) \vert - 3 \delta \geq 0$.
\end{enumerate}
\end{proposition}
 \begin{corollary}  Assume $ \dim \ \vert \mathcal O_{\mathbf P^2}(d) \vert - 3\delta \geq 0$ and $(d, \delta) \neq (1,9)$. Then the Severi variety
 $\mathcal N_{g, \delta}$ is rational.
\end{corollary} 
\begin{proof} The assumption implies that $h_{g, \delta}$ is dominant. On the other hand $\mathcal N_{g,\delta}$ is birational to a 
projective bundle over $h_{g, \delta}(\mathcal N_{g,\delta})$. Since $Hilb_{\delta}(\mathbf P^2)$ is rational, the statement follows. \end{proof} \par 
Relying on this basis, see also \cite{AS} for further precisions, Severi's unirationality result follows:
\begin{theorem}[Severi] $\mathcal M_g$ is unirational for $g \leq 10$. \end{theorem} 
\begin{proof} By proposition 2.1 both $h_{g, \delta}$ and $m_{g, \delta}$ are dominant at least if: \begin{itemize} \it  \item[$\circ$] $d = 6$ and $g \leq 6$, \item[$\circ$] $d = 8$ and $7 \leq g \leq 8$,
\item[$\circ$] $9 \leq g \leq 10$ and $d = 9$. \end{itemize} In this range $\mathcal M_g$ is dominated by $\mathcal N_{g, \delta}$, which is rational  by corollary 2.6. \end{proof}
\begin{remark} \rm The same proof applies to the universal Brill-Noether loci $\mathcal W^2_{d,g}$. Hence they are unirational for all the values of
$(d,g)$ satisfying the conditions considered in proposition 2.1. This means that $\mathcal W^2_{d, g}$ is unirational in the following range $g \leq \frac 32d - 3 < 11$. 
More in general the unirationality of $\mathcal W^2_{d,g}$ holds true for $g \leq 9$ and $d \geq g + 2$. This follows because then $\mathcal W^2_{d,g} \cong  \Pic_{d,g}$
and the latter space is unirational for $g \leq 9$, \cite{Ve1}.  On the other hand the bound $g \leq 9$ is sharp for the unirationality of $\mathcal W^2_{d,g}$,
$d \geq g + 2$, because $ \Pic_{d,g}$ has non negative Kodaira  dimension if $g \geq 10$ and $(d, 2g-2) = 1$,  \cite{BFV}.  It seems plausible that the unirationality of
$\mathcal W^2_{d,g}$ always holds true for $g \leq 9$.  \end{remark}

 It is the aim of this exposition to emphasize down to earth constructions and examples.  Therefore we present some concrete rational parametrizations of $\mathcal M_g$,
$g \leq 10$, lying behind the previous theorem. We use families of  linear systems $  \vert \mathcal I^2_Z(d) \vert $ of nodal curves  of genus $g$ and \it minimal  degree \rm $d$ such that $d \geq$ $\frac 23g + 2$. Here $Z$ is a set of $\delta$  points  in general position and  $\mathcal I_Z$ is its ideal sheaf. The set $Z$ moves appropriately along a subvariety of $Hilb_{\delta}(\mathbf P^2)$. We recall that a $0$-dimensional subscheme $Z \subset \mathbf P^2$ is in \it general position \rm if  the restriction map $H^0(\mathcal O_{\mathbf P^2}(n)) \to H^0(\mathcal O_Z(n))$ has maximal rank for each $n$.
\begin{example} \rm \ $g \leq 6$: \ in this case we can use a \it unique \rm linear system $\vert \mathcal I^2_Z(d) \vert$ to parametrize $\mathcal M_g$. 
We omit the standard proof of the next statement.
\begin{proposition} Let $\delta \geq 4$ and let $Z \subset \mathbf P^2$ be a fixed set of $\delta$ points in general position.  Then, for $g = 10 - \delta$, $\mathcal M_g$ is dominated by the natural map
$$
m_g: \vert \mathcal I_Z^2(6) \vert \to \mathcal M_g.
$$ 
\end{proposition} \par 
Let us describe the structure of $m_g$ for $g$ equal to 6 or 5. \medskip \par  
\underline {\it Genus six} \ \  The equation of $\vert \mathcal I_Z^2(6) \vert$ can be fixed as
$$
z_1(l_2l_3l_4)^2+ z_2(l_1l_3l_4)^2 + z_3(l_1l_2l_4)^2 + z_4(l_1l_2l_3)^2 + al_1l_2 + bl_3l_4 + cl_1l_2l_3l_4 = 0
$$
where  $(l_1:l_2:l_3)$ are coordinates on $\mathbf P^2$, $l_4 = l_1 + l_2 + l_3$ and the forms $a, b, c$ are as follows:  $a \in \mathbf C[l_1^2,l_2^2]$, $b \in \mathbf C[l_3^2, l_4^2]$, $c \in \mathbf C[l_1, l_2, l_3]$.  The base locus of $\vert \mathcal I^2_Z(6) \vert$ is the set
$B = \lbrace l_1l_2 = l_3l_4 = 0 \rbrace$.  The map $m_6$ factors as
$$
\begin{CD}
{\vert \mathcal I_Z^2(6) \vert} @>{\tilde m_6}>> {\mathcal W^2_{6,6}} @>f>>~{\mathcal M_6} \\
\end{CD}
$$
It is well known that $deg \ f = 5$, see for instance \cite{ACGH} p.209. The branch divisor of $f$ is just the Petri divisor in $\mathcal M_6$, parametrizing curves $C$ endowed with an $L \in W^1_4(C)$
such that $L^{\otimes 2}$ is special. \par It turns out that the symmetric group $S_5$ acts on $\vert \mathcal I_Z^2(6) \vert$ as the group of Cremona transformations generated by linear transformations fixing $Z$ and quadratic transformations centered at three of the four points of $Z$.  
Then it follows that the degree of $m_6$ is 120, cfr. \cite{SB} Corollary 5. \par The construction is related to Shepherd-Barron's  proof of the rationality of $\mathcal M_6$, \cite{SB}. It will be revisited later.   \medskip \par 
\underline {\it Genus five} \ \  \rm  Consider the analogous factorization
$$
\begin{CD}
{\vert \mathcal I_Z^2(6) \vert} @>{\tilde m_5}>> {\mathcal W^2_{6,5}} @>f>>~{\mathcal M_5} \\
\end{CD}
$$
In this case $\tilde m_5$ is generically injective, because the only linear automorphism fixing $Z$ is the identity. Let $S \subset \mathbf P^4$ be the anticanonical embedding of the blowing of $\mathbf P^2$ along $Z$. As is well known the strict transform of an element $\Gamma$ of $  \vert \mathcal I^2_Z(6) \vert$ is the canonical model $C$ of $\Gamma$. $S$ is the smooth base locus of a pencil $P$  of quadrics. Moreover $C$ is the base locus of a net $N$ of quadrics, in particular $P$ is a line in $N$. It is easy to conclude that $deg \ f \circ \tilde m_5$ is the number of pencils $P' \subset N$ projectively equivalent to $P$. Notice also that $P$ and $P'$ are projectively equivalent  if and only if the same is true for the sets of 5 points in
$\mathbf P^1$: $P \cap \Delta$ and $P' \cap \Delta$, where $\Delta$ is the discriminant quintic curve of the net $N$.  \medskip \par 
We will return to the following observation: \it we have seen that a general curve of  genus $g \leq 6$ appears in a fixed linear system  on a fixed surface $S$. \rm
\end{example}
\begin{example} $7 \leq g \leq 9$: \  $ \vert \mathcal I^2_Z(d) \vert$ cannot be fixed and $ \dim \vert \mathcal I_Z(d) \vert > 0$.
   \ \rm \medskip \par  
\underline {\it Genus seven} We have $d = 7$ for the minimal degree. A general  $\Gamma \in \vert \mathcal I^2_Z(7) \vert$ is  a nodal  septic with  8 nodes. Still the normalization $C$ of $\Gamma$ embeds in a Del Pezzo surface
$S$, which is defined by the blow up $\sigma: S \to \mathbf P^2$ of $Z$. The strict transform of $\Gamma$ is $C$. Note that $C^2 = 17$ and $h^0(\mathcal O_C(C)) = 11$. Hence we obtain $ \dim \ \vert C \vert = 11 =  \dim \vert \mathcal I^2_Z(7) \vert$ from the standard exact sequence
$$
0 \to \mathcal O_S \to \mathcal O_S(C) \to \mathcal O_C(C) \to 0.
$$ 
$\vert C \vert$ cannot dominate $\mathcal M_7$.  On the other hand $\mathcal N_{7,7}$ is birational to a
$\mathbf P^{11}$-bundle $\mathbb P$ over $U = h_{7,7}(\mathcal N_{7,7})$ of fibre $ \vert \mathcal I_Z^2(7) \vert$ at $Z$. Notice that
$$
\mathcal N_{7,7} / PGL(3) \cong \mathbb P / PGL(3) \cong \mathcal W^2_{7,7}.
$$
Let $\mathcal P_8 \cong U / PGL(3)$ be the moduli of 8 general points in $\mathbf P^2$. With more effort one can show that $\mathbb P$ descends to a $\mathbf P^{11}$- bundle on $\mathcal P_8$. So it follows
\begin{proposition}~$\mathcal W^2_{7,7}   \cong \mathcal P_8 \times \mathbf P^{11}$. \end{proposition}
\par  
 For $g =  8, 9$ the situation is similar, we omit further  details.
\end{example}
\begin{example} $g = 10$: $\vert \mathcal I^2_Z(d) \vert$ is $0$-dimensional. \ \rm \medskip \par  
In genus 10 the minimal degree is $d = 9$ and $\delta = 18$. Therefore it follows that $ \ \dim$  $\vert \mathcal I^2_Z(9) \vert = 0$. Hence $\mathcal N_{9,18} \cong Hilb_{18}(\mathbf P^2)$ and  $\mathcal W^2_{9,10}$ is unirational. 
For $d = 10$ the uniruledness of $\mathcal W^2_{d,10}$ can be proved. The proof relies on some arguments to be used frequently, so we outline it. Let $(C,L)$ be a pair defining a general $x \in \mathcal W^2_{10,10}$.  Then the Petri map $ \mu: H^0(L) \otimes H^0(\omega_C(-L) ) \to H^0(\omega_C)$ is injective and induces an embedding
$$ C \subset \mathbf P^2 \times \mathbf P^1$$ such that $\mathcal O_C(1,1) \cong L \otimes \omega_C(-L) \cong \omega_C$. It turns out that $C$ lies in a smooth surface $S \in \vert \mathcal O_{\mathbf P^2 \times \mathbf P^1}(4,3) \vert$:
we omit for brevity the not difficult proof of this fact. Note that $S$ is regular and that $\omega_S \cong \mathcal O_S(1,1)$. Hence we have $\mathcal O_C(C) \cong \mathcal O_C$ by adjunction formula. Furthermore we have $\dim \vert C \vert = 1$, this follows from the regularity of $S$ and the standard exact sequence
$$
0 \to \mathcal O_S \to \mathcal O_S(C) \to \mathcal O_S(C) \to 0.
$$
Let $m: \vert C \vert \to \mathcal W^2_{10,10}$ be the moduli map. If $m$ is constant, a general $D \in \vert C \vert$ is a copy of $C$. Hence there exists a  dominant map
$C \times \vert C \vert \to S$ whose restriction to $C \times \lbrace D \rbrace$ is  the identity map $C \to D$: a contradiction because $S$ is not ruled. Thus  $m(\vert C \vert)$ is a rational curve through $x$. This implies that:
\begin{proposition} $\mathcal W^2_{10,10}$ is uniruled. \end{proposition} 
We point out that the Zariski closure $\mathcal N_{9,18}$ is not a \it scroll \rm in $\vert \mathcal O_{\mathbf P^2}(9) \vert$. In other words $\mathcal N_{10,18}$ is ruled but it is not ruled by a family of linear subspaces of $\vert \mathcal O_{\mathbf P^2}(9) \vert$ of  dimension $>$ $0$. The same appears to be true for $\mathcal N_{10,26}$. This kind of situation is further analyzed in the next section, for instance in proposition 2.11.\end{example}
 \subsection{When a scroll in $\vert \mathcal O_{\mathbf P^2}(d) \vert$ dominates $\mathcal M_g$?}
The latter section highlights the fact that the unirationality of $\mathcal M_g$, $g \leq 10$,  is strongly related to the world of rational surfaces. Then it is natural to ask what one can say on the embeddings
$$ C \subset S $$ of a curve $C$ with general moduli in a rational surface $S$ and about the linear systems  $\vert \mathcal O_S(C) \vert$.  On the other hand an intrinsic limit of the results we have described is due to the use of families of \it nodal \rm plane curves, instead of more general families.  In the next proposition an effect of this limit is observed.   Let $\Gamma$ be any integral plane curve of geometric genus $g$ and  let $n: C \to \Gamma$ be  its normalization.  
\begin{definition} $\Gamma$ is linearly rigid if $ \dim \vert i_*C \vert = 0$, for every factorization $n = f \circ i$, such that $i: C \to S$ is an embedding in a smooth, integral surface and $f: S \to \mathbf P^2$ is
a birational morphism.
\end{definition}
 The next proposition makes quite clear that, to go further with the previous methods,  one has to use families of singular plane curves $\Gamma$ such that the orbit of $\Gamma$ by the Cremona group of $\mathbf P^2$
does not contain a nodal curve.
\begin{proposition} Let $\Gamma$ be general of genus $g \geq 13$. Assume that $\Gamma$ is the strict transform of a nodal curve by a Cremona transformation.
Then $\Gamma$ is linearly rigid.
\end{proposition}
\begin{proof} Let $n = f \circ i$ as above. Let $\sigma: \mathbf P^2 \to \mathbf P^2$ be a birational map such that $\Gamma$ is 
the strict transform of a nodal curve $\Gamma'$ of degree $d'$. Solving the indeterminacy of $\sigma \circ f$ we have $\sigma \circ f  \circ \phi = \psi$, where $\phi: S' \to S$ and $\psi: S' \to \mathbf P^2$ are birational morphisms and
$S'$ is smooth. 
The strict transforms of $\Gamma'$ by $\psi$ and of $\Gamma$ by $f  \circ \phi $ are the same curve $C' $,  biregular to $i_*C$. Note that $$ \phi_* \vert C' \vert = \vert i_*C \vert. $$ Hence it suffices to show that
$ \dim  \vert C' \vert = 0$. Let $Z' =  \Sing \ \Gamma'$. Since $\psi_*: \vert C' \vert \to \vert \mathcal I_{Z'}^2(d') \vert$ is injective, we show that $ \dim \vert \mathcal I_{Z'}^2(d') \vert = 0$.
Note that  $\Gamma'$ is a general point of a Severi variety $\mathcal N_{g, \delta'}$ dominating $\mathcal M_g$. This implies as usual that $\binom{d'-1} 2 - \delta' = g \leq \frac 32 d' - 3$. On the other hand the condition
$ \ \dim \vert C' \vert \geq 1$ implies that ${C'}^2 = {d'}^2 - 4\delta' \geq 0$. Then the two inequalities imply ${d'}^2 -9d' - 10 \leq 0$, that is, $0 \leq d' \leq 10$. Since $C'$ is general, it follows $g \leq 12$. This contradicts our
assumption.
\end{proof} 
  The use of families of \it not nodal \rm singular  curves is the starting point of a second step in the history we are discussing.   This is a first negative step with respect to
the conjectured unirationality of $\mathcal M_g$.  It is  due to Beniamino Segre.   \medskip \par 
In 1928 Segre presented a communication to the International Congress of Mathematicians held in Bologna. This is  on linear systems of singular plane curves with general moduli.
 It summarizes `Sui moduli delle curve algebriche', \cite{Se1},  a paper answering the next questions $\rm Q1$, $\rm Q2$, $\rm Q3$.  
\medskip \par 
Consider the Grassmannian $G_{n,d}$ of $n$-spaces in $\vert \mathcal O_{\mathbf P^2}(d) \vert$, where $n = 0$ is not excluded. If $t \in G_{n,d}$ we will denote by $\mathbb P_t$ the corresponding linear system 
of plane curves. Let $T \subset G_{n,d}$ be an integral variety such  that
$$
\mathbb P = \bigcup_{t \in T} \mathbb P_t  
$$
is a family of integral plane curves $\Gamma$ of geometric genus $g$. Then $\mathbb P$ is endowed with the usual diagram
$$
\begin{CD}
{\mathbb P} @>{m_T}>> {\mathcal M_g} \\
@V{h_T}VV @. \\
{Hilb_{\delta}(\mathbf P^2)} @. \\
\end{CD}
$$ 
By definition $h_T(\Gamma) =  \Sing \ \Gamma$. $m_T$
is the natural map in $\mathcal M_g$.  The questions considered by Segre are the following, cfr. \cite{Se1} sections 1-4 and 11: \rm 
 \medskip
\begin{itemize} \it
\item[$\circ$] $\rm [Q1]$  Does there exist $\mathbb P$ such that $m_T$ is dominant and $T$ is a point?
\item[$\circ$] $\rm [Q2]$ Does there exist $\mathbb P$ such that both $m_T$ and $h_T$ are dominant?
\item[$\circ$] $\rm [Q3]$ Does there exist $\mathbb P$ such that $m_T$ is dominant and $ \dim \ \mathbb P_t > 0$?
 \end{itemize} 
\begin{remark} \rm It is clear that: \par (1) \it a positive answer to $Q_1$ or $Q_2$ implies that $\mathcal M_g$ is unirational,\par  \rm (2) \it
 a positive answer to $Q_3$ implies that $\mathcal M_g$ is uniruled. \rm 
\end{remark} \par 
Let $\mathbb H \subset \vert  \mathcal O_{\mathbf P^r}(d) \vert$ be a linear system of hypersurfaces of degree $d$ and let $Z$ be its base scheme.
$\mathbb H$ is said to be complete if  $ \mathbb H = \vert \mathcal I_Z(d) \vert$. We recall that:
\begin{definition} A complete linear  system $\mathbb H$ is regular if  the restriction map
$
r: H^0(\mathcal O_{\mathbf P^r}(d)) \to H^0(\mathcal O_Z(d))
$
has maximal rank.  
\end{definition} 
 Assume $h^0(\mathcal I_Z(d)) > 0$, then $\vert \mathcal I_Z(d) \vert$ is regular  if and only if $h^1(\mathcal I_Z(d)) = 0$.   In what follows it will be not restrictive to assume that $\mathbb P_t$ is complete. 
\medskip \par  Let $Z_t$ be the base locus of $\mathbb P_t$. Up to shrinking $T$ we can assume that the family 
$\lbrace Z_t, t \in T \rbrace$ is a flat family of 0-dimensional schemes and that
$ \dim \ \mathbb P_t$ is constant. Let $o \in T$ be a general point and let $\Gamma_o$ be general in $\mathbb P_o$. It follows from  Noether's theorem on the 
reduction of a plane curve to a plane curve with ordinary singularities, that there exist birational morphisms
$$
\sigma_o: S_o \to \mathbf P^2 \ , \ \psi_o: S_o \to \mathbf P^2
$$
so that: (i) $S_o$ is smooth, (ii) the strict transform $C_o$ of $\Gamma_o$ by $\sigma_o$ is smooth, (iii) $\psi_o/C_o$ is generically injective and 
$\Gamma'_o := \psi_o(C_o)$ has ordinary singular points. It is standard to show that, up to a finite base change $\pi: T' \to T$, the fourtuple $(C_o, \Gamma_o, \sigma_o, \psi_o)$
moves in an irreducible family $(C_t,  \Gamma_t, \sigma_t, \psi_t),  t \in T'$, such that
$
\sigma_t: S_t \to \mathbf P^2 \ , \ \psi_t: S_t \to \mathbf P^2,  
$
are birational morphisms, $C_t$ is the strict transform by $\sigma_t$ of a general $\Gamma_t \in \mathbb P_t$ and conditions (i), (ii), (iii) considered above for $o$ 
are satisfied. Note that $\vert C_t \vert$ is the strict transform of $\vert \mathcal I_{Z_t}(d) \vert$ by $\sigma_t$. Moreover the elements of the linear system
 $$
 \mathbb P'_t := \sigma_{t*} \vert C_t \vert \subset \vert \mathcal O_{\mathbf P^2}(d') \vert
$$
are curves of degree $d'$ and  genus $g$ with ordinary singularities.  We have $ \dim \ \mathbb P_t =  \dim \vert C_o \vert =  \ \dim \ \mathbb P'_t$. Finally, notice also that:
\begin{lemma}  The following conditions are equivalent: \begin{enumerate} \item $\mathbb P_t$ is regular, \item $h^1(\mathcal O_{S_o}(C_o)) = 0$ \item $\mathbb P'_t$ is regular. \end{enumerate} \end{lemma} \par 
Due to the previous remarks, we make from now on the following \medskip \par 
\underline {\it Assumption} \it A general element of $\mathbb P$ has ordinary singularities. \rm
\medskip  \par 
This assumption is clearly not restrictive in order to fully answer questions Q1 and Q3. Instead, to fully answer question Q2 under such an assumption, one has
to rely on  some well known conjectures on the regularity of linear systems of plane curves. As we will see in a moment, these conjectures go back to Beniamino Segre.
\medskip \par 
In his paper Segre exhibits the following answers to  the previous questions:   \par 
\begin{itemize} \it
\item[$\circ$] $\rm [Q1]$ No $\mathbb P$ exists for $g \geq 7$,
\end{itemize}
\begin{itemize} \it
\item[$\circ$]  $\rm [Q2]$ No $\mathbb P$ exists for $g \geq 37$ or for $g \geq 11$ and $\mathbb P_t$ regular,
\end{itemize}
\begin{itemize}
\item[$\circ$] $\rm [Q3]$ \it No $\mathbb P$  exists for $g \geq 37$ or for $g \geq 11$ and $\mathbb P_t$ regular.
\end{itemize}  \medskip \par  
Segre points out that the negative answer to $\rm Q2$ and $\rm Q3$ could be extended to $g \geq 11$ without further assumptions.   This is possible, he says, if  one relies on an unproved  claim of intuitive evidence, \cite{Se1} 
6 p. 79. It can be stated as follows: 
\medskip \par 
\underline {\it Claim} \it \  Let $p_1 \dots p_s$ be \it general points  in $\mathbf P^2$ and $\nu_1 \dots \nu_s$  positive integers. Consider the ideal sheaf $\mathcal I_Z$ of  $ \ Z = \bigcup_{i = 1 \dots s} Z_i$, where   
$Z_i$ is  $Spec \ \mathcal O_{\mathbf P^2, p_i} / m_i^{\nu_i} $. Then the linear system $\vert \mathcal I_Z(d) \vert$ is regular.   \rm
\medskip \par  
This  is probably the remote origin of a well known conjecture. Many years later, in 1961, Segre appropriately reformulates this claim:
\medskip \par 
\underline{\it Conjecture} (B. Segre \cite{Se2}).Ê \it Let $p_1, \dots, p_s$ and $Z$ be as above. Assume that an element $\Gamma \in \vert \mathcal I_Z(d) \vert$ is
a reduced curve, then $\vert \mathcal I_Z(d) \vert$ is regular. \rm
\medskip \par 
See \cite{Ci1}  for an account on the influence of this conjecture and the so many related conjectures and theorems. We refer in particular to Harbourne-Hirschowitz conjecture, cfr. \cite{Ci1} 4.8, and to the following
conjecture we have already seen to work when $\nu_1 = \dots = \nu_s = 2$.
 \medskip \par 
\underline{\it Conjecture} \it  (A. Hirschowitz) Let $Z$ be defined as in the previous claim. If $\vert \mathcal I_Z(d) \vert$ is regular then a general $\Gamma \in \vert \mathcal I_Z(d) \vert$ has ordinary singularities, unless
the unique element of $\vert \mathcal I_Z(d) \vert$ is a cubic of multiplicity $m$ and $m = \nu_1 = \dots = \nu_s$. \rm
\rm \medskip \par 
Our subject is influenced by this frame, in particular the next question is of special interest. Let $\Gamma$ be any plane curve of genus $g$ with general moduli: \begin{itemize} \it
\item[$\circ$] $\rm [Q4]$ Is $\Gamma$ linearly rigid as soon as $g \geq 11$?
\end{itemize} \par 
Segre's conjecture and his answer to Q3 imply that this is true at least if the singular points of $\Gamma$ are in general position. So we want to describe more in detail  the arguments of  Segre
for answering $\rm Q3$ (and Q2). \medskip \par 
\underline {\it The answer to $\rm Q3$} \par We start with a family as above $\mathbb P = \bigcup \mathbb P_t, \ t \in T$, of curves of degree $d$. As already remarked it is not restrictive, for our purpose, to assume that a general
curve $\Gamma \in \mathbb P$ has ordinary singularities. Also, it will be not restrictive to assume that $\mathbb P$ is equal to its orbit by $PGL(3)$. We denote as
$$
\overline \nu := \nu_1 >  \dots > \nu_s
$$
the decreasing sequence of the multiplicities of the $\delta$ singular points of $\Gamma$. Let $Z_i \subset  \Sing \ \Gamma$ be the set of points of fixed multiplicity $\nu_i, \ i = 1 \dots s$. The cardinality of $Z_i$ is  constant for a general $\Gamma$,  it will be denoted as $\delta_i$.
 Finally,  the product of the Hilbert schemes of  $0$-dimensional subschemes of $\mathbf P^2$ of length $\nu_i$ contains a non empty open subset parametrizing elements $(Z_1, \dots, Z_s)$ such that $Z_i$ is supported on $\nu_i$
distinct points and $Z_i \cap Z_j \  i \neq j$, $i,j = 1 \dots s$. This open set will be denoted as
$$
Hilb^o_{\overline \nu}(\mathbf P^2).
$$ 
Let $(Z_1, \dots, Z_s) \in Hilb^o_{\overline \nu}(\mathbf P^2)$. Consider $Z = \bigcup_{i = 1 \dots s} Z_i$ and the blowing up
$ \sigma: S \to \mathbf P^2$ of $Z$. Clearly $(Z_1, \dots, Z_s)$ defines a triple $$ (S, \mathcal O_S(H), \mathcal O_S(C))$$ 
where $S$ is a smooth rational surface and $\mathcal O_S(H)$, $\mathcal O_S(C)$ are such that 
\medskip \par  \it
$\circ$ $\vert H \vert$ defines the blowing up $ \sigma: S \to \mathbf P^2$ of a set $Z$ of $\delta$ distinct points.
\medskip \par   
$\circ$ $C \sim dH - \sum_{i = 1 \dots s} \nu_i E_i$,  where  $E_i = \sigma^*Z_i$ and $Z = \bigcup Z_i$ is a partition.   
\rm \medskip \par 
Conversely a triple $(S, \mathcal O_S(H), \mathcal O_S(C))$ immediately defines a point of $Hilb^o_{\overline \nu}(\mathbf P^2)$ up to the action of $PGL(3)$. Consider the GIT-quotient
$$
\mathcal P_{\overline \nu} := Hilb^o_{\overline \nu}(\mathbf P^2) / PGL(3).
$$
Then $\mathcal P_{\overline \nu}$ is birational to the moduli space of the triples $(S, \mathcal O_S(H), \mathcal O_S(C))$ as above. As expected, this moduli space has clearly dimension
$$
2s- 8 = 10 \chi (\mathcal O_S) - 2K^2_S.
$$
Actually, $10 \chi(\mathcal O_S) - 2K_S^2$ is a lower bound for every irreducible component of the moduli space of surfaces with fixed $\chi(\mathcal O_S)$ and $K_S^2$, see \cite{Ca1}. Consider a general $\Gamma \in \mathbb P$ and  the standard diagram
$$
\begin{CD}
{\mathbb P} @>{m_T}>> {\mathcal M_g} \\
@V{h_T}VV @. \\
{Hilb^o_{\overline \nu}(\mathbf P^2)} @. \\
\end{CD}
$$
Let $\sigma: S \to \mathbf P^2$ be the blowing up of $ \Sing \ \Gamma$. Then the strict transform $C$ of $\Gamma$ is smooth and $\vert C \vert$ is the strict transform of $\mathbb P_t$. In particular, from the standard exact sequence
$$
0 \to \mathcal O_S \to \mathcal O_S(C) \to \mathcal O_C(C) \to 0,
$$
we have
$$
 \ \dim \ \mathbb P_t =  \dim \ \vert C \vert = h^0(\mathcal O_C(C)).
$$
Notice also that $$ h^1(\mathcal O_C(C)) = h^1(\mathcal O_S(C)). $$  Hence  \it $\mathbb P_t$
is not regular  if and only if $\mathcal O_C(C))$ is special. \rm Furthermore we have
$$
 \dim \ \mathcal M_g \ \leq  \dim \ \mathbb P -  \dim \ PGL(3) \leq   \dim \ \mathcal P_{\overline \nu} + h^0(\mathcal O_C(C)).
$$
The latter is a key inequality. Furthermore one has
\begin{lemma}  \ \begin{enumerate} \it
\item $ h^0(\mathcal O_C(C)) \geq 1$ so that $C^2 \geq 0$,
\item $ 3g + 5 \leq 2\delta + h^0(\mathcal O_C(C))$,
\item $ h^1(\mathcal O_C(C)) \leq g$.
\end{enumerate}
\end{lemma}
\begin{proof} (1) follows from the assumption $ \ \dim \ \mathbb P_t > 0$. (2) is equivalent to the latter inequality. (3) follows immediately
from $h^0(\mathcal O_C(C)) \geq 1$. \end{proof}  \par 
Building on the previous inequalities, the hard part of Segre's work is developed essentially in  the abstract  lattice
$$
\mathbb L :=  \Pic \ S =  \oplus_{ij} \mathbb Z[E_{ij}] \oplus \mathbb Z [H], \ \ 1 \leq i \leq s \ \ , \ \ 1 \leq j \leq \delta_i.
$$
Here, keeping the previous notations, $\mathcal O_S(H) \cong \sigma^* \mathcal O_{\mathbf P^2}(1)$ and the $E_{ij}$'s  are the irreducible exceptional divisors of $\sigma$. \par  
The main point is to consider the group of isometries $G \subset O(\mathbb L)$ which is generated by reflections $q: \mathbb L \to \mathbb L$ induced by  quadratic transformations
centered at three distinct points of $ \Sing \ \Gamma$. \par We discuss the case where $\mathbb P_t$ is regular. By Brill-Noether theory and the lemma, a positive 
answer to question Q3 for $\mathbb P$  implies that:
\begin{enumerate} 
\item $ g \leq \frac 32(d - 2)$,
\item $h^0(\mathcal O_C(C)) \geq 1$
\item $3g + 5 \leq 2\delta + h^0(\mathcal O_C(C))$.
\end{enumerate}  
Since $\mathbb P_t$ is regular, we have that $h^1(\mathcal O_C(C)) = h^1(\mathcal O_S(C)) = 0$. Hence (3) is equivalent to $4g + 4 \leq 2\delta + C^2$. 
 Consider the set of linear forms 
$$
\lbrace h_p: \mathbb L \to \mathbb Z, \ p \in G \rbrace
$$
which are defined as follows: $\forall D \in \mathbb L, h_p(D) = \langle p(D),H \rangle$, where $\langle \ , \rangle$ is the intersection product. Segre shows that:
\begin{theorem}   Assume $h^1(\mathcal O_S(C)) = 0$ and that (1), (2), (3) hold true. If  $11 \leq \frac 32(d - 2)$ then there exists $p \in G$ such that
$h_p([C]) < d$.
 \end{theorem} \par 
 Now assume $11 \leq g$. Since $g \leq \frac 32(d - 2)$, the theorem implies that there exists a birational map $f: \mathbf P^2 \to \mathbf P^2$ such that
 $f(C)$ has degree $d' < d$. Globalizing the construction of $f$, one can show that there exists a family of curves $\mathbb P'$ satisfying all the
 conditions assumed for $\mathbb P$, but of degree $d' < d$. The iteration of this argument yelds a contradiction. Then it follows:
\begin{corollary} Assume that $C$ has general moduli, that $\mathcal O_C(C)$ is non special and that  $ \dim \ \vert \mathcal O_S(C) \vert \geq 1$. Then $g < 11$.
\end{corollary} \par 
If $\mathcal O_C(C)$ is special and effective we have $2h^0\mathcal O_C(C)) \leq C^2$ by Clifford's theorem. Then the same argument yelds:
\begin{theorem}   Assume that (1), (2), (3) hold true. If  $37 \leq \frac 32(d - 2)$ then there exists $p \in G$ such that
$p([C]) < d$.
\end{theorem} 
\begin{corollary} Assume $C$ has general moduli and that  $ \dim \ \vert \mathcal O_S(C) \vert \geq 1$. Then $g \leq 37$.
\end{corollary} \par 
Let us derive some weak conclusions concerning question Q4: \it is any plane curve $\Gamma$ having general moduli and genus $g \geq 11$  linearly rigid? \rm  \par
Segre's theorem suggests  that such a $\Gamma$ should be linearly rigid. Equivalently, it seems very possible that no smooth curve $C$, having general 
moduli and genus $g \geq 11$, embeds in a rational surface $S$ so that $ \dim \ \vert \mathcal O_S(C) \vert \geq 1$ and the map $m: \vert \mathcal O_S(C) \vert \to \mathcal M_g$
is not constant. \par
Segre's conjecture implies such a property if $\Gamma$ has ordinary  singular points in general position. Indeed the latter assumptions imply the regularity of the linear system
$\mathbb P_t$ of all curves having at least the same multiplicity of $\Gamma$ at each point $x \in  \Sing \ \Gamma$. Then, by corollary 2.15, it follows $ \dim \ \mathbb P_t = 0$ and this
implies that $\Gamma$ is linearly rigid. \par The non existence, for any $g \geq 11$, of families of non linearly rigid curves $\Gamma$ with general moduli  remains however unproved.  
Perhaps unexpectedly, Segre discovered that, as soon as $g$ grows,  $\mathcal M_g$ is far from being covered by rational curves $m(P)$, where $P$ is a pencil on a rational
surface and $m: P \to \mathcal M_g$ is the moduli map.   
\subsection{Curves with general moduli and algebraic surfaces}
Due to year 1938 racial laws and to  racial prosecution, Beniamino Segre left Italy for Britain. After Second World War he came back and was professor at the University La Sapienza in Rome, on the chair of Higher Geometry. He retired in the year 1973. \par   Young students 
following his lectures were exposed to the main problems in Algebraic Geometry.  This  was certainly one of the effective ways in which  the classical circle of ideas could be
passed to  young generations. \par  As an example of this passage, it is natural to mention some work which is directly related to problems Q1, Q2, Q3. 
These questions were indeed reconsidered in 1975 by E. Arbarello in \cite{Ar}. In particular this paper is  a connecting  point  between the past of our story and  the forthcoming part, where not only rational surfaces are in use.  \par
Question Q1 is considered in \cite{Ar} for any surface.  The result obtained relies on the case of rational surfaces, due to Castelnuovo and completed by Segre, see \cite{C1} and \cite{Se1}. Joining together all these results we have:
 \begin{theorem} Let $P $ be a linear system of curves  of genus $g$ on  a smooth, connected surface $S$.  If $ P$ dominates $\mathcal M_g$ then $g \leq 6$ and $S$ is rational.
\end{theorem}   
\begin{proof} Up to replacing $S$ by an appropriate birational model,  we can assume that a general $C \in P$ is smooth and that $P$ is base point free. We can also assume $g \geq 3$. Since $\vert C \vert$ dominates $\mathcal M_g$, we have $ \ \dim \ \vert C \vert \geq \  \dim \ P \geq 3g - 3$.  Hence $\mathcal O_C(C)$ is non special of degree $C^2 \geq 4g - 3$ and adjunction formula yelds $CK_S \leq -2g + 1$. Since $\vert C \vert$ is base point free, it follows that $\vert mK_S \vert$ is empty for $m \geq 1$. Hence $S$ is ruled and birational to $R \times \mathbf P^1$, in particular the projection $p: R \times \mathbf P^1 \to R$ induces a finite map $p_C: C \to R$. Since the curves of  $\vert C \vert$ have general moduli, this is impossible unless $R$ is rational.  Hence $S$ is rational. Now let $g \geq 10$, then we have $ \dim \vert C \vert \geq 3g - 3 \geq 2g + 7$. Moreover, a well known theorem of Castelnuovo on linear systems of curves on a rational surface, \cite{C1} 1.3, implies that then the elements of $\vert C \vert$ are hyperelliptic. This contradicts the generality of $C$ and implies the statement for $g \geq 10$. The cases $g = 7, 8, 9$ are excluded by Segre in \cite{Se1} 4.
\end{proof} \par 
 The bound $g \leq 6$, offered by the previous theorem, is sharp. Indeed we have already seen in section 2 that, when $g \leq 6$, the space $\mathcal M_g$ is dominated by a fixed linear system $P$ of integral plane sextics with $10-g$ double points.\medskip \par 
Later, in 1981, Arbarello and Cornalba reconsidered another remarkable result of Segre in the paper `Footnotes to a Paper by Beniamino Segre',  see \cite{AC1} and \cite{Se3}. As summarized in \cite{AC1} after the title, these papers deal with \it the number of $g^1_k$'s on a general k-gonal curve, and the unirationality of the Hurwitz Spaces of $k$-gonal curves. \rm We reformulate these issues in the vein of our exposition. Let us consider the Hurwitz space 
$$ \mathcal H_{g,k}$$ 
of the finite covers of degree $k$ of $\mathbf P^1$ by curves of genus $g$. The following is a well known result proved by Fulton in \cite{Fu2}:
\begin{theorem} $\mathcal H_{k,g}$ is irreducible. \end{theorem} \par 
 This implies that the corresponding universal Brill-Noether locus $\mathcal W^1_k$ in $ \Pic_{k,g}$ is irreducible too. A first question is:
\begin{itemize} \it
\item[$\circ$] $\rm [Q5]$ Is $\mathcal H_{g,k}$, and hence $\mathcal W^1_k$, unirational or uniruled for some $g$? 
\end{itemize} \par 
Moreover consider the forgetful map
$$
f: \mathcal W^1_k \to \mathcal M_g
$$
and denote its image by $\mathcal M^1_{k,g}$ as usual. A second question is:
\begin{itemize} \it
\item[$\circ$] $\rm [Q6]$ If $f: \mathcal W^1_k \to \mathcal M^1_{k,g}$ is generically finite, what is its degree?
\end{itemize} \par 
Question $\rm Q6$ makes sense if  $\rho(k,g,1) \leq 0$. Otherwise $f$ is not generically finite. If $\rho(k,g,1) = 0$  the answer is offered by the degree of the class
of $W^1_k(C)$ in $ \Pic^k (C)$. Computing it as in \cite{ACGH} 4.4 p. 320, one obtains:
$$
deg \  f = \frac {g!}{(g - k + 1)! (g - k + 2)!}.
$$
If $\rho(k,g,1)$ is negative, namely if $k < [\frac{g+3}2]$, then it is proved in \cite{Se3} that  $f$ is generically finite. More precisely it is shown that the fibre of $f$ is finite at the moduli point $x \in \mathcal M^1_{k,g}$ of a curve $C$ such that $W^1_{k'}(C) = \emptyset$  for $k' < k$. \medskip  \par 
Let us sketch briefly some geometric motivations behind the proof that $f$ is generically finite.  Let $L \in W^1_k(C)$ and $k < [\frac {g+3}2]$. Then the Petri map
$$
\mu: H^0(L) \otimes H^0(\omega_C(-L)) \to H^0(\omega_C)
$$
is not injective. This follows from geometric Riemann-Roch and the count of  dimensions. On the other hand it is well known that $Coker \ \mu$ is isomorphic to the tangent space at $L$ to $W^1_k(C)$,
\cite{ACGH} prop. 4.2. Since $\mathcal W^1_k$ is irreducible, it follows that $f$ is finite if $Coker \ \mu$ is zero dimensional for a general pair $(C,L)$.  Hence it suffices to produce a pair $(C,L)$ such
that $\mu$ is surjective. This is equivalent to say that $ \dim \ \mbox{Ker} \  \ \mu = g - 2k + 2$. Counting dimensions, it is not restrictive to assume that $\vert L \vert$ is a base point free pencil. Then the base point free pencil trick, \cite{ACGH} p. 126, implies that $ \dim \ \mbox{Ker} \  \ \mu = h^0(\omega_C(-2L)) = h^1(L^{\otimes 2})$. \par Now consider the plane blown up in a point i.e. the Hirzebruch surface $\mathbb F_1$. Let $\vert F \vert$ be its ruling and $E$ its exceptional line. Segre shows that the condition $h^1(L^{\otimes 2}) = g - 2k + 2$ is satisfied if $C$ is the normalization of a general nodal integral curve of geometric genus $g$
$$
\Gamma \in \vert kE + mF \vert,
$$
where $m := [\frac{g+k+3}2]$ \cite{Se3} 4 p. 542. In other words $C$ is birational to a plane curve $\Gamma'$ of degree $d = k + m$ such that $ \Sing \ \Gamma'$ consists of finitely many nodes and a unique other ordinary singularity of multiplicity $m$. This result of Segre can be strengthened. It is proved in \cite{AC1} that: \begin{theorem} Assume $k < [\frac{g+3}2]$ then $f$ has degree one. Moreover $W^1_k(C)$ consists of a single element if $L \in W^1_k(C)$ is globally generated.
\end{theorem}  \par 
 In the proof a new modern tool is essential to go beyond the results of classical geometry of the construction. Deformation theory, in particular first order deformations of
singular curves, is indeed used systematically to achieve the main steps.  The answer \cite{AC1} provides to question $\rm Q5$ relies on the same methods. It is somehow a kind of surprise:
 \begin{theorem} Let $k \leq 5$ then $\mathcal W^1_k$ is unirational. \end{theorem} \par 
More precisely this statement follows from theorem 5.3 of \cite{AC1}:
\begin{theorem}  $\mathcal H_{k,g}$ is unirational in the following cases:  
\begin{enumerate} \it
\item $3 \leq k \leq 5$ and $g \geq k-1$,
\item $k = 6$ and $5 \leq g \leq 10$ or $g = 12$,
\item $k = 7$ and $g = 7$.
\end{enumerate} 
\end{theorem}
\par 
Continuing in our vein, we remark that this theorem  can be also viewed as a study of families of nodal curves on a Hirzebruch surface
$
S = \mathbb F_e.
$
We have $$  \Pic \ S \cong \mathbf Z[E] \oplus \mathbf Z[F], $$ where $[E]$ is the class of a minimal section, $E^2 = -e$, $F$ is the fibre of the natural  projection 
$
p: S \to \mathbf P^1.
$
Any curve $C$ in $S$ is endowed with the line bundle $L = \mathcal O_C(F)$. Clearly $L$ is in $W^1_k(C)$, where $k := \cdot$. Keeping $k$ fixed, we have  $m = E \cdot C$ and $ C \sim (m+ke)F + kE$. \par  
To parametrize the Hurwitz space $\mathcal H_{k,g}$, one can study in $S$ the families 
$$
\mathcal N_{g, \delta}(S)
$$
of integral $\delta$-nodal curves of genus $g = p_a(C) - \delta$ and fixed $m = E \cdot C$. \par 
As for Severi varieties of nodal curves in $\mathbf P^2$, we have natural maps
$$
\begin{CD}
{Hilb_{g, \delta}(S)} @<{h_{g, \delta}}<< {\mathcal N_{g, \delta}(S)} @>{p_{g,\delta}}>> {\mathcal H_{k,g}}.\\
\end{CD}
$$ 
By definition $h_{g, \delta}(\Gamma) =  \Sing \ \Gamma$ and $p_{g, \delta}(\Gamma) = p  \circ n$, where $n: C \to \Gamma$ is the normalization map. To prove  the latter theorem we can simply
assume $e = 1$. Then the range of $k$ in its statement is equivalent to the condition that $ \dim \ \vert C \vert - 3\delta \geq 0$ and $p_{g, \delta}$ be dominant. Adequate deformation theoretic arguments are then  needed
to deduce that $h_{g, \delta}$ is dominant and hence to obtain the unirationality results. \medskip \par 
A very general open problem, arising from the previous discussion, is about the structure of a finite cover of degree $k$ in the case of algebraic curves and not only. This could give more informations
about the unirationality  of  $\mathcal W^1_k$ for $k = 6$ and maybe more, see \cite{Ge} to have an update on the present knowledge on the case of degree 6 finite covers. It is shown in this paper that
$\mathcal H_{6,g}$ is unirational for $g \leq 28$ and $g = 30, 31, 35, 36, 40, 45$. \par 
The results in \cite{AC1} imply the unirationality of $\mathcal M_g$, $g \leq 10$. It is of some chronological interest to note that this paper still appeared before the radical change due to the 1982 first paper on 
the Kodaira dimension of $\mathcal M_g$ by Harris and Mumford, \cite{HM}. One can read in \cite{AC1}: \par
\it `We want to stress that the unirationality of $\mathcal M_g$, $g \leq 10$ (and of $\mathcal M^1_{g,3}$ of course) is classically known [10].  It has been recently proved by Sernesi [8], using entirely different
methods, that $\mathcal M_{12}$ is also unirational. The problem of deciding whether $\mathcal M_g$ is unirational is, at the moment, open for $g = 11$ or $g > 12$.' \rm
\medskip \par 
It is the right moment to recall the fundamental results, due to Eisenbud, Harris and Mumford, on the Kodaira dimension of $\mathcal M_g$, \cite{HM}, \cite{H1} and \cite{EH}:
\begin{theorem} $\mathcal M_g$ is of general type for $g \geq 24$. \end{theorem} \par 
We have already described some typical, modern evolutions of the problem posed by Severi on the unirationality $\mathcal M_g$. We  continue describing further evolutions and perspectives, on the side of
uniruledness/unirationality of $\mathcal M_g$ in low genus. Reading the previous pages, it is quite clear that we are stressing three typical aspects of the modern evolution of our subject: \medskip \par 
(1) a first aspect is the study of deformation theory for $(C,S)$, where $C$ is a genus $g$ curve embedded in a smooth surface $S$. As an application, one can try to understand when a concrete family of
pairs $(C,S)$ dominates $\mathcal M_g$.  \medskip \par 
(2) A second aspect is the study of the families of singular curves $C$ of genus $g$ in a given smooth surface $S$. In particular one would like to describe the family of all $\delta$-nodal curves of genus $g$ having fixed
homology class.
\medskip \par 
(3) A third aspect is the study of families of pairs $(C, S)$ such that $C$ is a smooth curve in $S$, it has general moduli and $ \dim \ \vert C \vert > 0$. This  is related to the study of uniruledness and unirationality properties of $\mathcal M_g$.  
\medskip \par 
A common feature to (1), (2), (3) is that there is no restriction on the type of surface $S$ to be considered. The study of problems (1) and (2) is a central theme since many years and it is due to many authors, see e.g.  \cite{Ser2} for the related deformation theory. Also the recent study on Severi varieties of nodal curves on any surface has several sources, among them \cite{DH}, \cite{Tan1}, \cite{Tan2}, \cite{CS}, \cite{CC}, \cite{Fu1}.
We restrict now to problem (3). \par In 2007 E. Sernesi has given a partial answer to such a question. Even if these results appear later in the chronology, it  is now the moment to describe them.
 \medskip \par 
 The starting point can be a rational curve $R \subset \overline {\mathcal M}_g$.  Taking its normalization we have a morphism
 $$
 m: \mathbf P^1 \to \overline {\mathcal M}_g.
 $$
 The general idea is to study, in some sense, the normal bundle to the map $m$. This can give useful informations on the following question:
 \begin{itemize} \it
 \item[$\circ$]~When does $R$ move in a family of rational curves covering $\overline {\mathcal M}_g$?
 \end{itemize} \par 
 Of course the main purpose is to understand by a direct analysis of the deformations of $R$, whether $\overline {\mathcal M}_g$ is ruled or not. One can turn to effective applications of these ideas as follows. \par A \it rational fibration of genus $g$ \rm is just a relatively minimal morphism
 $$
 f: S \to \mathbf P^1
 $$
 such that $S$ is a smooth surface and each fibre is a stable curve of genus $g$.  Since its target space is $\mathbf P^1$, such a fibration is called a rational fibration.\par We will assume that $f$ is not isotrivial i.e.  the associated morphism $$ m_f: \mathbf P^1 \to \overline {\mathcal M}_g $$ is not constant.  In particular this implies, by Arakelov theory, that $h^0(T_S) = 0$ and $h^0(T_{S/\mathbf P^1}) = h^1(T_{S/\mathbf P^1}) = 0$, \cite{Ser3} 1.4. \par Starting from the functor
 of sheaves$f_* Hom$, one can consider its first derived functor. Denoting it by $Ext^1_f$, one can then consider the sequence associated to the local-global spectral sequence for $Ext_f$. This is in fact the exact sequence of sheaves
 $$
 0 \to R^1f_*T_{S/\mathbf P^1} \to Ext^1_f(\Omega^1_{S/\mathbf P^1}, \mathcal O_S) \to f_* Ext^1_S(\Omega_{S/\mathbf P^1}, \mathcal O_S) \to 0.
 $$
The sheaf in the middle is a vector bundle on $\mathbf P^1$. Since $f$ has fibres of genus $g$ its rank is $3g-3$, hence we have
$$
Ext^1_f(\Omega^1_{S/\mathbf P^1}, \mathcal O_S) \cong \oplus_{i = 1 \dots 3g-3} \mathcal O_{\mathbf P^1}(a_i),
$$
cfr. \cite{Ser3} 1.5. Recall that $f$ is said to be  \it free \rm rational fibration if  $a_i \geq 0$ $i = 1 \dots 3g-3$. We can also say that \it $f$ has general moduli \rm if the exists a dominant rational map $m: \mathbf P^1 \times B \to \overline {\mathcal M}_g$ such that $B$ is integral and $m / \mathbf P^1 \times \lbrace o \rbrace = m_f$ for some $o \in B$. The previous vector bundle contains many informations about the deformations of the map $m_f:  \mathbf P^1 \to \overline {\mathcal M}_g$. 
 In particular one has that
 \begin{theorem} If $f: S \to \mathbf P^1$ has general moduli, then $f$ is free. \end{theorem}
 \par 
 Building on this basis, Sernesi tries to understand when no free rational fibration of genus $g$ exists. The goal is to find the maximal value of $g$ such that $\overline {\mathcal M}_g$ is uniruled.    The results obtained put in evidence once more  the beautiful interplay between the geometry of $\mathcal M_g$ and the algebraic surfaces. \par
 Let $C \subset S$ be general of genus $g \geq 3$ and $ \dim \ \vert C \vert \geq 1$. $S$ is a smooth surface of  Kodaira dimension $k(S)$.    
\begin{theorem}[Sernesi] Under the previous assumptions one has. 
\begin{enumerate} \it
\item Let $S$ be of general type, then $ \dim \ \vert C \vert \geq 2$ and $K_S^2 \geq 3\chi(\mathcal O_S) - 10$ implies $g \leq 19$.
\item Let $S$ be an elliptic surface such that $k(S) \geq 0$, then $g \leq 16$.
\item Let $S$ be a surface such that $k(S) \geq 0$, then
\begin{itemize} \it
\item[$\circ$] $g \leq 6$ if $p_g = 0$,
\item[$\circ$] $g \leq 11$ if $p_g = 1$,
\item[$\circ$] $g \leq 16$ if $p_g = 2$ and $h^0(\omega_S(-C)) = 0$.
\end{itemize}
\end{enumerate}
\end{theorem}
 \medskip \par 
  \subsection{Families and rulings of unirational varieties in $\overline {\mathcal M}_g$}
  Going back to the chronological order of  events, we start this section with K3 surfaces endowed with a genus $g$ polarization.  In 1983 Mori and Mukai  use these surfaces in \cite{MM} to prove the uniruledness of $\mathcal M_{11}$. This is a new step in the study of the global geometry of the moduli of curves of low genus. \par Furthermore \cite{MM} is a seminal paper  for the systematic use of K3 surfaces in the study of curves and their moduli. Since then Mukai, and many other authors, started to investigate the deep and beautiful relations between K3 surfaces and curves of low genus. \par  This has important consequences in our subject and we will see a few of them later.
Now, forgetting about rational surfaces, we want to profit of the K3-geometry and of some other surfaces. We want to discuss in low genus: \medskip \par   \it
  (1) the uniruledness of $\mathcal M_g$ and of some loci $\mathcal W^r_{d,g}$, \par
  (2) ruledness results for the same loci and for $\mathcal M_g$. \medskip \par  \rm 
 A \it polarized K3 surface of genus $g$ \rm is a pair $(S, \mathcal O_S(C))$ such that: $S$ is a K3 surface,   $p_a(C) = g$ and $\mathcal O_S(C)$ is very ample and primitive in $ \Pic \ S$.
 \par The uniruledness of $\mathcal M_{11}$ is a consequence of the more general result we are now going to state, see \cite{Mu1} and \cite{MM} as well. Let
$$
\mathcal F_g
$$
be the moduli space of  polarized K3 surfaces of genus $g$. $\mathcal F_g$ is known to be irreducible. Moreover it is endowed with a projective bundle
$$
p_g: P_g \to \mathcal F_g,
$$
having fibre $\vert C \vert$ at the moduli point of $(S, \mathcal O_S(C))$. Consider the map
$$
m_g: P_g \to \mathcal M_g.
$$
Counting dimensions we have $ \dim \ P_g = 19 + g$. Hence $m_g$ is not dominant for $g \geq 12$.  On the other hand one has:
\begin{theorem} $m_g$ is dominant for $g \leq 11$ and $g \neq 10$. \end{theorem} \par 
Of course the uniruledness of $\mathcal M_{g}$ follows for the same values of $g$.  Let us sketch a recreative proof of the theorem for $g = 11$,
which is based on the Brill-Noether locus $\mathcal W^1_{6,11}$ parametrizing 6-gonal curves of genus 11.
\medskip \par 
\underline {\it $m_{11}: P_{11} \to \mathcal M_{11}$ is dominant} 
 \begin{proof}   Let $(C,L)$ be a pair defining a general point of $\mathcal W^1_{6,11}$. We have seen that then $\vert L \vert$ is a base point free pencil.
 Let $H := \omega_C(-L)$, then:
 \begin{lemma} $h^0(H) = 6$ and $H$ is very ample. \end{lemma} \begin{proof} Since $h^1(H) = h^0(L) = 2$, we have $h^0(H) = 6$. Assume $H$ is
 not very ample, then $h^0(H(-d)) \geq 5$ for some effective divisor $d$ of degree 2. But then we would have $h^0(L(d)) \geq 3$ and $W^2_8(C) \neq \emptyset$. This is
impossible for dimension reasons, since  $ \dim \ \mathcal W^1_{6,11} >  \dim \ \mathcal W^2_{8, 11}$. \end{proof} So we can assume that $C$ is embedded in $\mathbf P^5$ 
by $\vert H \vert$ as a connected, smooth, linearly normal curve of degree 14 and genus 11.  Let $\mathcal I_C$ be the ideal sheaf of $C$, one computes that $h^0(\mathcal I_C(2)) \geq 3$.  
\par  Let $\mathcal H$ be the open set of the Hilbert scheme of $C$ whose elements 
are curves with the same properties. Note that, for each $D \in \mathcal H$, $\vert \omega_D(-1) \vert$ is a base point free $g^1_6$ and that $\mathcal H / PGL(6)$ is birational to $\mathcal W^1_{6,11}$. 
Consider the moduli map $m: \mathcal H \to \mathcal M_{11}$ and the Zariski closure of its image
$$
\mathcal M^1_{11,6}:= \overline {m(\mathcal H)}.
$$
Then $\mathcal M^1_{11,6}$ is the Petri divisor in $\mathcal M_{11}$ parametrizing 6-gonal curves.
\begin{lemma} For a general $D \in \mathcal H$ the following conditions are satisfied: \par
1) $h^0(\mathcal I_D(2)) = 3$,  the base locus of $\vert \mathcal I_D(2) \vert$ is a smooth K3 surface $S$. \par
2) $ \Pic \ S \cong \mathbb Z[F] \oplus \mathbb Z[D]$, where $\vert F \vert$ is an elliptic pencil and $L \cong \mathcal O_C(F)$.
\end{lemma} 
\begin{proof}  1 and 2 are open conditions on $\mathcal H$, hence it suffices to produce one $D \in \mathcal H$ satisfying them.
Fix an elliptic K3 surface $S$ such that: $$  \Pic \ S \cong \mathbb Z[H] \oplus \mathbb Z[F], $$ where $H$ is a very ample, smooth
integral curve of genus 5, $\vert F \vert$ is an elliptic pencil and $DF = 6$. By the surjectivity of the period map, $S$ exists. One
computes that no $E \in  \Pic \ S$ exists such that $E^2 = 0$ and $EH = 3$. Let $S \subset \mathbf P^5$ be the embedding of $S$ 
by $\vert H \vert$. Then, by \cite{SD}, $S$ is a complete intersection of 3 quadrics. Since $H$ is very ample the same is true for
$H + F$. It is easily seen that a general $D \in \vert H + F \vert$ is a smooth, linearly normal curve of genus 11 and degree 14,
so that $D \in \mathcal H$ and $ \vert \mathcal O_D(F) \vert$ is a base point free $g^1_6$. Then, to prove that $D$
satisfies 1 and 2, it remains to show that $h^0(\mathcal I_D(2)) = 3$. This follows from  the standard exact sequence
$$
0 \to \mathcal I_S(2H) \to \mathcal I_D(2H) \to \mathcal O_S(2H-D) \to 0
$$
because $h^0(\mathcal O_S(2H-D)) = 0$. Indeed we have $2H - D = H - F$ so that $(H - F)^2 = -4$ and $H(H - F) = 2$. Assume
that $L \in \vert H - F \vert$, then it follows that $L = L_1 + L_2$, $L_1$ and $L_2$ being disjoint lines. This is numerically impossible in $ \Pic \ S$,
hence $h^0(\mathcal I_S(2)) = h^0(\mathcal I_D(2)) = 3$.
 \end{proof}
We can now conclude our proof: consider the Zariski closure $ \mathcal P^1_{11,6} \subset \mathcal P_{11} $ of the divisor parametrizing pairs $(S,C)$ such
that \par (1) $C \subset S$ is a smooth, very ample curve of genus 11, \par (2) $ \ \Pic \ S \cong \mathbb Z[F] \oplus \mathbb Z[C]$, where $\vert F \vert$ is an elliptic pencil 
and $F \cdot C  = 6$.  \par By the  lemma, the image of $\mathcal P^1_{11,6}$ by $m_{11}: \mathcal P_{11} \to \mathcal M_{11}$ is $\mathcal M^1_{11,6}$. Then, since this is a divisor
and $\mathcal P_{11}$ is irreducible, either $m_{11}(\mathcal P_{11}) = \mathcal M^1_{11,6}$ or $m_{11}$ is dominant. Assume $m_{11}(\mathcal P_{11}) = \mathcal M^1_{11,6}$
and consider a pair $(S,C)$ defining a general $x \in \mathcal P_{11}$. Then $ \Pic \ S \cong \mathbb Z[C]$ and $m_{11}(x)$ is general in $\mathcal M^1_{11, 6}$, hence
$C$ is 6-gonal. Let $L \in W^1_6(C)$,  a well known theorem implies that $L \cong \mathcal O_C(F)$, where $F \subset S$ is a curve
such that $F^2 = 0$, \cite{GL}. This is numerically impossible in $ \Pic \ S$. Hence $m_{11}$ is dominant.
\end{proof} \par 
Let $g \leq 11$,  $g \neq 10$. Due to the geometry of K3 surfaces and Mukai-Mori theorem, the uniruledness of the universal Brill-Noether
locus $\mathcal W^r_d$ easily follows in many cases. Let $W \subset \mathcal W^r_d$ be an irreducible component, we have:
\begin{proposition}  Let $g \leq 11$ and $g \neq 10$. Assume that $W$ dominates $\mathcal M_g$ and parametrizes pairs $(C,L)$ with  $L$ special. 
 Then $W$ is uniruled. 
\end{proposition}
\begin{proof} Let $(C,L)$ be a pair defining a general $x \in W$. We have 
 $C \subset S$, where $S$ is a K3 surface and $ \Pic \ S \cong \mathbb Z[C]$. Let
$Z \in \vert L \vert$ and let $\mathcal I_Z$ be its ideal sheaf in $S$. Since $Z$ is a special divisor, the standard exact sequence
$$
0 \to \mathcal O_S \to \mathcal I_Z(C) \to \mathcal O_C(C-Z) \to 0
$$
implies $ \dim \ \vert \mathcal I_Z(C) \vert \geq 1$. Moreover, one has $h^0 (\mathcal O_D(Z)) \geq r+1$ for each $D \in P :=  \vert \mathcal I_Z(C) \vert$. Then $(D, \mathcal O_D(Z))$ defines a point
of $W$. It is easy to see that the moduli map $P \to W$ is not constant. Hence $W$ is uniruled.
 \end{proof}   \par 
We recall that a  variety $X$ is $d$-ruled if it is is birational to $Y \times \mathbf P^d$ with $d > 0$. A  beautiful theorem of Mukai, \cite{Mu2}, says that:
\begin{theorem} $m_{11}: P_{11} \to \mathcal M_{11}$ is birational. Hence $\mathcal M_{11}$ is $g$-ruled.
\end{theorem} \par  
Now we describe further uniruledness constructions going up to genus
$$
g \leq 16.
$$
Actually there are several of these constructions and they often provide stronger results than uniruledness, like for instance unirationality. We will concentrate on  this property
in the next section. Here we introduce the constructions to be used and prove for them some ruledness properties.  \par 
Besides K3 surfaces,  \it canonical surfaces  which are complete intersection  \rm in $\mathbf P^n$ or $\mathbf P^a \times \mathbf P^b$
will be used. \rm The list of them is of course short:  in $\mathbf P^n$ complete intersections $S$ of type (5), (3,3), (2,4),
(2,2,3), (2,2,2,2) are the  only possible cases. Let us see two reasons for using these surfaces. \par  (1) If $C$ of genus $g$ embeds in $S$ and $\mathcal O_C(1)$ is special, then
$ \dim \ \vert C \vert \geq 1$. Therefore $\vert C \vert$ defines a rational curve through the moduli point of $C$. \par (2) Possibly $C$ is linked to a second curve $B$ in $S$ of genus $p < g$. This
fact can be used to parametrize $\mathcal M_g$ via a family of curves of lower genus $p$. 
\begin{definition} A ruling of $X$ by unirational varieties of dimension $d$ is a dominant rational map $f: X \to Y$ with unirational fibres of dimension $d > 0$.
\end{definition} \par 
Here is a list of examples in low genus which are interesting for us: 
\begin{theorem} \ \par \rm
(1) \it genus 15: $\mathcal W^1_{9,15}$ has a ruling of rational surfaces, \rm \medskip \par 
(2) \it genus 14: $\mathcal W^1_{8, 14}$ is birational to $ \Pic_{14, 8} \times \mathbf P^{10}$,  \rm \medskip \par 
(3) \it genus 13: $\mathcal W^2_{11,13}$ is dominated by $ \Pic_{12,8} \times \mathbf P^8$, \rm  \medskip \par 
(4) \it genus 12: $\mathcal W^0_{5,12}$ is birational to $ \Pic_{15,9} \times \mathbf P^5$,  \rm \medskip \par 
(5) \it genus 11: $\mathcal W^0_{6,11}$ is  birational to $ \Pic_{13,9} \times \mathbf P^3$. \rm \medskip \par 
 \end{theorem} \begin{proof} We refer to \cite{BV} and  \cite{Ve1} for a complete account of these cases and their proofs. See also Schreyer's paper \cite{Sch} in this Handbook,  where the same methods
 are simplified via an effective use of Computer Algebra.\par 
 (1) \underline {\it Genus 15, 14, and 12} \par 
Let $(C,L)$ be a general pair such that either $g = 14$ and $L \in W^1_8(C)$ or $g = 15$ and $L \in W^1_9(C)$ or $g = 12$ and $L \in W^0_5(C)$.  It is easy to see that $\omega_C(-L)$ is very ample and defines an embedding
$C \subset \mathbf P^6$. We summarize some further properties of $(C,L)$: see \cite{Ve1} section 4 and \cite{BV} section 3. \par Let $\mathcal I_C$ be the ideal sheaf of $C$ then: \par 
\begin{enumerate} \it
\item[(i)] $h^0(\mathcal I_C(2)) = 4$ for $g = 15$ and $h^0(\mathcal I_C(2)) = 5$ for $g = 14$ and $g = 12$.
\item[(ii)] $C$ is contained in a smooth complete intersection $S$ of 4 quadrics.
\item[(iii)] $g = 14$: $C$ is linked to a projectively normal curve $B$ of genus 8 by a complete intersection of 5 quadrics and $h^0(\mathcal I_B(2)) = 7$.
\item[(iv)] $g = 12$: $C$ is linked to a projectively normal curve $B$ of genus 9 by a complete intersection of 5 quadrics and $h^0(\mathcal I_B(2)) = 6$.
\end{enumerate} 
\medskip \par 
 (1.1) \it The case of genus $14$. \rm \par In particular $B$ is smooth of degree 14 and $\mathcal O_B(1)$ is a line bundle of degree 14. Clearly, all the  mentioned
properties of the pair $(B, \mathcal O_B(1))$ are true for a general smooth, connected curve in the Hilbert scheme of $B$. This  is irreducible and dominates 
$ \Pic_{14,8}$.  Let  $B' \subset \mathbf P^6$ a general curve of degree 14 and genus 8 such that $(B', \mathcal O_{B'}(1))$ defines a general point $x \in  \Pic_{14, 8}$. One has $h^0(\mathcal I_{B'}(2)) = 7$. Then, on an open set of $ \ \Pic_{14,8}$, there exists a Grassmann bundle with fibre $G(5,7)$
$$
\phi: \mathcal V_{14} \to  \Pic_{14,8}
$$
such that
$
\phi^{-1}(x) = G(5, H^0(\mathcal I_B'(2)).
$ 
Notice that, counting dimensions,
$$
 \dim \ \mathcal V_{14} =  \dim \ \mathcal W^1_{8, 14} =  \dim \ \mathcal M_{14}.
$$
The linkage now defines a rational map
$$
\psi: \mathcal W^1_{8,14} \dashrightarrow \mathcal V_{14}.
$$
Indeed $\Lambda := H^0(\mathcal I_C(2))$ is 5-dimensional and  the base locus of $\vert \mathcal I_C(2) \vert$ is $B \cup C$.  In particular $\Lambda$ defines a point of
the Grassmannian $G(5, H^0(\mathcal I_C(2))$. Hence $(B, \mathcal O_B(1), \Lambda)$ defines a point of $\mathcal V_{14}$.
The construction  is clearly modular. By definition, $\psi$ is the induced map of moduli spaces. 
Note that $\psi$ is invertible onto its image. Indeed the triple $(B, \mathcal O_B(1), \Lambda)$ uniquely defines $C$ and $\omega_C(-1) = L$. Since $\mathcal W^1_{8,14}$ and 
$\mathcal V_{14}$ are irreducible of the same dimension, it follows that $\psi$ is birational. Since $ \dim \ G(5,7) = 10$, we conclude that $ \mathcal W^1_{8, 14} \cong  \Pic_{14, 8} \times \mathbf P^{10}. $
  \medskip \par 
(1.2) \it The case of genus $12$. \rm \par
The proof of this case is completely analogous. Let  $B' \subset \mathbf P^6$ a general curve of degree 15 and genus 8 such that $(B', \mathcal O_{B'}(1))$ defines a general point $x \in  \Pic_{14, 8}$. One has $h^0(\mathcal I_{B'}(2)) = 6$.
Then there exists a $\mathbf P^5$-bundle
$$
\phi: \mathcal V_{12} \to  \Pic_{15,9}
$$
such that $\phi^{-1}(x) = \mathbf PH^0(\mathcal I_{B'}(2))^*$. Again one defines by linkage a  map
$$
\psi: \mathcal W^0_{5,12} \to \mathcal V_{12}.
$$
By definition $\psi$ sends  the moduli point of $(C,L)$ to the moduli point of $(B, \mathcal O_B(1))$, where $B \cup C$ is the base locus of $\vert \mathcal I_C(2) \vert$. One can show as above  that $\psi$ is birational. Hence it follows that $\mathcal W^0_{5, 12} \cong  \Pic_{15,9} \times \mathbf P^5$. \medskip \par 
(1.2) \it The case of genus $15$. \rm \par 
By \cite{BV} we have $h^0(\mathcal I_C(2)) = 4$ and $C \subset S \subset \mathbf P^6$,  where $S$ is a smooth (2,2,2,2) complete intersection. Moreover $ \dim \ \vert C \vert = 2$ and  the image of $\vert C \vert$
in $\mathcal W^1_{9,15}$, via the natural map, is a rational surface $R$. We can consider the Noether-Lefschetz family of smooth complete intersections $S \subset \mathbf P^6$ of type (2,2,2,2) such that
$ \Pic \ S$ is generated by $\mathcal O_S(1)$ and $\mathcal O_S(D)$, where $D$ is a smooth, integral curve of the Hilbert scheme of $C$. Let $\mathcal S$ be the GIT quotient of such a 
family. We have a dominant rational map $f: \mathcal W^1_{9,15} \dashrightarrow \mathcal S$ sending the moduli point of $(C,L)$ to the class of $S$ in $\mathcal S$. Since $S$ is the base locus of 
$\vert \mathcal I_C(2) \vert$, the fibre of $f$ at the moduli point of $S$is the rational surface $R$.
\medskip \par 
(2) \underline {\it Genus $13$ and $11$} \medskip \par 
(2.1) \it The case of genus $13$. \rm \par 
Let $\mathcal N_{13,32}$ be the irreducible family of 32-nodal plane curves of degree 11 and genus 13. We consider pairs $(\Gamma, o)$ such that $o \in  \Sing \ \Gamma$ and $\Gamma \in \mathcal N_{13,32}$. \par
Let $\nu: C \to \Gamma$ be the normalization map, $L := \nu^* \mathcal O_{\mathbf P^2}(1)$ and $n := \nu^*o$. Then $(C,L)$ defines a general point of $\mathcal W^2_{11,13}$.  We denote by $\tilde {\mathcal W}^2_{11,13}$ the
moduli space of triples $(C,L,n)$ and consider the natural  forgetful map 
$$
f: \tilde {\mathcal W}^2_{11,13} \to \mathcal W^2_{11,13}.
$$
The map $f$ has degree 32. We want to construct in our usual way a birational isomorphism
$$
\tilde {\mathcal W}^2_{11,13} \dashrightarrow  \Pic_{12,8} \times \mathbf P^{12}.
$$ 
It is standard to check that $H := \omega_C(-L) \otimes \mathcal O_C(n)$ defines a morphism
$$
h: C \to C_n \subset \mathbf P^4
$$
which is generically injective. Since $h^0(H(-n)) = 4$, it follows that the curve $C_n := h(C)$ is 1-nodal and its only node  is $h(n)$. It is known that $C_n$ is linked to a projectively normal curve $B$ of degree 12 and genus 8 by a  (3,3,3) nodal complete intersection, see \cite{Ve1} section 10. \par Since $h^0(\mathcal I_B(3)) = 6$, it follows that  there exists a \it unique \rm $F_n \in \vert \mathcal I_B(3) \vert$ having multiplicity two at $o := h(n)$. Note that this property is satisfied by a general pair $(B,o) \in \mathcal H \times \mathbf P^4$, where $\mathcal H$ is an irreducible open neighborhood of $B$, in its Hilbert scheme, parametrizing smooth curves. \par Notice also that the family of projectively normal
curves of genus 8 and degree 12 is irreducible and dominates $ \Pic_{12,8}$. Hence we can assume that $(B, \mathcal O_B(1))$ defines a general point $x \in  \Pic_{12,8}$.  \par Now let $p \in \mathbf P^4$ be general and let $F_p \in \vert  \mathcal I_B(3) \vert$ be the unique cubic with multiplicity two at $p$. Let $V_p \subset H^0(\mathcal I_{B/F_p}(3))$ be the subspace of sections vanishing at $p$. Then, on an open set of $\mathbf P^4$, there exists a natural Grassmann bundle $G_B$,  whose fibre at the point $p$ is  the Grassmannian $G(2,V_p)$. Hence the construction defines a dominant rational map
$$
\phi: \mathcal V_{13} \dashrightarrow  \Pic_{12,8},
$$
with fibre $G_B$ at the moduli point $x$. The proof then goes as previously: let $(C,L,n)$ be as above. Keeping our notations, we have a unique cubic $F_o$, containing $C_n$ and which is singular at $o$. Let
$\Lambda \subset H^0(\mathcal I_{F_o/B}(3))$ be the 2-dimensional image of $H^0(\mathcal I_{C_n}(3))$ via the restriction map. Then $\Lambda$ is an element of $G_B$. Hence the assignement 
$(C,L) \mapsto (B, \mathcal O_B(1), \Lambda)$ defines  a rational map
$$
\psi: \tilde {\mathcal W}^2_{11,13} \dashrightarrow \mathcal V_{13}.
$$
It turns out that $(C,L)$ is uniquely reconstructed from the triple $(B, \mathcal O_B(1), \Lambda)$. Hence $\psi$ is generically injective. Since it is a map between varieties of
the same dimension, then $\psi$ is birational. This implies that  
$$ \tilde {\mathcal W}^2_{11,13} \cong  \Pic_{12,8} \times \mathbf P^4 \times G(2,4) \cong  \Pic_{12,8} \times \mathbf P^8. $$
\begin{remark} \  \rm We will see  that $ \Pic_{12,8}$ is unirational. Hence it follows that: \it the family  $\mathcal N_{13,32}$ of nodal plane curves of genus 13 and degree 11 is unirational. \rm By
proposition 2.11 it is not ruled by linear spaces.  \end{remark}
\medskip \par 
(2.1) \it The case of genus $11$. \rm \par
We simply summarize our usual recipe for this case. Assume $(C,L)$ defines a general point of $\mathcal W^0_{6,11}$. Then $\omega_C(-L)$ embeds $C$ in $\mathbf P^4$ 
as a projectively normal curve. Moreover  $C$ is linked to a projectively normal curve $B$, of genus 9 and degree 13, by a (3,3,3) complete intersection. Then, over an open set of $ \Pic_{13,9}$, one has a $\mathbf P^3$-bundle
$
\phi: \mathcal V_{11} \to  \Pic_{13,9},
$
with fibre $\vert \mathcal I_B(3) \vert^*$ at the moduli point of $(B, \mathcal O_B(1))$. By linkage we can associate to $(C,L)$ the triple $(B, \mathcal O_B(1), P)$, where $P := \vert \mathcal I_{C \cup B}(3) \vert$ is an element of
$\vert \mathcal I_B(3 \vert^*$. This defines a birational map 
$\psi: \mathcal W^0_{6,11} \to \mathcal V_{11}$, see \cite{Ve3} section 8.
\end{proof} \medskip \par 
 \subsection{Unirationality results and rationality issues for $\mathcal M_g$}
 Finally we review the known rational parametrizations of $\mathcal M_g$ and their recent history. After Severi's unirationality results the first new step
 is due to Sernesi, \cite{Ser1}. In 1981 he proves the unirationality of $\mathcal M_{12}$.   In 1984 Chang and Ran prove the unirationality of $\mathcal M_g$, $g = 11, 12, 13$, \cite{CR1}. \par
 The focus was on curves in $\mathbf P^3$ and on vector bundles of rank two or higher on $\mathbf P^n$, a central topic for that period. Let $C \subset \mathbf P^3$ be a smooth, connected curve of degree $d$
 and genus $g$. The general idea is to represent $C$ as the degeneracy scheme of $(\mbox {rank} F - 1)$ global sections of a vector bundle $F$ on $\mathbf P^3$.  
 Assume $d$ is minimal to have  Brill-Noether number $\rho(d,g,3) \geq 0$ and that $C$ has general moduli. Furthermore let $h^1( \mathcal I_C(1)) = h^0(\mathcal I_C(2)) = h^0(\omega_C(-2)) = 0$. Then it is shown 
 in \cite{CR1} that a vector bundle $F$ as above always exists as  the cohomology sheaf of a complex $\Gamma$
 $$
 \mathcal O_{\mathbf P^3}^a(-1) \to \mathcal O_{\mathbf P^3}^b \to \mathcal O_{\mathbf P^3}^c(1),
 $$
where $a = \rho(d,g,3)$, $b = 5d-3g-17$, $c = 2d-g-9$, see \cite{CR1} I, Prop. 3 and Remark 3.1.  Let $\mathcal H_{d,g}$ be an irreducible open set of the Hilbert scheme of $C \subset \mathbf P^3$
which dominates $\mathcal M_g$ and contains the parameter point of $C$. Then there exists a quasi-projective variety $\mathcal F_{d,g}$ parametrizing triples $(\Gamma, F, \Lambda)$ such that $(F, \Gamma)$ is a pair as 
above and $\Lambda \subset H^0(F)$ is a vector space of dimension $(\rank \ F - 1)$ whose degeneracy scheme is an element of $\mathcal H_{d,g}$. Clearly the assignement $(F, \Gamma, \Lambda) \mapsto C$ induces
a dominant rational map
$$
f_{d,g}: \mathcal F_{d,g} \dashrightarrow \mathcal M_g.
$$
Notice also  that $\mathcal F_{d,g}$ is a Grassmann bundle over the family $\mathcal F$ parametrizing pairs $(\Gamma, F)$. Thus the unirationality of $\mathcal M_g$ follows if $\mathcal F$ is unirational. It turns out that this
approach to the unirationality problem works for the cases of genus $g = 11, 12, 13$, see \cite{CR1} Proposition 4. Hence we have: 
\begin{theorem} $\mathcal M_g$ is unirational for $g = 11, 12, 13$. \end{theorem} \par 
\begin{remark} \rm To apply the previous method a necessary condition on $F$ is that $h^1(F) = 0$. Interestingly, this is related to the existence of a quintic surface containing $C$, see \cite {CR1}ÊI remark 7.   Therefore it is related to the possibility that $C$ 
moves in a linear system on a surface of general type, a question already considered in the previous sections. \end{remark}
 Continuing with the recent history we come to the case of genus 14. The unirationality of $\mathcal M_{14}$ is a result appeared later, namely in 2005. The method of proof is relatively simple and also applies to the cases of lower genus $g \leq 14$, see \cite{Ve1}. Let us describe it with some more details. From the previous section we have \medskip \par 
\begin{itemize} \it
\item[$\circ$] $\mathcal W^1_{8,14} \cong  \Pic_{14,8} \times \mathbf P^{10}$,
\item[$\circ$] $\tilde  {\mathcal W}^2_{15,13} \cong  \Pic_{12,8} \times \mathbf P^8$,
\item[$\circ$] $\mathcal W^0_{5,12} \cong  \Pic_{15,9} \times \mathbf P^5$
\item[$\circ$] $\mathcal W^0_{6,11}  \cong  \Pic_{13, 9} \times \mathbf P^3$ 
\end{itemize}
\medskip \par 
Therefore $\mathcal M_g$ is unirational, for $g = 11, 12, 13, 14$,  if the universal Picard varieties considered above are unirational. This is true and it can be viewed as a concrete consequence of Mukai's theory
of canonical curves of low genus. \par A general canonical curve $C$ is a complete intersection for $3 \leq g \leq 5$. This just means that $C$ is a linear section of a suitable Veronese embedding: of
$\mathbf P^{g-1}$ if $g = 3,5$ and of a smooth quadric if $g = 4$. \par Is it possible to realize a general canonical curve as a linear section of a fixed variety for a few further values of $g$? The answer is yes when $g = 7, 8, 9$. Let us  provide it:
\begin{itemize} \it \medskip \par
\item[$\circ$]~For $g = 7, 8, 9$ a general canonical curve $C$ is linear section of a rational homogenous space
$
S_g \subset \mathbf P^{N_g}.
$
\end{itemize} \par 
More precisely consider the following homogeneous spaces: \medskip \par 
\rm (1) \it $S_7 :=$  the orthogonal Grassmannian $OG(5,10)$ in $\mathbf P^{15}$, \medskip \par 
\rm (2) \it $S_8 := $ the Grassmannian $G(2,6)$ in $\mathbf P^{14}$, \medskip \par 
\rm (3) \it $S_7 := $ the symplectic Grassmannian $S(3,6)$ in $\mathbf P^{13}$.
\rm \medskip \par  
These are subvarieties in their corresponding Grassmannian and we assume that the latter is embedded by its Pl\"ucker map. In particular $N_g$ is the dimension of the linear subspace spanned
by $S_g$. Let $C$ be the canonical model of a smooth integral curve of genus $g = 7, 8, 9$. Then:
\begin{theorem} \ 
\begin{itemize} \it
\item[$\circ$]~$C$ is a linear section of $S_7$  if and only if $W^1_4(C) = \emptyset$, 
\item[$\circ$] $C$ is a linear section of $S_8$  if and only if $W^2_7(C) = \emptyset$,
\item[$\circ$] $C$ is a linear section of $S_9$  if and only if $W^1_5(C) = \emptyset$.
\end{itemize} 
\end{theorem}
The theorem is due to Mukai. It follows from the study of higher rank Brill-Noether theory for a curve $C$ as above. For such a curve there exists a unique vector bundle $E$, of rank greater than 1 and determinant $\omega_C$, such that $H^0(E)$ defines an embedding of $C$ in $S_g$ as a linear section, cfr \cite{Mu1}, \cite{Mu3} and \cite{Mu4}. More precisely one has the following description of $E$:
\medskip \par 
\begin{itemize} \it
\item[$g = 7$:] $E$ is the unique  orthogonal rank 5 vector bundle on $C$ such that $h^0(E) = 10$ and $\\det E \cong \omega_C$, 
\item[$g = 8$:]$E$ is the unique  rank 2 vector bundle on $C$ such that $h^0(E) = 6$ and $\det E \cong \omega_C$, 
\item[$g = 9$:] $E$ is the unique symplectic rank 3 vector bundle on $C$ such that $h^0(E) = 6$.
\end{itemize} \medskip \par 
We remark that the spaces $S_g$ are rational. Moreover, by the previous description, a set $Z$ of $g$ general points on $S_g$ defines a divisor of degree $g$ on
a general curve $C$ of genus $g$. \par  Indeed the linear span $\langle Z \rangle$ of $Z$ cuts on $S_g$ a general canonical curve $C$ of genus $g$ and $Z$ turns out to be a general divisor of degree $g$ on $C$. 
Using appropriately this remark one obtains:
\begin{theorem} $ \Pic_{d,g}$ is unirational for $g \leq 9$ and  every $d$. \end{theorem} \par 
\begin{proof} See \cite{Ve1} thm1.2. Let $g = 7, 8, 9$. Consider the open set of elements $z = (z_1, \dots, z_g) \in S_g^g$ such that $z_1 \dots z_g$ are
linearly independent and the linear span $\mathbf P^{g-1}_z$ of them is transversal to $S_g$. Let
$$
\mathcal C := \lbrace (x,z) \in S_g \times S^g_g \ \vert \ x \in \mathbf P^{g-1}_z \cap S_g \rbrace.
$$
be the universal canonical section and $p: \mathcal C \to U$ its the projection map. Fix $n_1, \dots, n_g$ so that $n_1 + \dots + n_g = d$ and $n_1 \dots n_g \neq 0$. It is standard to construct a rational map 
$
\phi_d: S^g_g \to  \Pic_{d,g}
$
sending $z \in U$ is the moduli point of $(C,L)$, where $C = \mathbf P^{g-1}_z \cap S_g$ and $L = \mathcal O_C(n_1z_1+ \dots + n_gz_g)$. On the other hand a general
$L \in  \Pic^d(C)$ is isomorphic to $\mathcal O_C(n_1z_1+ \dots + n_gz_g)$ for some $z \in C^g$, cfr. \cite{Ve3} 1.6. Hence $\phi_d$ is dominant and $ \Pic_{d,g}$ is unirational.  
\end{proof}
\begin{remark} \rm The unirationality of $ \Pic_{d,g}$ for $g \leq 9$ and each value of $d$  is sharp. Indeed it has been recently proven that the Kodaira dimension of $ \Pic_{d,g}$
is not $-\infty$ for each $d$ as soon as $g \geq 10$: see \cite{FV} and \cite{BFV}.  For instance the Kodaira dimension of $ \Pic_{g,g}$ is 0 for  $g = 10$ and 19 for $g = 11$. Moreover it  is $3g - 3$ for $g \geq 12$, \cite{FV}.  
If $d$ and $2g-2$ are coprime, then the same result holds for $ \Pic_{d,g}$, \cite{BFV}. As an immediate corollary of the results considered above we have: \end{remark} \par
\begin{theorem} The following Brill-Noether loci are unirational: $$ \mathcal W^1_{8,14}, \ \mathcal W^2_{11,13}, \ \mathcal W^0_{5,12}, \ \mathcal W^0_{6,11}. $$
In particular $\mathcal M_g$ is unirational for $g = 11, 12, 13, 14$. 
\end{theorem}  \medskip \par 
\begin{remark} What about $g = 15, 16$? \rm \par The uniruledness of $\mathcal M_{16}$ follows from the general theory on the cone of effective divisors of an algebraic variety: it is proved in \cite{CR3} that the canonical 
class of $\overline {\mathcal M}_{16}$ is not a pseudo-effective class. Then the uniruledness follows from the main result of \cite{BDPP}, cfr. \cite{F2} thm. 2.7 . \par
In \cite{BV} it is proved that $\mathcal M_{15}$ is rationally connected.  The proof relies again on curves with general moduli moving on surfaces. Indeed a general curve $C_i$ of genus 15
moves in a Lefschetz pencil $P_i \subset \vert C_i \vert$ of a smooth (2,2,2,2) complete intersection of $S_i \subset \mathbf P^6$, $i = 1,2$.  The image of $P_i$ in $\overline {\mathcal M}_{15}$
is a rational curve $R_i$ through the moduli point $x_i$ of $C_i$. Since $P_i$ is Lefschetz $R_i$ intersects the divisor $\Delta_0$ parametrizing integral nodal curves. 
It turns out that $y_i \in R_i \cap  \Delta_0$ is general in $\Delta_0$. On the other hand $\Delta_0$ is unirational, \cite{BV} thm. 4.3. Then $y_1, y_2$ are connected by a rational curve $R_0$ of $\Delta_0$ and
$R_0 \cup R_1 \cup R_2$ is a chain of rational curves connecting $x_1$ to $x_2$. This implies the rational connectedness of $\overline {\mathcal M}_g$.
\end{remark} \par  
 \medskip \par 
We want to turn now to a very natural question:
\begin{itemize} \it
\item[$\circ$] What about the rationality of $\mathcal M_g$?
\end{itemize} \par 
Perhaps it will be considered a surprising phenomenon, in a future history of $\mathcal M_g$, that the rationality problem was still unsettled, for many unirational moduli spaces of curves, at the
present time. This is however a very difficult problem, often not approachable with the available techniques. \par We recall that the rationality of $\mathcal M_g$ is known for $g \leq 6$.  The case $g = 1$
is classical. Since $\mathcal M_1$ is a unirational curve, its rationality also follows from L\"uroth theorem. In 1960 J. Igusa proved the rationality of $\mathcal M_2$. The rationality of $\mathcal M_4$
and $\mathcal M_6$ came later: it was obtained by N. Shepherd-Barron in 1985, \cite{SB}. The rationality of $\mathcal M_5$ was then proved by Katsylo \cite{K1} and he also proved the rationality of
$\mathcal M_3$, cfr. \cite{Bo}. \par
The proof of the rationality of $\mathcal M_6$ is related to some arguments considered in this exposition. In particular it deals with 4-nodal plane sextics. Moreover
the point of view of the proof has some interplay with the parametrization of the Prym moduli space $\mathcal R_6$ considered in later section. \par  
Let $\mathbb P$ be the linear system of plane sextics passing with multiplicity $\geq 2$ through the points 
$F_1 (1:0:0)$, $F_2 (0:1:0)$, $F_3 (0:0:1)$, $U (1:1:1)$ of $\mathbf P^2$.  Consider a general $\Gamma$ in $\mathbb P$ and its normalization
$\nu: C \to \Gamma$, then $C$ is a genus 6 curve. Notice also that $\nu^* \mathcal O_{\Gamma}(1) \cong \omega_C(-L)$, where $L \in W^1_4(C)$.
We have also seen in section 2 that the degree of the natural map
$$
\phi: \mathbb P \to \mathcal M_6
$$
is 120. On the other hand consider the group of Cremona transformations
$$
G = \lbrace a \in \mathrm{Bir} (\mathbf P^2) \ \vert \text {\ \it the strict transform of $\mathbb P$ by $a$ is $\mathbb P$} \rbrace.
$$
$G$ contains the symmetric subgroup $S_4$ of projective automorphisms fixing the set $\lbrace F_1, F_2, F_3, U \rbrace$. The standard quadratic transformation
$q: \mathbf P^2 \to \mathbf P^2$, centered at $F_1, F_2, F_3$, is also in $G$. One can show that
\begin{lemma} $G$ is isomorphic to $S_5$ and it is generated by $S_4$ and $q$. \end{lemma} \par 
If $ \sigma: S \to \mathbf P^2$ is the blow up of $F_1, F_2, F_3, U$ then $S$ is a Del Pezzo surface of degree 5. Moreover it turns out that
$$
G = \mbox{Aut} \ S,
$$
the group of biregular automorphisms of $S$. Then $G$ acts linearly on the strict transform of $\mathbb P$ by $\sigma$ which is  $ \vert \omega_S^{-2} \vert$.
In particular it follows that
$$
\mathcal M_6 \cong \vert \omega_S^{-2} \vert / G.
$$
The rationality of $\mathcal M_6$ then follows because one can prove that the quotient $\mathbf P / G$ is rational. This relies on the analysis
of the linear representation of $G$ on $H^0(\omega_S^{-2})$, see \cite{SB}. 
\medskip \par 
What about the \it rationality \rm  of $\mathcal M_g$ for $7 \leq g \leq 16$? As already remarked, this appears to be a difficult question.   Even the \it ruledness \rm
of $\mathcal M_g$ seems to be unknown for the same values of $g$. We have seen examples of ruled universal Brill-Noether loci dominating
$\mathcal M_g$. We don't know similar examples for $\mathcal M_g$ when $7 \leq g \leq 16$. Since rational implies 1-ruled, a criterion of non rationality
is the non existence of a ruling of rational curves. So we conclude our discussion with the following problem:   \par \it costruct, when it is possible, rulings of
$\mathcal M_g$ by rational curves. \rm
\subsection{Slope of $\overline {\mathcal M}_g$ and related questions}
A large part of  the knowledge on the birational geometry of $\mathcal M_g$ was made accessible thanks to  Deligne-Mumford compactification $\overline{\mathcal M}_g$  of $\mathcal M_g$
and to the analysis of its cone of effective divisors. Here a substantial difference appears with respect to classical methods. However, also from this new point of view, the bridge between classical and 
modern times stays well visible. We end this part of the paper with an account on the slope of the cone of effective divisors of $\overline{\mathcal M}_g$, as well as with some free speculations: \par
Let us recall some basic facts and notations. The boundary $\overline {\mathcal M}_g - \mathcal M_g$ is the union of the following irreducible divisors:  \par  
(1)  $\Delta_0$ whose general point represents an integral, nodal curve of arithmetic genus $g$ with exactly one node, \par
(2) $\Delta_i$ whose general point represents a nodal curve $C_1 \cup C_2$, $C_1, C_2$ being integral, smooth curves 
of genus $i$ and $g-i$, where $i = 1 \dots [\frac g2]$. \par
By definition  $\delta_i$ is the class of $\Delta_i$ and $\delta$ is  the class of  $\overline {\mathcal M}_g - \mathcal M_g$. Furthermore we  denote by $\lambda$ the determinant of the Hodge bundle, having fibre $H^0(\omega_C)$ at the moduli point of $C$.  It is well known that $\lambda, \delta_0, \dots, \delta_{[\frac g2]}$ is a basis of  $ \Pic \ \overline {\mathcal M}_g \otimes \mathbf R$.  It is a well known fundamental fact that
$$
K_{\overline {\mathcal M}_g} = 13 \lambda - 2\delta - \delta_1
$$
where $K_{\overline {\mathcal M}_g}$ is the canonical class of $\overline {\mathcal M}_g$ and $g \geq 4$. Let $d = a\lambda - \sum b_i \delta_i \in  \Pic \ \overline {\mathcal M}_g$ be any divisorial class such that $b_i \geq 0$ for $i = 0, \dots, [\frac g2]$. Assume that $D$
is a divisor of class $d$, then we introduce the following:
\begin{definition} The slope of $D$ is $s(D) := \frac ab$ if $b = \mathrm{min} \lbrace b_1 \dots b_{\frac g2} \rbrace > 0$. For the case $b = 0$ we set $s(D) := \infty$. Moreover the slope of $\overline {\mathcal M}_g$ is 
$$
s(\overline {\mathcal M}_g) \  := \ \mathrm{min} \lbrace s(E), \ / \ E \  \text {\it is an effective divisor} \rbrace.
$$
\end{definition} \par 
For the slope of a canonical divisor $K$ of $\overline {\mathcal M}_g$ we  have of course $s(K) = \frac {13}2$. Notice also that $s(E) < \infty$ for every effective $E$ such that $E \cap (\overline {\mathcal M}_g - \mathcal M_g)$ is non empty: see \cite{HM}. Hence it follows that $mK_{\overline {\mathcal M}_g}$ is not an effective class if $s(g) > \frac{13}2$ and $m \geq 1$. \par
Let $E$ be an effective divisor of class $a\lambda - b\delta$ such that $b > 0$. Assume that $R \subset \overline {\mathcal M}_g$ is an integral curve moving in a family which covers $\overline {\mathcal M}_g$. An obvious remark is that then $ER \geq 0$. But then it follows 
$$
s(E) = \frac ab \geq \frac {\delta R}{\lambda R}.
$$
Hence lower bounds for the slope of $\overline {\mathcal M}_g$ can be obtained by considering  curves $R$ moving in a family as above. An effective application of this remark
is possible when
$
R = m(P')
$,
where $P'$ is a base point free pencil of stable curves of genus $g$ on a smooth surface $S'$.  In such a case one has the well known formulae
$$
\lambda R = \chi(\mathcal O_{S'}) + g - 1 \ , \ \delta R = c_2(S') + 4(g - 1)
$$
cfr. \cite{CR2}. It is of course tempting to test them on a K3 surface $(S, \mathcal O_S(C))$ of genus $g$. In this case $P'$ is the strict transform on $S'$ of a pencil $P \subset \vert C \vert$ and $S'$ is the blowing up of $S$ at
the base locus of $P$. One computes that
$$
\frac {\delta R}{\lambda R} = 6 + \frac {12}{g + 1}.
$$
This is a fascinating formula because it equals the slope of the canonical class $K_{\overline {\mathcal M}_g}$ for $g = 23$. Now $g = 23$ was, before of later results of Farkas, the minimum value of $g$ for 
which the non uniruledness of $\mathcal M_g$ was known.  Indeed Farkas proved in \cite{F6} that $\overline {\mathcal M}_{23}$ has Kodaira dimension $k(\overline {\mathcal M}_{23}) \geq 2$.
The slope conjecture, see \cite{HMo}, relies on this and further motivations. \medskip \par 
\underline{\it Slope conjecture} 
The conjecture says that 
$$
s(\overline {\mathcal M}_g) \leq 6 + \frac {12}{g + 1}.
$$
The conjecture implies that $\overline {\mathcal M}_g$ has Kodaira dimension $-\infty$ for $g \leq 22$. The formula yelds a lower bound to $s(\overline {\mathcal M}_g)$ for $g \leq 11$ and $g \neq 10$, because for these values a general
$C$ is a hyperplane section of a K3 surface and moves in a pencil $P$ as above. For $g = 10$ the conjecture says that $s(10) = 7 + \frac 1{11} > 7$.  On the other hand the family of stable curves of genus 10 which can be embedded in a K3 surface defines an integral effective divisor
$$
E_{K3} \subset \overline {\mathcal M}_{10}.
$$
This divisor is a first counterexample to the conjecture. In \cite{FP} Farkas and Popa show that \it $s(E_{K3})$ is exactly 
7, \rm so that the slope conjecture is false in genus 10. This is a starting point of a series of counterexamples to the conjecture constructed in \cite{F4}.  In particular one has:
\begin{theorem} The slope conjecture is false for genus $g = 6i + 10$.
\end{theorem} \par 
Counterexamples are in this case the divisors $D \subset \overline {\mathcal M}_g$ parametrizing the locus of curves $C$ such that, for some $L \in W^1_{3i+6}(C)$,  the line bundle $\omega_C(-L)$ does not satisfy the $N_i$ condition of Green-Lazarsfeld,  \cite{GL1}. \par  Going back to genus 10 case, we have that $E_{K3}$ is one of these divisors. As pointed out in \cite{FP} the divisor $E_{K3}$ has indeed different incarnations. Let $C$ be a general curve of genus 10 and $L \in W^1_6(C)$. Consider $H := \omega_C(-L)$ and the multiplication map
$$
v: \mathrm{Sym}^2 H^0(H) \to H^0(H^{\otimes 2}).
$$
Since $h^0(H^{\otimes 2}) =  15$ and $h^0(H) = 5$, it follows that $v$ is a map between vector spaces of the same dimension. Then the condition that $v$ is not an isomorphism is a divisorial condition on $\mathcal M_{10}$. On the other hand this is just condition $N_1$ and characterizes the K3 divisor, see  \cite{FP}:
\begin{proposition} A stable $C$ as above defines a general point of $E_{K3}$  if and only if $v$ is not an isomorphism for some $L \in W^1_6(C)$.
\end{proposition}  \par 
\begin{remark} \rm Let us give, building on the previous properties, a recipe to give a proof that $s(E_{K3}) \leq 7$. We leave to the reader its completion. \par Fix a smooth complete intersection of three quadrics $ X \subset \mathbf P^5$,
which contains an integral, non degenerate, sextic elliptic curve $E$. Then fix a pencil  $P$ of degree 2 on $E$. This defines the rational quartic scroll
$$
Y = \cup \langle d \rangle, \ d \in P,
$$
where $\langle d \rangle$ is the line spanned by $d$. Then consider the surface $S = X \cup Y$ and observe that the hyperplane sections of $S$ are reducible, stable curves of arithmetic genus 10 and degree 12. Let $P \subset \vert \mathcal O_S(1) \vert$ be a general pencil of hyperplane sections and $m: P \to \overline {\mathcal M}_{10}$ be the moduli map. Computing as usual the numbers $m^* \delta$ and $m^* \lambda$ one obtains $\frac {m^* \delta }{m^* \lambda} = 7$. \par This  implies that $s(E) \leq 7$ for every effective divisor $E$ such that $m^{-1}(E)$ is empty. But this is the case when $E = E_{K3}$. Indeed, it is standard to show that the singular surface $S$ is regular and projectively normal. The regularity of $S$ follows from the Mayer-Vietoris type exact sequence 
$$
0 \to \mathcal O_S \to \mathcal O_X \oplus \mathcal O_Y \to \mathcal O_E \to 0,
$$
passing to its associated long exact sequence. In a similar way, the property $h^1(\mathcal I_S(m)) = 0$, $m \geq 1$, follows from the exact sequence of ideal sheaves
$$
0 \to \mathcal I_S(m) \to \mathcal I_X(m) \oplus \mathcal I_Y(m) \to \mathcal I_E(m) \to 0.
$$
Finally let $C$ be any hyperplane section of $S$. Then one can deduce that $h^1(\mathcal I_C(m)) = 0$, $m \geq 1$, using the standard exact sequence
$$
0 \to \mathcal I_S(m) \to \mathcal I_C(m) \to \mathcal O_C(m-1) \to 0.
$$
But then the previous map $v$ is an isomorphism for each $C \in P$ and $m^{-1}(E_{K3})$ is empty. \par 
 \end{remark} Let us conclude this section by stressing the fact that an important step is still missed in the knowledge of the moduli spaces $\overline {\mathcal M}_g$, namely: \medskip \par
\centerline { \it What is the Kodaira dimension for $17 \leq g \leq 21$? } \medskip \par 
In recent times, sporadic examples have been discovered of moduli spaces $\mathcal X_g$, related to curves of genus $g$, of intermediate Kodaira dimension. This brings about a change of perspective on moduli spaces related to curves. \par
Let us denote the Kodaira dimension of $\mathcal X_g$ by $ k({\mathcal X}_g)$. We can say that  \it $\mathcal X_{g_0}$ has intermediate Kodaira dimension \rm if the value of the function $k_{\mathcal X}$ is not $-\infty$ nor equal
to its maximum.  \par It would be of course very interesting  to discover examples of intermediate Kodaira dimension in the sequence of the moduli spaces $\mathcal M_g$.

 \section {Moduli of spin curves}   \rm
 \subsection{Modern origins and fundamental constructions}
Passing from $\mathcal M_g$ to the universal Picard varieties
$$
p:  \Pic_{d,g} \to \mathcal M_g,
$$
we  want now to consider some remarkable multisections of $p$. These are the moduli space of pairs $(C,L)$ such that $C$ is a smooth, integral curve of genus $g$
and $L \in Pic^d C$ is a line bundle satisfying some special property. For a multisection as well, we want to discuss the same type of problems considered in the previous
section for $\mathcal M_g$. The pairs $(C,L)$ to be considered in this section have a name:
\begin{definition} A spin curve is a pair $(C,L)$ as above  such that $L$ is a theta characteristic. \end{definition} \par 
Recall that a theta characteristic $L$ is even (odd) if $h^0(L)$ is even (odd). A spin  curve $(C,L)$ is said to be \it even \rm (\it odd \rm) if $L$ is even (odd). 
A modern study of families of spin curves was taken up by Mumford in 1971, \cite{M7}. In particular he proved the following theorem.
\begin{theorem}[Mumford] Let $\lbrace (C_t, L_t), \ t \in T \rbrace$ be a family of spin curves. Then $h^0(L_t) \mathrm{mod} \ 2$ is constant on each connected component
of $T$. \end{theorem} \par 
The  theorem implies that the moduli space $\mathcal S_g$ of spin curves is not connected. It is well known that the connected components of $\mathcal S_g$
are exactly two and that they are irreducible. As usual we denote them as
$$
\mathcal S^+_g \ , \ \mathcal S^-_g.
$$
They are the moduli spaces of even and odd spin curves respectively. The first natural compactifications 
$$
\overline {\mathcal S}^+_g \ , \ \overline {\mathcal S}^-_g
$$
of these moduli spaces were constructed by Cornalba. They are normal projective variety whose fundamental properties are studied in \cite{C}. A further analysis of
these spaces and their compactifications can be found in \cite{CC} and \cite{BF}. See also \cite{AJ} for the moduli spaces of generalized spin curves, that is, pairs $(C,L)$
such that $L$ is a $p$-root of $\omega_C$, for a fixed $p$. \par  
In order to improve the picture of the birational geometry of $\overline {\mathcal S}^+_g$ and $\overline {\mathcal S}^-_g$, it is useful to recall the boundary divisors
of Cornalba's compactification, see \cite{C} section 7. 
One has to consider \it semistable \rm curves $C$ of genus $g$. 
\begin{definition} An irreducible component $E$ of a semistable curve $C$ is exceptional if $E = \mathbf P^1$ and $E \cap  \Sing \ C$ is a set of two
points. \end{definition} \par 
\begin{definition} A spin structure on $C$ is a pair $(a, L)$ such that: \par 
(1) $L \in  \Pic \ C$ and $\deg L \otimes \mathcal O_E = 1$, for any exceptional component $E \subset C$. \par 
(2) $a$ is a homomorphism $a: L^{\otimes 2} \to \omega_C$ wich is not zero on each irreducible, non exceptional component.
\end{definition} \par 
Let $D \subset C$ be the union of the  non exceptional components. The locally free sheaf $\mathcal O_D(L)$ is a square root of $\omega_D$, see \cite{C} section 2 p.564. $L$ is a theta characteristic
if $C$ is smooth. Cornalba's compactification of $\mathcal S^{\pm}_g$ is the moduli space of triples $(C,a,L)$. To have a quick view of the boundary divisors,
consider the forgetful map
$$
f: \overline {\mathcal S}^+_g \cup \overline {\mathcal S}^-_g \to \overline {\mathcal M}_g,
$$
sending the moduli of $(C, a, L)$ to the moduli point of the stable reduction of $C$. Let $1 \leq i \leq \frac g2$. For each boundary divisor $\delta_i \subset \overline {\mathcal M}_g$ one has: \medskip \par 
\begin{itemize} \it
\item[$\circ$]
$f^*\delta_i \cdot \overline {\mathcal S}^+_g = 2 (\alpha^+_i + \beta^+_i)$, 
\item[$\circ$] $f^*\delta_i \cdot \overline {\mathcal S}^-_g = 2(\alpha_i^- + \beta_i^-)$.
\end{itemize} \medskip \par 
where $\alpha_i^+ , \beta_i^+ , \alpha_i^- , \beta_i^-$ are integral divisors. They are defined as follows. A general $x$ in $f^{-1}( \delta_i)$ is the moduli point of a $(C, a, L)$ such that: \medskip \par  \it
(i) $C = C_1 \cup E \cup C_2$, where $C_1, C_2$ are smooth, integral curves respectively of genus $i$ and $g-i$, $E$ is an exceptional component; \par
(ii) if $x \in \alpha_i^+$, ($x \in \beta_i^+$), then $L$ restricts to an even (odd) theta on $C_1$, $C_2$; \par
(iii) if $x \in \alpha_i^-$, ($x \in \beta_i^-$), then $L$ restricts to an even (odd) theta on $C_1$ and to an odd (even) theta on $C_2$.
\rm \medskip \par 
Moreover one has
$$
f^*(\delta_0) = (\alpha_0^+ + \alpha_0^-) + 2(\beta_0^+ + \beta_0^-),
$$ 
where $\alpha^{\pm}$, $\beta^{\pm}$ are integral divisors and
$$
\alpha_0^+ \cup \beta_0^+ \subset \overline {\mathcal S}_g^+ \ , \ \alpha_0^- \cup \beta_0^- \subset \overline {\mathcal S}_g^-.
$$
Here a general $x \in f^{-1}(\delta_0)$, is a triple $(C, a, L)$ such that: \medskip \par  \it
(i) if $x \in \alpha_0^+ \cup  \alpha_0^-$ then $C$ is integral with exactly one node, \par
(ii)  if $x \in \beta_0^+ \cup \beta_0^-$ then $C = D \cup E$, $D$ is integral, $E$ is exceptional.
\rm  \medskip \par 
From now on we denote as
$$
f^+: \overline {\mathcal S}^+_g \to \overline {\mathcal M}_g \ , \ f^-: \mathcal S^-_g \to \overline {\mathcal M}_g,
$$
the restrictions of $f$  to $\overline {\mathcal S}^+_g$ and $\overline{\mathcal S}^-_g$. The maps $f^+$, $f^-$ are finite and their degrees are  respectively the numbers of even and odd thetas on a smooth $C$. 
\subsection{The picture of the Kodaira dimension} The ramification divisor 
of $f$ is supported on the boundary. Moreover the previous characterizations of the divisors $f^*(\delta_i)$ explicitely describe such a divisor.  Therefore, using the formula for the canonical class of
$\overline {\mathcal M}_g$ and the standard formula for the canonical class of a finite covering, one can compute the canonical classes of 
$\overline {\mathcal S}^+_g$ and $\overline{\mathcal S}^-_g$, cfr. \cite{F1} and \cite{FV} p.4.  Let $\lambda$ be the first Chern class of the Hodge bundle on $\overline {\mathcal M}_g$, we fix the further notation:
$$
\lambda^{\pm} := f^{\pm*}\lambda.
$$
Then the canonical classes we are looking for are the following:
$$
K_{\overline {\mathcal S}^+_g} \equiv f^{+*} K_{\overline {\mathcal M}_g} + \beta^+_0 \equiv 13 \lambda^+ - 2\alpha^+_0 - 3\beta^+_0  - 2 \sum_{1\leq i \leq \frac g2}(\alpha^+_i + \beta^+_i) 
- (\alpha^+_1 + \beta^+_1)
$$
and
$$
K_{\overline {\mathcal S}^-_g} \equiv f^{-*} K_{\overline {\mathcal M}_g} + \beta^-_0 \equiv 13 \lambda^- - 2\alpha^-_0 - 3\beta^-_0  - 2 \sum_{1\leq i \leq \frac g2}(\alpha^-_i + \beta^-_i) 
- (\alpha^-_1 + \beta^-_1)
$$
A systematic approach to the birational geometry of the moduli spaces of spin curves, with special regard to the Kodaira dimension, is due 
to Farkas. In \cite{F1} some crucial divisorial classes are studied as well as some analogs of the slope conjecture for $\overline{\mathcal M}_g$. See also \cite{F3} for a general account on this subject.  
Among the most interesting divisors of $\overline {\mathcal S}^+_g$ one has at least to mention the \it  spin Brill-Noether divisors \rm considered
in \cite{F5}. \par Fix $r \geq 0$ such that $(r+1)s = g$ and $d = 2i = r(s+1)$. Then the Brill-Noether number $\rho(r,d,g)$ is zero. Consider the Zariski closure
$$
E^{\pm}_{r,d,g} \subset \overline {\mathcal S}_g^{\pm}
$$
of the moduli of general spin curves $(C,L)$ endowed with some $H \in W^r_d(C)$ satisfying the next condition $\star$.   Let $\phi: C \to \mathbf P^n$ be the morphism defined by $H \otimes L$, then:
 \medskip \par    \it
$\star$ there exists  a subspace $\Lambda$ of dimension $i-2$ such that $\phi^* \Lambda$ has degree $i$. 
\rm  \medskip  \par   
A linear subspace $\Lambda$ as above is said to be $i$-secant to the map $\phi$.   Farkas shows that the $E^{\pm}_{r,g,d}$ is a divisor. The divisors $E^{\pm}_{r,g,d}$ are in some sense analogs of the Brill-Noether divisors of $\overline {\mathcal M}_g$.
They have good extremality properties in the cone of effective divisors, so one can use them to test the effectivity of the canonical class
of $\overline {\mathcal S}_g^{\pm}$. \medskip \par   \underline {\it Even spin curves } \par Building on these methods Farkas proves in \cite{F1} that:
 \begin{theorem} The moduli space of even spin curves is of general type for $g \geq 9$ and  uniruled for $g \leq 7$.
\end{theorem} 
 \par 
 \underline {\it Genus 8} \ \par 
Let us discuss the only left case, namely the case of genus 8. A further very interesting divisor in $\overline {\mathcal S}^+_g$ is the theta-null divisor
$$
\theta^+_{null}.
$$
This is just the Zariski closure in $\overline {\mathcal S}^+_g$ of the moduli of spin curves $(C,L)$ such that $L$ is a theta-null, that is,
$h^0(L) = 2$. Its image by $f$ is the usual theta-null divisor: $\theta_{null} \subset \overline {\mathcal M}_g$. In particular the equivalence
$$
\theta^+_{null} \equiv \frac 14 \lambda^+ -  \frac 1{16} \alpha_0^+  - \frac 12 \sum_{i = 1 \dots [\frac g2]} \beta_i \in \Pic \overline {\mathcal S}^+_g
$$
is a remarkable expression for the class of $\theta_{null}^+$, proved in \cite{F1} theorem 0.2. Note that the formula does not depend on the genus $g$ of $C$.   Using it one computes that  
in genus 8 the canonical class is effective too. Indeed the formula implies that
$$
K_{\overline {\mathcal S}^+_8} \equiv c\pi^+ + 8\theta^+_{null} + \sum_{i = 1 \dots 4} (a_i \alpha^+_i + b_i \beta^+_i),
$$
see \cite{F1}. Here $c, a_i, b_i$ are in $\mathbb Q^+$ and $\pi^+ := f^*P$, where  
$
P  \subset \overline {\mathcal M}_8
$
 denotes the Zariski closure of the moduli of plane curves of  degree 7 and geometric genus 8.   
 To conclude the picture, it is shown in \cite{FV}  that:
 \begin{theorem} The moduli space of even spin curves of genus 8 has Kodaira dimension zero.
 \end{theorem}   \par 
 Differently from the case of $\overline {\mathcal M}_g$, the description of the Kodaira dimension of $\overline {\mathcal S}^+_g$ is therefore 
 complete. Still a natural question is open: \medskip \par 
 QUESTION: \it Does it exists a Calabi-Yau variety birational to $\overline {\mathcal S}^+_8$?
 \medskip \par  \rm
  \underline {\it Odd spin curves} \par
Passing to the moduli space $\overline {\mathcal S}^-_g$, the complete picture of the Kodaira dimension is obtained in  \cite{FV2}. Here there is no intermediate
 case and the Kodaira dimension assume only two values. Actually one has:
 \begin{theorem} The moduli space of odd spin curves is of general type for $g \geq 12$ and uniruled for $g \leq 11$.
 \end{theorem} \par 
The picture of the Kodaira dimension of the moduli of spin curves is therefore complete for every genus.   The methods of proof, in the even
and the odd case, are in some sense related. Concerning the odd case, it is worth to mention  that the proof relies on another effective divisor 
of $\overline {\mathcal M}_g$, different from the divisor $\theta_{null}$. This is the Zariski closure
$$
D_g \subset \overline {\mathcal S}^-_g
$$
of the moduli of pairs $(C,L)$ such that $C$ is smooth and $L \cong \mathcal O_C(d)$, where $d = 2x_1 + x_2 + \dots + x_{g-2}$ is a \it singular \rm 
effective divisor of $C$. Computing the class of $D_g$ is an important step, see \cite{FV2} 6.1:
\begin{theorem} Let $\sigma$ be the class of $f^{-*}D_g$ in $\overline {\mathcal S}^-_g$. Then:
$$
\sigma = (g + 8)\lambda - \frac {g + 2}4 \alpha_0  - 2\beta_0 - \sum_{i = 1 \dots \frac g2} 2(g-i)\alpha_i - \sum_{i = 1 \dots \frac g2} 2i\beta_i.
$$
\end{theorem} 
 \subsection{K3 surfaces and the uniruledness of $\mathcal S^{\pm}_g$ in low genus}
Once more the uniruledness results for $\overline {\mathcal S}_g^{\pm}$, as well as the Kodaira dimension 
zero of $\overline {\mathcal S}^+_8$, appear as completely related to the world of K3 surfaces.  \par
Let $S$ be a smooth surface and $\vert H \vert$ a linear system on $S$ whose general element is a smooth, integral curve. We introduce the following: 
\begin{definition} 
\par 
A theta pencil of $(S, \vert H \vert)$ is a triple $(P, Z, E)$ such that:
  \begin{itemize} \it
\item[$\circ$] $P \subset \vert H \vert$ is a pencil whose general member is smooth, irreducible;
\item[$\circ$] $Z$ is a subscheme of the base locus of $P$ and $E \in  \Pic \ S$;
\item[$\circ$] $E \otimes \mathcal O_C(Z)$ is a theta characteristic for any smooth $C \in P$.
 \end{itemize} \par 
A theta pencil $(P, Z, E)$ is  even (odd) if $E \otimes \mathcal O_C(Z)$ is even (odd).
\end{definition} \par 
Assume that $(C,L)$ is a spin curve of genus $g$, defining a \it general \rm point $x$ in the
moduli space $\overline {\mathcal S}_g$ of even or odd spin curves.  The aim of this section is to put in evidence the following, so to say,  principle:
 \medskip \par 
$\star$ \it $x$ moves in a rational curve of $\overline {\mathcal S}_g$  if and only if $(C,L)$ moves in a theta pencil on a K3 surface $S$.
\rm \medskip \par  
\medskip \par \underline {\it Odd theta pencils on K3 surfaces} \rm \par 
An easy, but useful,  example of theta-pencil is provided by a K3 surface $S$ polarized by $\mathcal O_S(H)$. Assume $C \in \vert H \vert$ is integral, smooth and let 
$Z \subset C$ be an effective divisor defining a theta characteristic $L = \mathcal O_C(Z)$.  Consider the ideal
sheaf $\mathcal I_Z$ of $Z$ in $S$ and the standard exact sequence 
$$
0 \to \mathcal I_Z(C) \to \mathcal O_S(C) \to \mathcal O_Z(C) \to 0.
$$
From $\mbox {deg} \  Z = g-1$ and $h^0(\mathcal O_S(C)) = g+1$, it follows $h^0(\mathcal I_Z(C)) \geq 2$. Let 
$$
P \subset \vert \mathcal I_Z(C) \vert
$$
be any pencil such that $C \in P$. Then $(P, Z, \mathcal O_S)$ is a theta-pencil. More geometrically assume $H$ is very ample and consider the embedding 
$$
S \subset \mathbf P^g,
$$
defined by $\vert H \vert$. Then $C$ is a hyperplane section of $S$ and a canonical curve. In particular $Z$ spans a linear space $\Lambda$ of codimension
$h^1(\mathcal O_C(Z)) + 1 \geq 2$. The pencil $P$ is cut on $S$ by a pencil of hyperplanes through $\Lambda$. \par In the general case we have $h^0(\mathcal O_C(Z)) = 1$, 
so that $\mathcal O_C(Z)$ is an odd theta characteristic.  It remains to produce examples of even theta pencils
$(P, Z, E)$, such that $E(Z) \otimes \mathcal O_C$ is non effective. \medskip \par 
\underline {\it Even theta pencils on Nikulin surfaces} \rm \par 
\begin{definition} A Nikulin surface of genus $g$ is a triple $(S, \mathcal O_S(H), E)$, where $(S, \mathcal O_S(H))$ is a K3 surface  of genus $g$ and $E \in  \Pic \ S$  is  such
that \begin{itemize} \item[$\circ$] $ E^{\otimes 2} \cong \mathcal O_S(E_1 + \dots + E_8)$, \item[$\circ$]  $E_1, \dots, E_8$ are two by two disjoint copies of $\mathbf P^1$, \item[$\circ$]  $HE = 0$.
\end{itemize}
\end{definition} \par 
Nikulin surfaces are well known, see \cite{N} and \cite{vGS} for a recent account. The divisor $E_1 + \dots + E_8$ is the branch locus of the finite double covering 
$$
\hat \pi: \hat S \to S
$$
defined by $E$. As is well known $\pi^{-1}(E_1), \dots, \pi^{-1}(E_8)$ are exceptional lines and their contraction is a minimal K3 surface $\tilde S$ endowed with an involution
$$
\iota: \tilde S \to \tilde S
$$
with 8 fixed points. An involution $i$ on a K3 surface $X$ with exactly 8 fixed points is known as a \it Nikulin involution. \rm   
\begin{lemma} Let $(S,\mathcal O_S(H), E)$ be a Nikulin surface of genus $g \geq 2$. Then $E \otimes \mathcal O_C$ is a non trivial 2-torsion element of $ \Pic \ C$ for any $C \in \vert H \vert$.
\end{lemma} 
\begin{proof} It suffices to show that $\pi^{-1}(C)$ is connected for any $C \in \vert H \vert$. If not we would have $\pi^*C = C_1 + C_2$,  where $C_1$ and $C_2$ are disconnected copies of $C$. 
Since $g \geq 2$, it would follow $C_1^2 = C_2^2 = 2g-2 > 0$ and  $C_1C_2 = 0$. This contradicts Hodge Index Theorem because then $C_1^2C_2^2 - (C_1C_2)^2 > 0$.
 \end{proof} \par 
Let $(S, \mathcal O_S(H),E)$ be a Nikulin surface of genus $g \geq 2$. To construct some examples of theta pencils such that $E \otimes \mathcal O(Z)_C$ is even, fix a smooth $C \in \vert H \vert$ 
and consider $\eta := E \otimes \mathcal O_C$. \par
 Via tensor product $\eta$ defines a permutation $p$ of the set of all thetas  of $C$. Since $\eta$ is not trivial $p$ is not the identity. Then it is well known that there exists effective divisors $Z \subset C$
 such that $\eta(Z)$ is an even theta characteristic. Fixing such a $Z$ we can consider any pencil
$$
P \subset \vert \mathcal I_Z(H) \vert.
$$
Then $(P, Z, E)$ is a theta-pencil such that $E \otimes \mathcal O_C$ is even. Indeed, by Mumford's theorem stated in 3.1, $E \otimes \mathcal O_C(Z)$ is  even for each  $C \in P$. Notice also that
$h^0(E \otimes \mathcal O_C)$ is constant too, since it is equal to $h^0(\mathcal I_Z(C)) - 1$. \par  \medskip \par 
Let $(P, Z, E)$ be a theta pencil on a smooth surface $S$ birational to a K3 surface. Then we have the natural map in the moduli space
$$
m: P \to \overline {\mathcal S}^{\pm}_g.
$$
\begin{definition} A K3 rational curve $R$ of $ \overline {\mathcal S}^{\pm}_g$ is a curve
$$
R = m(P).
$$
\end{definition} \par 
The above examples are all what we essentially need to prove that:
\begin{theorem} \  \begin{enumerate}
\item $\mathcal S^-_g$ is covered by the family of K3 rational curves for $g \leq 11$,
\item $\mathcal S^+_g$ is covered by the family of K3 rational curves for $g \leq 7$.
\end{enumerate}
In particular:  \ \begin{enumerate} \item $\mathcal S^-_g$ is uniruled for $g \leq 11$, 
\item ${\mathcal S}^+_g$ is  uniruled for $g \leq 7$.
\end{enumerate}
\end{theorem}  
\begin{proof} To give the proof of the theorem we distinguish the two cases:  \medskip  \par 
(1) \underline { \it Odd spin curves} \rm \par 
Let $(C,L)$ be a smooth, general  odd spin curve of genus $g $. Then $L$ is isomorphic to $\mathcal O_C(Z)$,  where $Z$ is  effective, supported on $g - 1$ distinct points and isolated. Assume $g \leq 11, 
g \neq 10$. From Mukai's description of canonical curves of low genus, we know that there exists an embedding of $C$ in a K3 surface $S$ of genus $g$ so that $\mathcal O_S(C)$ is very ample, see \cite{Mu1}. 
Let $P = \vert \mathcal I_Z(C) \vert$ be as above. Then $(P, Z, \mathcal O_S)$ is a theta pencil and $m(P)$ is a K3 rational curve through the moduli point of $(C,L)$. \par 
Since a general smooth curve $C$ of genus 10 does not embed in a K3 surface, the latter argument does not work for $g = 10$. However, let $(C, x, y)$ be a general 2-pointed curve of genus 10. It follows
from \cite{FKPS} theorem 5.1 and remark 5.2  that:
\begin{proposition} \ \begin{enumerate} \item $C$ embeds in a surface $S$ which is a K3  blown up in one point, \item $x+y = C \cdot E$, where $E$ is the exceptional line in $S$, \item the image of $C$ in the minimal model
of $S$ is very ample.
\end{enumerate} 
\end{proposition} \par 
Now assume that $x, y \in Z$, where $L := \mathcal O_C(Z)$ is an odd theta. From $\mathcal O_C(C+E) \cong \omega_C$ and $\mathcal O_C(E) \cong \mathcal O_C(x+y)$, we have  the standard exact sequence
$$
0 \to \mathcal O_S \to \mathcal O_S(C) \to \omega_C(-x-y) \to 0.
$$ 
Since $S$ is regular, it follows that there exists a pencil
$
P \subset \vert C \vert
$
with base locus the effective divisor $Z - x - y$.    In particular $(P, Z-x-y, \mathcal O_S(E))$ is a theta-pencil. So it defines a K3 rational curve $m(P)$ passing through the moduli point of $(C,L)$.  
Assume that $(C,L)$ is a general, odd spin curve of genus 10 and that $x, y \in Z$, where $Z$ is an effective divisor and $L \cong \mathcal O_C(Z)$.  It follows from \cite{FV2}, theorem 3.10 and its proof,  that then $C$ is embedded, exactly as in the previous proposition, 
in a K3 surface $S$ blown in one point.  This extends statement (1) to genus 10 and completes the proof for the moduli of odd spin curves.
\medskip \par 
(2) \underline{\it Even spin curves} \rm \par 
In this case we take profit of  Nikulin surfaces and their theta pencils. As is well known,
the moduli space of Nikulin surfaces of genus $g$ has  dimension 11 while $ \ \dim \ \mathcal F_g = 19$. Hence, counting dimensions, we have
\begin{lemma} Let $C$ be a general curve of genus $g$. Then:
\begin{enumerate}
\item $C$ does not embed in a K3 surface of genus $g$ if $g \geq 12$,
\item $C$ does not embed in a Nikulin surface of genus $g$ if $g \geq 8$.
\end{enumerate}
\end{lemma} 
To complete the proof of theorem 3.7, we previously answer the following question.
 Let $C$ be a smooth, general curve of genus $g$ 
\medskip \par 
\begin{itemize} \it
\item[$\circ$]  When there is a Nikulin surface $(S, \mathcal O_S(H), E)$ such that $C \in \vert H \vert$?  
\end{itemize} \medskip \par 
This is the analog for Nikulin surfaces of the same question for general K3 surfaces of genus $g$.  As we know the answer to the latter one is $g \leq 11$ with the exception $g = 10$.
Interestingly, an exception appears in the case of Nikulin surfaces too.
\begin{theorem} Let $C$ be a smooth, general curve of genus $g$. Then there exists a Nikulin surface $(S, \vert H \vert, E)$ such that $C \in \vert H \vert$
 if and only if $g \leq 7$ and $g \neq 6$.
\end{theorem} 
\begin{proof} We refer to \cite{FV}  for the complete argument. Because of the nice geometry behind it, we sketch the proof in genus 7. 
 Fix a non trivial line bundle $L$ on $C$ such that
$L^{\otimes 2} \cong \mathcal O_C$. Since $C$ is general $\omega_C \otimes L$ defines an embedding
$
C \subset \mathbf P^5
$
of $C$ as a projectively normal curve. In particular one has $h^0(\mathcal I_C(2)) = 3$. One can show that
$$
C \subset S,
$$
where $S$ is a smooth complete intersection of 3 quadrics. Actually, by \cite{FV}  2.3,  $S$ is a Nikulin surface polarized by $\mathcal O_S(C)$. To
see this consider $E := C - D$, where $D$ is a hyperplane section of $S$. Then observe that $2E$ is an effective class. This follows considering the standard exact sequence
$$
0 \to \mathcal O_S(C-2D) \to \mathcal O_S(2C - 2D) \to L^{\otimes 2} \to 0.
$$
Since $C$ is quadratically normal, it follows that $h^1(\mathcal O_S(2D - C)) = 0$. Then, by Serre duality, we have $h^1(\mathcal O_S(C - 2D)) = 0$. Passing to the associated long exact sequence,
it follows that $h^0(\mathcal O_S(2E)) = h^0(L^{\otimes 2}) = 1$. Let $F$ be the unique element of $\vert 2E \vert$. Note that $F^2 = -16$ and that $DF = 8$. A more careful analysis
shows that $F = F_1 + \dots + F_8$, the summands being two by two disjoint lines. Hence $(S, \mathcal O_S(C), E)$ is a Nikulin surface of genus 7. \end{proof} \par 
 \begin{remark} \rm  Consider a general Nikulin surface $(S, \mathcal O_S(C), E)$ of genus 6. It is possible to show that  $\vert C - E \vert$ defines an embedding $S \subset \mathbf P^4$ so that $S$
  is a complete  intersection of type (2,3). Then it follows $\mathcal O_C(1) \cong \omega_C(E)$. Moreover $L := \mathcal O_C(E)$ is a non trivial element of $ \Pic^0_2(C)$.
 In other words $C$ is embedded in $\mathbf P^4$ as a Prym canonical curve, (see the next section 4). \par It is well known that a general Prym canonical curve $C$ of genus 6 is quadratically normal. This follows,
 for instance, from the description of the Prym map in genus 6 and of its ramification, see \cite{DS} section 4.  In particular this implies that no quadric contains 
 $C$ and contradicts the existence of $S$. This also explains why the theorem fails in genus 6.
 \end{remark} \medskip \par We can now complete the proof of theorem 3.7 for even spin curves: \par
 Let $(C,L)$ be a general even spin curve of genus $g \leq 7$, $g \neq 6$. Then there exists a Nikulin surface $(S, \vert H \vert, E)$ so that $\eta \cong E \otimes \mathcal O_C$.  Fix $\eta := L(-Z)$, where $Z \subset C$ is an odd theta characteristic. Let $P = \vert \mathcal I_Z(C) \vert$,  then $(P, Z, E)$ is a theta pencil on $S$, as in the example, and $m(P)$ passes through the moduli point of $(C,L)$. 
It remains to show  that $\mathcal S^+_g$ is covered by K3 rational curves when $g = 6$. Replacing the word 'K3 surface' by 'Nikulin surface', the proof is completely analogous to the one considered when $g = 10$ in the proof of proposition 3.8. We omit further details about this case. \end{proof}  \medskip \par 
 Finally we remark that the $\star$ principle, mentioned at the beginning of this section, holds true: $\mathcal S^{\pm}_g$ is uniruled  if and only if it is covered by K3 rational curves. This is a difference with respect to $\mathcal M_g$ as soon as $g \geq12$.
\subsection{Geometry of the moduli of spin curves in genus 8}
We have already seen that the geometry of curves of genus 8 is in many ways related to the study of other moduli spaces of curves of low genus. Such an experimental fact is confirmed for the moduli of spin curves. So there are good reasons to concentrate on the genus 8 geometry.  \par In this section we prove the unirationality of $\mathcal S^-_8$, which is representative of other unirationality results for $\mathcal S^-_g$. 
Then we will discuss the transition from negative to non negative Kodaira dimension in the even and odd cases. 
  \medskip \par 
\underline{\it Canonical curves of genus 8} \par
A crucial property is the realization of a general canonical curve $C$ of genus 8 as a linear section of the Pl\"ucker embedding 
$
G \subset \mathbf P^{14}
$
of the Grassmannian of lines of $\mathbf P^5$.   We recall something more on this, see \cite{Mu1} and \cite{Mu4}. \medskip \par 
Let $C$ be general of genus 8, so that $W^1_5(C)$ is finite. Then there exists a unique rank 2 vector bundle $E$ on $C$ such that $\det E \cong \omega_C$ and $h^0(E) = 6$.
Such an $E$ fits in an exact sequence
$$
0 \to A \to E \to  \omega_C(-A) \to 0,
$$
where $A \in W^1_5(C)$. This is uniquely defined by some $e \in \mathbf PExt^1(\omega_C(-A),A)$. As we will see, it turns out
that $E$ does not depend on the choice of $A$ in $W^1_5(C)$.  Let $V = H^0(E)^*$. Then $E$ defines a map
$$
f_E: C \to G \subset \mathbf P ( \wedge^2 V),
$$
where $G$ is the Pl\"ucker embedding of the Grassmannian $G(2,V)$. $f_E(x)$ is the point represented by the vector $\wedge^2 E^*_x$.  One can show the surjectivity of the natural  determinant map
$$
d: \wedge^2 H^0(E) \to H^0(\omega_C).
$$
Let $K^{\perp} \subset \wedge^2 V$ be the space orthogonal to $K := \mbox{Ker} \  d$, then $\mathbf PK^{\perp}$ is the linear span of $f_E(C)$.
Since $d$ is surjective, $f_E: C \to \mathbf PK^{\perp}$ is the canonical embedding.  Let us identify $C$ to $f_E(C)$. Then we conclude that
$$
C \subseteq \mathbf PK^{\perp} \cap G.
$$
\begin{theorem}[Mukai] Let $C$ be a smooth curve of genus $g$. The following condition are equivalent
\begin{enumerate} 
\item $W^2_7(C)$ is empty,
\item $C = \mathbf PK^{\perp} \cdot  G$.
\end{enumerate}
\end{theorem}  
\begin{remark} \ \rm Let $G^* \subset \mathbf P ( \wedge^2 V)^*$ be the dual Grassmannian. One expects that $\mathbf PK \cdot G^*$ is 0-dimensional of length 14. This is indeed the  degree of $G^*$ and it is also the length of $W^1_5(C)$. Actually there is a biregular map
 $$
w:  \mathbf PK \cdot G^* \to W^1_5(C).
$$
defined as follows: let $b_1 \wedge b_2$ be a decomposable, non zero vector defining a point $b \in \mathbf PK$. Then $b_1, b_2$ generate a line budle $B \subset E$. One can
compute that $B \in W^1_5(C)$. Then, by definition, $w(b) = B$. The inclusion $B \subset E$ induces an exact sequence
$$
0 \to B \to E \to \omega_C(-B) \to 0.
$$
In particular $E$ is independent on the choice of $B$ in $W^1_5(C)$, cfr. \cite{Mu4}.  \end{remark}  \par 
Let us fix from now on the following notations:
\begin{definition} \    \it
$\overline {\mathcal C}$ is the family of linearly normal, stable canonical curves of genus 8  in $G$. 
$\mathcal C$ denotes the open
subset parametrizing smooth curves. 
\end{definition} \par 
 We want to show that
\begin{theorem} $\mathcal S^-_8$ is unirational. \end{theorem} \par 
We outline the proof given in \cite{FV2}. To this purpose we consider the moduli map $m: \overline {\mathcal C} \to \overline {\mathcal M}_8$. Due to the mentioned results of Mukai, it is known that $m$
factors through the quotient map $\overline {\mathcal C} \to \overline {\mathcal C}/Aut \ G$ and a dominant  birational morphism $\overline m: \overline {\mathcal C}/Aut \ G \to \overline {\mathcal M}_8$.
Let us introduce some preliminary constructions. \medskip \par 
\underline {\it Moduli of 7-nodal elliptic curves} \par  
At first we consider the Zariski closure
$$
\mathbb B \subset \overline {\mathcal M}_8
$$
of the moduli of 7-nodal integral elliptic curves $N$. A general $N$ admits an embedding $N \subset G$ as a linear section of $G$ and hence as an element of $\overline {\mathcal C}$. This
just follows because a general K3 surface $S$ of genus 8 is a linear section of $G$. Hence $\vert \mathcal O_S(1) \vert$ contains a 1-dimensional family of integral, stable curves with moduli in $\mathbb B$. Notice also
that, for a general element $N$ in this family, the seven points of $ \Sing \ N$ are linearly independent, cfr. \cite{Ch}.
$\mathbb B$ is an integral projective variety of dimension 14.
 \begin{definition} $\mathcal N = m^{-1}(U)$, where $U$ is the open set of $\mathbb B$ parametrizing curves $N$ such that:
\begin{itemize} \it
\item[$\circ$] $N$ is an integral linear section of $G$,
\item[$\circ$]~the 7 points of $ \Sing \ N$ are linearly independent.
\end{itemize}
\end{definition} \par 
 Now we relate the odd spin moduli space $\overline {\mathcal S}^-_8$ to a $\mathbf P^7$-bundle over $U$. 
 Let $(C,L)$ be a general odd spin curve of genus 8. We can assume $h^0(L) = 1$ and $L = \mathcal O_C(Z)$, where $Z$ is a smooth effective divisor of degree $g-1$.
We can also assume that $C$ is canonically embedded as a linear section of $G$. Consider the universal singular locus
$$
\mathcal S = \lbrace (N, o) \in  \mathcal N \times G \ / \ o \in  \Sing \ N \rbrace.
$$
and its ideal sheaf  $\mathcal I$ in $\mathcal N \times G$.  On $\mathcal N$ we have the rank 8 vector bundle
$$
\mathcal E := p_{1*} \mathcal I \otimes p_2^* \mathcal O_G(1)
$$
where  $p_1, p_2$ are the projection maps of $\mathcal N \times G$.  The fibre of $\mathcal E^*$ at $N$ is $H^0(\mathcal I_Z(1))^*$. The projectivization of this vector space is naturally isomorphic to the following  family  of odd spin curves:  
$$
\mathbf PH^0(\mathcal I_Z(1))^* = \lbrace (C, \mathcal O_C(Z) \ / \ C \in \vert \mathcal I_Z(1) \vert^* \rbrace.
$$
The next properties are proved in \cite{FV2}: \medskip  \par 
\begin{enumerate} \it
\item $\mathcal E$ descends to a vector bundle on $U$,
\item $\mathcal N$ is smooth and irreducible,
\item  the natural morphism $\overline m: \mathcal N / Aut \ G \to \mathbb B$ is birational.
\end{enumerate} \medskip \par 
Therefore we conclude that
$$
\mathcal S^-_8 \cong \mathbb B \times \mathbf P^7.
$$
Furthermore we can show the following
\begin{theorem} $\mathbb B$ is unirational.
\end{theorem}
\begin{proof} Consider the correspondence
$$ I \subset  (\mathbf P^2)^7 \times (\mathbf P^{2*})^7 \times \vert \mathcal O_{\mathbf P^2}(3) \vert $$ 
parametrizing triples $(\underline o, \underline l, E)$ such that:
\par   
$\circ$ $\underline o = (o_1, \dots, o_7) \in (\mathbf P^2)^7$ and the points $o_1, \dots, o_7$ are in general position. \par
$\circ$ $\underline l = (l_1, \dots, l_7) \in (\mathbf P^2)^7$ and the lines $l_1, \dots, l_7$ are in general position. \par
$\circ$ $o_i$ belongs to the line $l_i$, $i = 1 \dots 7$. \par
$\circ$ $E$ is a smooth cubic passing through $o_i$ and $l_i$ is transversal to $E$. \par  
The correspondence $I$ is rational: see \cite{FV2} theorem 4.16 for the details of the proof. Furthermore $I$ is endowed with the rational map
$
\phi: I \to \mathbb B
$
sending $(\overline o, \overline l, E)$ to the moduli point of the 7-nodal curve $ \overline E$, obtained from $E$ by gluing together the two
points of $L_i \cap E$ different from $o_i$, $i = 1 \dots 7$. It is easy to see that $\phi$ is dominant. Hence $\mathbb B$ is unirational.
\end{proof}
\medskip \par 
The unirationality of $\mathcal S^-_8$ then follows from the theorem, because $\mathcal S^-_8$ is birational to $\mathbb B \times \mathbf P^7$.
\subsection{From uniruled to general type: the transition for $\overline {\mathcal S}^+_g$ and $\overline {\mathcal S}^-_g$ } \ \medskip \par
\underline {\it Moduli of even spin curves} \par
For the moduli of even spin curves the transition from the uniruledness to the general type case is represented by $\overline {\mathcal S}^+_8$,
whose Kodaira dimension is zero. We already remarked that $K_{\overline {\mathcal S}^+_8}$ is effective and now we want to 
see that any positive multiple of it is rigid. \par
To this purpose we introduce further geometric properties of the family of curvilinear sections of the Grassmannian $G$.
 We keep the previous notations. Let $X$ be a projective variety in $\mathbf P^n$. In what follows we denote by $\mathcal I_X$ 
the ideal sheaf of $X$ in the linear space spanned by $X$. \medskip \par 
\underline{\it Quadrics through $G$ and special divisors in $\overline {\mathcal M}_8$} \par
Special curves of genus 8 are characterized by special properties with respect to the  quadrics containing $G$. Let $C \in \mathcal C$, consider the restriction  
$$
r_C: H^0(\mathcal I_G(2)) \to H^0(\mathcal I_C(2)).
$$
It is useful to remark that $r_C$ is a map between spaces of the same dimension. Let $C \in \mathcal C$, we consider two natural divisorial conditions on $\mathcal C$: 
\medskip \par  
(1) \it $r_C$ is not an isomorphism, \par  \rm
(2) \it an element of $\vert \mathcal I_C(2) \vert$ is a rank 3 quadric. \rm
\medskip \par 
Both conditions (1) and (2) define effective divisors in $\mathcal M_8$ and hence in $\overline {\mathcal M}_8$, passing to their closure. They can be easily described as follows.
\medskip \par 
(1)  $C$ satisfies (1)  if and only if $C$ is not a linear section of $G$.  
By Mukai's theorem 3.11 this happens precisely when  $W^2_7(C)$ is non empty. \medskip \par  
(2) This condition is well known for any genus $g \geq 4$: it characterizes curves $C$ whose moduli point is in the divisor $\theta_{null}$. 
\medskip \par  
As above, $\pi$ will denote the divisor defined in (1)  or its class. $\pi^+$ will be its pull-back by the forgetful map $f^+: \overline{\mathcal S}^+_8 \to \overline {\mathcal M}_8$.  
\par  $\theta_{null}^+$ has been already defined, for every genus $g$, as the locus of moduli of spin curves $(C,L)$ such that $L$ is a theta null.
\medskip \par 
\underline{\it Geometry of the divisor $\theta_{null}^+$} \par 
In genus 8 the theta null condition (2) induces, via the quadrics containing $G$, an interesting covering family of rational curves
$$
R \subset \theta^+_{null}.
$$
The geometric reason for the existence of $R$ can be explained as follows. Let $(C,L)$ be a general  even spin curve of genus 8 such that $L$ is a theta null.
Since $\pi$ and $\theta_{null}$ are distinct irreducible divisors, we can assume that $$ C = G \cdot \langle C \rangle. $$
In $\langle C \rangle$ we have a unique quadric $q$ of rank three containing $C$. As is well known $q$  is characterized by the condition that its ruling of linear subspaces 
of maximal dimension cuts the pencil $\vert L \vert $ on $C$. Since $(C,L)$ defines a general point of $\theta_{null}^+$, we can assume that the restriction
$$
r_C: H^0(\mathcal I_G(2)) \to H^0(\mathcal I_C(2))
$$
is an isomorphism. Hence there exists exactly one $Q \in \vert \mathcal I_G(2) \vert$ such that $q = Q \cdot \langle G \rangle$.  One can show that $Q$ is smooth, provided
$(C,L)$ is sufficiently general, \cite{FV} proposition 6.1. Consider in the Pl\"ucker space $\mathbf P ( \wedge^2 V)$ the linear subspace
$$
\mathbf P_q = \cap \ Q_x, \ x \in \  \Sing \ q,
$$
where $Q_x$ denotes the tangent hyperplane to the quadric $Q$ at $x$. Note that $\langle C \rangle \subset Q_x$, since $x \in  \Sing \ q$.
Then, under the previous assumptions, one can show that:
\medskip \par 
(1)  \it Since $Q$ is smooth, it follows $ \ \dim \ \mathbf P_q = 9$. Since $x \in  \Sing \ q$. \rm \par 
(2)  \it Since each $Q_x$ contains $\langle C \rangle$, $\langle C \rangle$ is a hyperplane in $\mathbf P_q$. \rm \par 
(3) \it $\tilde q := Q \cdot \mathbf P_q$ is a quadric of rank 4 and $ \Sing \ \tilde q =  \Sing \ q$. \rm \par 
(4) \it $S := G \cdot \mathbf P_q$ is a smooth K3 surface contained in $\tilde q$. \rm \medskip \par 
See \cite{FV} section 6. The rulings of $\tilde q$ cut on $S$ two
elliptic pencils, say $\vert F_1 \vert$ and $\vert F_2 \vert$.  They
 restrict to the same ruling of $q$. This just means that
$$
\mathcal O_C(F_1) \cong L \cong \mathcal O_C(F_2).
$$
 One has $\mathcal O_S(F_1 + F_2) \cong \mathcal O_S(1)$ and $\vert F_1 \vert$, $\vert F_2 \vert$ define a product map
$$
\phi: S \to \mathbf P^1 \times \mathbf P^1
$$
of degree $7 = F_1 \cdot F_2$. In particular $\phi / C$ is the map defined by $\vert L \vert$ and $C$ belongs to $I$, where
$$
I :=   \vert \phi^* \mathcal O_{\mathbf P^1 \times \mathbf P^1}(1,1) \vert.
$$
 Moreover each smooth $D \in I $ is a canonical curve contained in a rank 3 quadric, namely the quadric $\tilde q \  \cap   \langle D \rangle$.
As in the case of $C$, it turns out that $\phi/D$ is defined by a theta null $L_D \cong \mathcal O_D(F_1) \cong \mathcal O_D(F_2)$. \par
The outcome of this construction is clear: let $P \subset I$ be a general pencil containing $C$, then the base locus of $P$ is
$Z + Z'$, for some $Z, Z' \in \vert L \vert$. Hence the triple $(P, Z, \mathcal O_S(F_i))$ is a theta pencil and $$ R = m(P)$$ is a rational curve in 
$\theta_{null}^+$, passing through the moduli point $x$ of $(C,L)$. \medskip \par 
The next theorem summarizes the property of genus 8 curves with a theta null we have described. 
\begin{theorem} Let $C$ be a general smooth integral curve of genus 8. The following conditions are equivalent:
\begin{itemize}
\item[$\circ$] there exists a theta null $L$ on $C$,
\item[$\circ$] $C \in \vert F_1 + F_2 \vert$, where $\vert F_1 \vert$, $\vert F_2 \vert$ are distinct pencils of elliptic curves on a K3
surface $S$ and $ L \cong \mathcal O_C(F_1) \cong \mathcal O_C(F_2). $
\end{itemize}
\end{theorem} \medskip \par 
 \underline {\it The Kodaira dimension of $\overline {\mathcal S}^+_8$ is zero} \par
Let us consider a general rational curve $R$  of the previous family of rational curves covering $\theta_{null}^+$. This means that $R = m(P)$, 
where $(P, Z; \mathcal O_S(F_i))$ is a theta pencil as above. We begin our discussion on the Kodaira dimension of $\overline {\mathcal S}^+_8$ 
with a result which is sharp and crucial.
\begin{proposition} $R \cdot \theta_{null}^+ = -1$. \end{proposition} \par 
The result follows counting with appropriate multiplicity the elements of a general $P \subset I$ which are singular.  In particular two elements of $P$ are 
of type $F_1 \cup F_2$, where $F_1, F_2$ are integral elliptic curves intersecting transversally in 7 points, see \cite{FV} lemma 6.6.  
Since $R$ moves in a family of irreducible curves which cover $\theta_{null}^+$, it follows that
\begin{corollary} $m\theta_{null}^+$ is rigid for each $m \geq 0$. \end{corollary}
\begin{remark} \ \rm The rigidity of $\theta_{null}^+$ can be proved for other moduli $\overline {\mathcal S}^+_g$ of even spin curves, namely for
$g \leq 9$. The proof  relies on a different family of rational curves $R'$ covering the theta null divisor of $\overline {\mathcal S}^+_g$, 
$g \leq 9$. One proves that $R' \cdot \theta_{null}^+ = - 2$, so that the rigidity follows. The condition $R \cdot \theta_{null} = -1$ is of
course sharp in genus 8. It is not clear wether $\theta^+_{null}$ is rigid for any $g$.
\end{remark} \par 
An analogous proposition holds true for the effective divisor $\pi^+$. This parametrizes nodal plane septics of geometric genus 8. Let 
$P' \subset \vert \mathcal O_{\mathbf P^2}(7) \vert$ be a general pencil of nodal septics of such a genus. We have again the moduli
map $m': R' \to \overline {\mathcal M}_8$ and we can consider the curve
$$
R' := f^{+*}(m(P')) \subset \pi^+.
$$
One computes that:
\begin{proposition} $R' \cdot \pi^+ < 0$. \end{proposition} \par 
$R'$ moves in a family of irreducible curves covering $\pi^+$. Hence it follows:
\begin{corollary} $\pi^+$ is rigid in $\overline {\mathcal S}^+_8$. \end{corollary} \par 
Finally we can summarize the proof of the following:
\begin{theorem} The Kodaira dimension of $\overline {\mathcal S}^+_8$ is zero. \end{theorem}
\begin{proof} See \cite{FV}, proof of theorem 0.1. We know that the canonical class is effective and that
$$
K_{\overline {\mathcal S}^+_8} \sim t\theta_{null}^+ \ + \  p \pi^+ \  + \ \sum_{i = 1 \dots 4} a_i A^+_i \ + \ b_i B^+_i,
$$
where $t, p, a_i, b_i$ are positive rational numbers. $A_i^+$, $B_i^+$ are the boundary divisors
already considered. We know that there exist covering families of irreducible curves $R \subset \theta_{null}^+$ and $R' \subset \pi^+$ such that 
$
R \cdot \theta_{null}^+ < 0 \ , \ R' \cdot \pi^+ < 0.
$
Moreover it is not difficult to show that
$$
R \cdot \pi^+  = R \cdot A_i^+ = R \cdot B_i^+ = 0 \ , \ R' \cdot \theta_{null}^+ = R' \cdot A_i^+ = R' \cdot B_i^+ = 0.
$$
This implies that $m \theta_{null}^+$ and $m \pi^+$ are fixed components of the $m$-canonical linear system for each $m \geq 1$. To see this consider a general $x \in \theta_{null}^+$. Then $x$ belongs to
a curve $R$ of the family considered above. Since $R \cdot m K_{\overline {S}^+_8} < 0$, it follows that $x$ is in  the base locus of $\vert mK_{\overline {\mathcal S}^+_8}vert$. Hence $m \theta_{null}^+$
is a fixed component of $\vert mK_{\overline {\mathcal S}^+_8} \vert$. The same argument works for $m\pi^+$. It follows that the moving part of the $m$-canonical linear system is contained in $\vert m \sum (a_iA_i^+ + b_iB_i^+) \vert$, for $m \geq 1$.
But it turns out that the latter is 0-dimensional too. Hence $ \ \dim \ \vert mK_{\overline {S}^+_8} \vert = 0$, for $m \geq 1$. \end{proof} \medskip \par 
\underline{\it Moduli of odd spin curves} \par
The transition from uniruledness to general type has no intermediate case for the moduli of odd spin curves of genus $g$. The last case where $\overline {\mathcal S}^-_g$ is uniruled is for $g = 11$.
One would like to understand better some geometric reasons for the change of the Kodaira dimension from genus 11 to genus 12.  We conclude this section with a kind of free speculation on this question. \par 
Let $X$ be a an integral projective variety, one could define the \it degree of uniruledness of $X$ \rm as
$$
u(X) :=  \mbox{max} \ \lbrace d \in \mathbb Z \ / \ \text {\it it exists a generically finite map $f: Y \times \mathbf P^d \to X $} \rbrace.
$$
$Y$ is assumed to be integral. Of course $Y$ is not uniruled if $u(X) = 0$. Mumford's conjecture on varieties of Kodaira dimension $-\infty$ says that in this case $u(X)$ is strictly positive . Notice also that  $X$ is unirational  if and only if $u(X) =  \ \dim \ X$. \par Now let $u(g) = u(\overline {\mathcal S}^-_g)$. We have seen that the unirationality of $\mathcal S^-_g$ is known for $g \leq 8$ so that $u(g) = 3g-3$ in this case. For $g = 9$ the unirationality of $\mathcal S^-_9$ seems plausible, though no complete proof is appeared until now. What about $u(10)$ and $u(11)$? \par  At least in genus 11 the situation could be quite different. Let us
see a possible reason. 
Consider the moduli space $\mathcal F_g$ of K3 surfaces $(S, \mathcal O_S(C))$ of genus $g$ and then the universal Hilbert scheme of points
$$
q: \mathcal F_g[g-1] \to \mathcal F_g
$$
with fibre at the moduli point of $(S, \mathcal O_S(C))$ the Hilbert scheme $S[g-1]$ of 0-dimensional subscheme of length $g-1$. Since $C^2 = 2g - 2$ we can define a natural involution
$$
i: \mathcal F_g[g-1] \to \mathcal F_g[g-1].
$$
Indeed let $x$ be the moduli point of $(S, \mathcal O_S(C), Z)$, where $Z \in S[g-1]$. Then the base locus of $\vert \mathcal I_{Z/S}(C) \vert$ is $Z + \overline Z$ where $\overline Z \in S[g-1]$. By
definition, $i(x)$ os the moduli point of $(S, \mathcal O_S(C), \overline Z)$. Let
$$
\mathcal T_g \subset \mathcal F_g[g-1]
$$
be the locus of fixed points of $i$. For a general pair $(S, \mathcal O_S(C))$ the fibre of $q / \mathcal T_g$ at the moduli point of $(S, \mathcal O_S(C))$ is  the family of the 0-dimensional schemes $Z$ such
that $(P_Z, Z, \mathcal O_S)$ is an theta pencil of genus $g$, where $$ P_Z := \vert \mathcal I_{Z/S}(C) \vert. $$
Equivalently this locus is the closure of the family of all $Z \subset S$ such that $Z$ embeds as an odd theta characteristic in a smooth element of $\vert C \vert$.   
On the other hand let us consider the standard projective bundle
$$
p: \mathbb P \to \mathcal K_{11}
$$
with fibre $\vert C \vert$ at the moduli point of $(S, \mathcal O_S(C))$.  Then the pull-back $(q / \mathcal T_g)^* \mathbb P$ contains a natural $\mathbf P^1$-bundle
$
P \subset \mathbf P,
$
with fibre $P_Z$ at the moduli point of $(S, \mathcal O_S(C), Z)$. Note that the diagram
$$
\begin{CD}
{\mathcal T_g} @<<< P @>{m^-}>> {\mathcal S^-_{11}} \\
@VVV @VV{f'^-}V @VV{f^-}V \\
{\mathcal K_{11} }@<<< {\mathbb P} @>m>> {\mathcal M_{11}} \\
\end{CD}
$$
is commutative and that the vertical arrows $f_P$ and $f^-$ are just the obvious forgetful map. In particular they have the same degree, which is the number of odd theta characteristics on curve of genus 11.
Since the moduli map $m$ is birational, the same is true for the moduli map $m^-$. This shows that
\begin{theorem} $\overline {\mathcal S}^-_{11}$ is ruled and  birational to $\mathcal T_{11} \times \mathbf P^1$. \end{theorem} \par 
$\mathcal T_g$ is a very interesting variety whose Kodaira dimension does not seem to be known. Knowing it is non negative in genus 11 could place $\mathcal S^-_{11}$, though uniruled,
in an intermediate position between uniruledness and general type.

\section {Prym moduli spaces}
\subsection{Prym pairs} 
Continuing our review of remarkable multisections of $ \Pic_{d,g}$, we naturally  come to the degree zero universal Picard variety
$$
p:  \Pic_{0,g} \to \mathcal M_g.
$$
$p$ is endowed with the zero section $C \to (C, \mathcal O_C)$ and the multiplication by $n$ map  $\mu_n:  \ \Pic_{0,g} \to  \ \Pic_{0,g}$. 
We can define in $ \Pic_{0,g}$ the following locus:
\begin{definition} $\mathcal R_{g,n}$ is the moduli space of pairs $(C,L)$ such that
  $C$ is a smooth, integral curve of genus $g$ and $L$ is a primitive $n$-root of $\mathcal O_C$.
 
\end{definition} \par 
In particular $\mathcal R_{g,n}$ is a component of the inverse image of $\mu_n$. We recall some constructions related to a pair $(C,L)$. Consider the following sets.
\begin{enumerate}
\item $S_1(C)$: the set of  primitive $n$-roots of $\mathcal O_C$;
\item $S_2(C)$: the set of pairs $(\pi, i)$ such that  $ \pi: \tilde C \to C $ is a cyclic \'etale cover of degree $n$ and $i: \tilde C \to \tilde C$ generates the Galois group of $\pi$;
\item $S_3(C)$: the set  of pairs $( \tilde C, i)$ such that $i: \tilde C \to \tilde C$ is an automorphism acting freely on a smooth, integral curve $\tilde C$ and $\tilde C / \langle i \rangle \cong C$.
\end{enumerate}
\begin{proposition} There exist natural bijective maps between the sets $S_1(C)$, $S_2(C)$ and $S_3(C)$.
\end{proposition}
\begin{proof} Let $L \in S_1(C)$ and $E := \bigoplus_{i = 0 \dots n} L^{\otimes i}$. Consider the curve
$$
\tilde C := \lbrace (x; 1, v, v^{\otimes 2}, \dots, v^{\otimes n-1}) \ / \ v \in L_x \ , \ v^{\otimes n} = 1 \rbrace \subset \mathbf PE.
$$
Then $\tilde C$ is smooth, connected and the projection $p: \mathbf PE \to C$ restricts to an \'etale cyclic cover $\pi := p/\tilde C$. $\tilde C$ is endowed
with the obvious automorphism $i$ induced by tensor product with $v$. This defines a map $ u_1: S_1(C) \to S_2(C) $ such that $u_1(L) := (\pi,i)$. A second map  $u_2: S_2(C) \to S_3(C)$ is defined as follows: $u_2(\pi,i) = (\tilde C, i)$.
Finally a pair $(\tilde C,i) \in S_3(C)$ defines an \'etale cyclic cover $ \pi: \tilde C \to C = \tilde C / \langle i \rangle. $ In particular  $\pi_* \mathcal O_{\tilde C}$ splits as a vector bundle $E$ as above and its summand  $L$ is uniquely defined by $i$. This
defines a map $u_3: S_3(C) \to S_1(C)$ such that $u_3(\tilde C,i) = L$. It is easy to see that  these maps are bijective.
\end{proof} \par 
From now on we will keep the notations
$$
\pi: \tilde C \to C \ \ , \ \ i: \tilde C \to \tilde C
$$
for the maps constructed from a pair  $(C,L)$ as above. 
 $\mathcal R_{g,n}$ is a very natural multisection and it deserves to be studied in its full generality for every $n$. Actually only the case $n = 2$ has been object of a detailed study 
 since the seventies of the last century. As exceptions see however \cite{CCC} for the compactification of $\mathcal R_{g,n}$ and \cite{LO} for the Prym map associated to $\mathcal R_{g,n}$.
 See also \cite{BC} and \cite{BaV} for the rationality problem when $n = 3$ and $g = 3,4$. From now on we will assume $n =  2$.
\begin{definition} A pair $(C,L)$ as above is called a Prym pair. \end{definition} \par 
For a Prym pair $\pi: \tilde C \to C$ is an \'etale double covering and $i: \tilde C \to \tilde C$ is a fixed point free involution. For $\mathcal R_{g,2}$ we will use the standard notation
$$
\mathcal R_g.
$$
The space $\mathcal R_g$ is well known as the \it Prym moduli space. \rm 
The special attention owed on Prym pairs is related to the notion of Prym variety, a principally polarized abelian variety of dimension $g-1$ which is associated to $(C,L)$ when $n = 2$. This bring us  to  the theory of principaly polarized abelian varieties and to the parametrizations of their moduli spaces.  \par Postponing further comments, we recall the definition of Prym variety. Let $(C,L)$ be a Prym pair consider the Norm map
$$
Nm_0:  \ \Pic^0(\tilde C) \to  \ \Pic^0(C),
$$
sending $\mathcal O_{\tilde C}(d)$ to $\mathcal O_C(\pi_*d)$. $Nm_0$ is a surjective morphism of abelian varieties. It turns out that the principal polarization of $  \Pic^0( \tilde C)$ restricts to twice a principal polarization $\Xi$ on the connected component $A$ of $\mbox{Ker} \  Nm_0$, see \cite{M8}. The pair $(A, \Xi)$ is therefore a principally polarized abelian variety:
\begin{definition} The Prym variety of $(C,L)$ is the pair $(A, \Xi)$.\end{definition}  
 \medskip \par 
What is known on the existence of rational parametrizations of $\mathcal R_g$ 
can be easily summarized as follows:
\begin{theorem}
 $\mathcal R_g$ is unirational for $g \leq 7$ and rational for $g \leq 4$.
\end{theorem}
\par 
We  discuss at first the rationality and then the unirationality results. Once more this gives the opportunity to uncover beautiful  constructions and  see
their common geometric links.
\subsection{ Rationality of  Prym moduli spaces of hyperelliptic curves} 
We start with the \it hyperelliptic locus in $\mathcal R_g$. \rm This has many irreducible components. The rationality of all of them, possibly but one, can be deduced from Katsylo's theorem on the rationality of moduli of binary forms, \cite{K1}. \par 
Let $C$ be  hyperelliptic of genus $g \geq 2$ and let $u: C \to \mathbf P^1$ be its associatd double cover. Then the ramification divisor of $u$ is supported on the set 
$$
W = \lbrace w_1, \ \dots, \ w_{2g+2} \rbrace
$$
of the Weierstrass points of $C$. Let $H$ be the hyperelliptic line bundle of $C$ and $E_t$ be the family of effective divisors of cardinality $2t$,  where $1 \leq t \leq [\frac {g+1}2]$, which are
supported on $2t$ \it distinct \rm Weierstrass points. For $e \in E_t$ the line bundle $H^{\otimes t}(-e)$ is a square root of $\mathcal O_C$, that is an element of
$  \Pic^0_2(C)$. Let
$$
\beta_t: E_t \to   \Pic^0_2 (C) - \lbrace \mathcal O_C \rbrace
$$
be the map sending $e \in E_t$ to the line bundle $H^{\otimes t}(-e)$ and let $$B_t := \beta_t(E_t).$$ The next lemma is standard:
 \begin{lemma} \ \par  (1) The map  $\beta_t$ is injective for $t \leq [\frac g2]$. \par (2)  For $t = \frac {g+1}2$ the map $\beta_t$ is 2:1 over its image.
 \par (3) $  \Pic^0_2(C) - \lbrace \mathcal O_C \rbrace = \bigcup_{ 1 \leq t \leq [\frac{g+1}2]} B_t$.
 \end{lemma} 
  \begin{proof} (1) Let $\beta_t(e) = \beta_t(f)$ for some $e, f \in E_t$, then $\mathcal O_C(f) \cong \mathcal O_C (e)$. 
  Now $e$, $f$ are isolated, since $deg \ e \leq g$. Hence $e = f$.  (2) Let $e = w_1 + \dots + w_{g+1}$, it is easy to see that $\beta_t^{-1}(\beta_t(e)) = \lbrace e , w-e \rbrace$, where $w = w_1 + \dots + w_{2g+2}$. (3) This is
  an easy well known property as well.
   \end{proof} \par
 \begin{definition}  $\mathcal {RH}^t_g$ is the locus in $\mathcal R_g$ of the moduli points of pairs $(C,L)$ such that $L \cong H^{\otimes t}(-e)$ for some $e \in E_t$.
 \end{definition} \par 
Let $\mathcal {RH}_g \subset \mathcal R_g$ be the hyperelliptic locus, parametrizing pairs $(C,L)$ such that $C$ is hyperelliptic. It is clear that
$$
\mathcal {HR}_g = \bigcup_{1 	\leq t \leq [\frac{g+1}2]} \mathcal {RH}^t_g.
$$
  \begin{theorem} $\mathcal {RH}^t_g$ is rational for $t \leq [\frac g2]$.
 \end{theorem}
\begin{proof} Let $(C,L)$ be a Prym pair defining a general $x \in \mathcal R \mathcal H^t_g$. Assume $t \leq [\frac g2]$. Then, by the previous lemma, there exists a unique $e \in E_t$ such that  $L \cong H^{\otimes t}(-e)$. Therefore  we can uniquely associate to $(C,L)$ the set  $W - Supp \ e$. Consider the 2.1 cover $h: D \to \mathbf P^1$, branched on $W - \\Supp \ e$, then $D$ is hyperelliptic of genus $g - t$. Let $\mathcal H_p$ be the moduli
space of hyperelliptic curves of genus $p$. Then we have a dominant rational map
$$
\phi: \mathcal {HR}^t_g \to \mathcal H_{g-t},
$$
induced by the assignement $(C,L) \mapsto D$. As is well known $\mathcal H_{g-t}$ is birational to the quotient $\mathbf P^1(2g - 2t + 2) / PGL(2)$, where $\mathbf P^1(n)$ denotes
the $n$-th symmetric product of $\mathbf P^1$. This quotient is rational by  Katsilo's theorem, \cite{K1}.  Moreover let $\mathbb P := (\mathbf P^1(2g - 2 - 2t) \times \mathbf P^1(2t)) / PGL(2)$. It
is well known that then $\mathbb P$ is a $\mathbf P^{2t}$-bundle over an open set $U \subset \mathcal H_{g-t}$. Hence $\mathbb P$ is rational. Let $x \in \mathbb P$ be the orbit of $(p,q) \in \mathbf P^1(2g + 2 - 2t) \times \mathbf P^1(2t)$. $x$ uniquely defines a pair $(C,L)$ such that $h: C \to \mathbf P^1$ is the 2:1 cover  branched on $p + q$ and $L \cong H^{\otimes t}(-e)$, where $e := h^*q$. Hence we have a rational map
$$
\psi: \mathbb P \to \mathcal R \mathcal H^t_g
$$
sending $x$ to the moduli point of $(C,L)$.    An inverse to $\psi$ exists: let $y \in \mathcal {HR}^t_g$ be the moduli point of $(C,L)$. Keeping our notations, this uniquely defines $e \in E$ such that $L \cong H^{\otimes t}(-e)$. This isomorphism uniquely defines the orbit $x \in \mathbb P$ of the pair $(p,q)$, where $p+q$ is the branch divisor of $h$ and $q = h_*e$. Therefore $\psi$ is generically injective. It is also birational because 
$\mathbb P$ and $\mathcal {HR}^t_g$ are integral of the same dimension. Then $\mathcal {HR}^t_g$ is rational for $t \leq [\frac g2]$.    \end{proof} \par  
In particular  it follows:
\begin{corollary} $\mathcal R_2$ is rational for $g = 2$.
\end{corollary}
\begin{proof} Just observe that $\mathcal R_2 = \mathcal H\mathcal R^1_2$. \end{proof}
 
 \medskip \par 
\underline {\it $\mathcal R_2$ and the Segre primal} \rm \par 
A beautiful proof of the rationality of $\mathcal R_2$ is due to Dolgachev and relies on Segre cubic primal and its associated geometry.  We want to introduce  this
proof and describe the main geometric constructions behind it. See \cite{D1} and \cite{VdG} for more details.   
The Segre primal is the cubic threefold $V$ defined in $\mathbf P^5$ by the equations
$$
y_1 + \dots + y_5 = y_1^3 + \dots + y_5^3 = 0.
$$
Clearly we have
$$
V \subset \mathbf P^4 := \lbrace y_1 + \dots + y_5 = 0 \rbrace.
$$
$V$ is a nodal cubic threefold with the maximal number of nodes, namely 10, and contains 15 planes. It was discovered
by Corrado Segre in \cite{Sg}. $V$ is very related to the moduli space $\overline {\mathcal M}_2^{(2)}$ of genus 2 curves
with a level 2 structure.  It is indeed the strict dual of the natural embedding
$$
\mathbb M \subset \mathbf P^{4*}
$$
of $\overline {\mathcal M}^{(2)}_2$ as a Siegel modular threefold of degree 4. Of course there is a lot of beautiful classical geometry 
behind this relation, see \cite{VdG}: \medskip \par 
(1) $ \Sing \ \mathbb M$  is the connected union of 15 lines, each of multiplicity 2. This curve
has 35 nodes which are triple points for $\mathbb M$.
\medskip \par 
(2) A point $o \in \mathbb M -  \Sing \ \mathbb M$, represents a genus 2 curve $C$ with a level 2 structure.  Let  $K_o := T_o \cap \mathbb M$,
$T_o$ being the tangent hyperplane at $o$. Then
$$
K_o =   \Pic^0(C) / \langle-1\rangle.
$$ 
$K_o$ is a Kummer quartic surface.
 \par 
(3)  Let $q:   \Pic^0(C) \to K_o$ be the quotient map, then $$q^{-1}( \Sing \ \mathbb M) =   \Pic^0_2(C) - \lbrace \mathcal O_C \rbrace. $$
\par 
(4) The forgetful map
$
f: \mathbb M \to \overline {\mathcal M}_2
$
is the quotient map for the action of $S_6 \cong Sp(4, \mathbb Z_2)$ on $\mathbb M$.
It is  induced by the permutations of $(y_1, \dots, y_6)$.
 
  \medskip \par  \underline{\it The rationality of $\mathcal R_2$} \par  Let $T$ be the set of 15 double lines of $\mathbb M$. It follows from (3) that each $o \in \mathbb M -  \Sing \ \mathbb M$ is endowed with a bijection
$
T \longrightarrow   \Pic^0_2(C) - \lbrace \mathcal O_C \rbrace,
$
defined by the assignement $l \mapsto q^{-1}(l)$. We can now deduce the rationality of $\mathcal R_2$: \medskip \par  Fix $l \in T$ and consider the map
$$
\phi: \mathbb M \to \mathcal R_2
$$
sending $o$ to the moduli point of $(C,L)$, where $L = q^{-1}(l)$.  Observe that $S_6$ acts on $T$ and that $\phi$ is the quotient map with respect to the stabilizer of $l$. 
It is known that this is conjugate to the subgroup $S_4 \times S_2$ of $S_6$. Therefore we conclude that: 
$$ 
\text { \it ${\mathcal R}_2$ is birational to $\mathbb M / S_4 \times S_2$. }
$$
The last step is the rationality of $\mathbb M / S_4 \times S_2$, proved in \cite{D1}.   
\bigskip \par 
 \subsection { Rationality of $\mathcal R_3$} \ \rm \par 
Various proofs of the rationality of $\mathcal R_3$ appear to be known since the eighties of the last century.  The subject was indeed considered by several authors, 
see \cite{D1} for some historical account.  In particular, the first complete paper containing this result  is due to Katsylo, \cite{K2}.  
The rationality of $\mathcal R_3$ stems from the vein of classical geometry we considered in the last part of the previous section. We continue
in this vein, keeping the same notations. \medskip \par 
\underline {\it $\mathcal R_3$ and the Segre primal} \par 
Consider the Siegel modular quartic threefold $\mathbb M = \overline {\mathcal M}^{(2)}_2$ and the sheaf
$$
 \Omega^1_{\mathbb M} \oplus \mathcal O_{\mathbb M}.
$$
The fibre of  $\mathbf P (\Omega^1_{\mathbb M} \oplus \mathcal O_{\mathbb M})$ at any smooth $o \in \mathbb M$ is the dual of the tangent hyperplane to $\mathbb M$ at $o$. In other words we have
an obvious identification
$$
\mathbf P (\Omega^1_{\mathbb M} \oplus \mathcal O_{\mathbb M})_{_o} = \vert \mathcal O_{K_o}(1) \vert,
$$
where $K_o =   \Pic^0(C) / \langle-1\rangle$ is the Kummer quartic surface considered in the previous section. Let $f: \mathbb M \to \overline {\mathcal M}_2$ be the forgetful map. $S_6$ acts linearly and almost freely on the hypersurface $\mathbb M$. Hence it acts as well on the fibres of $\mathbf P (\Omega^1_{\mathbb M} \oplus \mathcal O_{\mathbb M})$ and descends to a $\mathbf P^3$-bundle 
$
u: \mathbb K \to U
$
over an open set $U$ of $\mathcal M_2$. Let $p = f(o) \in \mathcal M_2$, we point out that there exists a well defined surface which is naturally associated to $\mathbb K_p$. This is the dual 
$$
D_p \subset \mathbb K_p
$$
of $K_o$. It is well known that $D_p$ is the union of 16 planes, corresponding to the 16 nodes of $K_o$, and of the strict dual surface
$$
\hat K_p \subset \mathbb K_p
$$
of $K_o$. This is a Kummer quartic surface again. Note that any $$ H \in \mathbb K_p - D_p$$ is  a smooth plane section of $K_o$. Moreover $H$ is endowed with a non split \'etale double covering $\pi: \tilde H \to H$. Indeed let $C$ be the moduli point of $p$, then $\pi$ is induced by the quotient map $  \Pic^0(C) \to K_o$. Let $L_H \in   \Pic^0_2(C)$ be the line bundle defining $\pi$,
we have reconstructed from $(H,p)$ a Prym pair $(H, L_H)$ and hence a point of $\mathcal R_3$. This defines a rational map
$$
\phi: \mathbb K \to \mathcal R_3.
$$
Its description $\phi$ follows from \cite{Ve2}:  $\phi$ fits in the commutative diagram
$$
\begin{CD}
{\mathbb K} @>{\phi}> >{\mathcal R_3} \\
@VV{u}V @VV{p_3}V \\
{\mathcal M_2} @>j>> {\mathcal A_2,} \\
\end{CD}
$$
where $p_3: \mathcal R_3 \to \mathcal A_2$ is the Prym map and $j: \mathcal M_2 \to \mathcal A_2$ is the Torelli map. It turns out that the fibre of the compactified Prym map at $j(p)$ is a suitable modification of the
quotient space
$$
\mathbb K_p / \mathbb Z_2^4.
$$
Here the group $\mathbb Z_2^4$ acts on $\mathbb K_p = \mathbf P^3$ as the group of projective automorphisms induced by the group of translations on $  \Pic^0(C)$ by 2-torsion elements. It is a well known property, from classical theta functions theory, that 
$$
\mathbf P^3 / \mathbb Z_2^4 = \mathbb M^*.
$$
$\mathbb M^*$ denotes the Segre cubic primal, that is, the strict dual of $\mathbb M$. This is shown in Hudson's book \cite{Hu}, where $\mathbb P^3 / \mathbb Z_2^4$ is explicitely described. This is the image of the
map $q: \mathbf P^3 \to \mathbf P^4$ defined by the linear system of the $\mathbb Z_2^4$-invariant quartic surfaces
$$
a(x_1^2x_4^2 + x_2^2x_3^2) + b(x_2^2x_4^2 + x_1^2x_3^2) + c(x_1^2x_4^2 + x_2^2x_3^2) + 2dx_1x_2x_3x_4 + e\sum x_i^4 = 0.
$$
Its cubic equation is
$$
4e^3-a^2e^-b^2e - c^2e + abc + d^2e = 0.
$$
It is shown in \cite{Ve2} that $q = \phi / \mathbb K_p$. Then we can conclude that
\begin{theorem} $\mathcal R_3$ is birational to a fibration over $\mathcal M_2$ with fibre $\mathbb M^*$.
\end{theorem} 
\underline {\it $\mathcal R_3$ and the Coble-Roth map} \par 
A proof of the rationality of $\mathcal R_3$, due to Dolgachev, relies on the classical Coble-Roth map. We want to introduce it and then summarize such a proof. Let
$$
\mathcal R^b_4
$$
be the moduli space of Prym pairs $(F,L_F)$ such that $F$ is a bielliptic curve of genus 4 and $L_F \in \sigma^*   \Pic^0_2(E)$, where $ \sigma: F \to E $ is a bielliptic map. For proving the rationality of 
$\mathcal R_3$, the main tool used is a rational map
$$
\phi: \mathcal R_3 \to \mathcal R^b_4.
$$
This map is classical, it is defined in \cite{D1} as the \it Coble-Roth map. \rm  The rationality of $\mathcal R_3$ is then a consequence of the following steps:
\begin{theorem} \ \begin{enumerate}
\item $\phi$ is birational,
\item $\mathcal R^b_4$ is rational.
\end{enumerate} 
\end{theorem}\par 
This theorem is due to Dolgachev, \cite{D1}. We describe here the proof of (1). This is an almost immediate consequence of the existence of $\phi$. \par 
Let $(C,L)$ be a \it general \rm Prym pair,  defining a point of $\mathcal R_3$. Since $L^{\otimes 2}$ is trivial,   the assignement
$N \mapsto L \otimes N$ defines  a fixed points free involution
$$
t_L:   \Pic^2(C) \to  \Pic^2(C).
$$
In $ \Pic^2(C)$ consider the theta divisor $W^0_2(C) = C^{(2)}$ and the curve
$$
\tilde F := \lbrace x+y \in C^{(2)} \ / \ h^0(\omega_C \otimes L(-x-y) \geq 1 \rbrace.
$$ 
$\tilde F$ is obtained from the 4-gonal pencil $\vert \omega_C \otimes L \vert$ via Recillas' construction, \cite{Re}.
It turns out that $\tilde F$ is a smooth, integral curve of genus 7. Moreover it is easy to check that $\tilde F = W^0_2 \cap t_L(W^0_2)$.
Hence $t_L$ induces a fixed points free involution $i_{\tilde F}$ on $\tilde F$. Then its  quotient map is an \'etale double covering 
$$
\pi_F: \tilde F \to F,
$$
induced by some $L_F$ in $ \Pic^0_2(F)$. Recillas' construction also implies that:
\begin{proposition} The Prym variety of $(F,L_F)$ is $ \Pic^0(C)$. \end{proposition} \par 
On $\tilde F$ we have a second involution
$
j: \tilde F \to \tilde F.
$
By definition this associates to $x+y \in \tilde F$ the unique  $z+t$ such that $x + y + z + t$ is a canonical divisor. Hence $x+y$ is fixed by $j$   if and only if $\mathcal O_C(x+y)$ and $L(x+y)$ are odd theta characteristics. The formula for the number of these odd thetas gives 12. Let $$ \tilde \sigma: \tilde F \to \tilde E := \tilde F / \langle j \rangle$$  be the quotient map. Then, by Hurwitz formula, $\tilde E$ is elliptic.    It is easy to see that $i_{\tilde F}$ and $j$ commute. Moreover
$i_{\tilde F}$ induces a fixed points free involution $i_{\tilde E}$ on $\tilde E$, so that $E = \tilde E/\langle i_{\tilde E}\rangle$ is elliptic. Hence we have the commutative diagram
$$
\begin{CD}
{\tilde F} @>{\pi_F}>> F \\
@VV{\tilde \sigma}V @VV{\sigma}V \\
{\tilde E} @>{\pi_E}>> E \\
\end{CD}
$$
where the vertical arrows are the quotient maps of $i_{\tilde F}$ and $i_{\tilde E}$ respectively. It follows from the diagram that $L_F \in \sigma^*  \Pic^0_2(E)$. Hence
the Prym pair $(F, L_F)$ defines a point of the Prym moduli space of bielliptic curves $\mathcal R^b_4$.  
\begin{definition} The Coble-Roth map $\phi: \mathcal R_3 \to \mathcal R^b_4$ is the map sending the moduli point of $(C,L)$ to the moduli
point of $(F,L_F)$.
\end{definition} \par 
Using the diagram, we can now define a rational map 
$$
\psi: \mathcal R^b_4 \to \mathcal R_3
$$
inverse to $\phi$. Starting from $(F, L_F)$ we can reconstruct at first the previous commutative diagram. Moreover $ \Pic^0(C)$ is the Prym of $(F,L_F)$ and  sits in $ \Pic^0(\tilde F)$ as the connected
component of zero of the Kernel of the Norm map $\pi_{F*}:  \Pic^0(\tilde F) \to  \ \Pic^0(F)$. On the other hand
$$
\tilde \sigma^*:  \Pic^0 (\tilde E) \to  \Pic^0(\tilde F) \ \rm and \it  \ \sigma^*:  \Pic^0(E) \to  \ \Pic^0(F)
$$
are injective because $\tilde \sigma$ and $\sigma$ are ramified, cfr. \cite{M8}. It follows from the previous commutative diagram, since $\tilde \sigma_*\tilde \sigma^*$ is just multiplication by 2,  that
$\tilde \sigma^* \Pic^0_2(\tilde E)$ is contained in $\mbox{Ker} \  \ \pi_{F*}$ and finally that
 $$
 \Pic^0_2(C) \cap \tilde \sigma^*  \Pic^0_2(\tilde E) = {\pi_F}^*  \Pic^0_2(E) \cong \mathbb Z_2.
$$
Furthermore, it turns out that $\pi_F^*  \Pic^0_2(E)$ is generated by $L$. Starting from $(F,L_F)$, we have reconstructed  $(C,L)$. It follows that the assignement $(F,L_F) \mapsto (C,L)$ defines a map $\psi$ which is inverse to 
the Coble-Roth map $\phi$. This implies that  $\phi$ is birational, see \cite{D1} for further details on this proof and for the proof that $\mathcal R^b_4$ is rational.
\medskip \par 
\subsection {Rationality of $\mathcal R_4$}\rm
As we already have seen, Prym moduli spaces in very low genus are frequently related to linear systems $\vert C \vert$ of smooth, connected curves on a surface $S$ endowed with a quasi \'etale double cover
$$
p: \tilde S \to S.
$$
An example in genus 3 is the Kummer quartic surface $S \subset \mathbf P^3$. In this case $p$ is the 2:1 cover of $S$ branched on its 16 nodes of $S$ and
$\vert C \vert = \vert \mathcal O_S(1) \vert$. We are going to meet further examples when considering $\mathcal R_g$ for $ g = 4, 5, 6$.  
\par  
The rationality of $\mathcal R_4$ is due to Catanese, \cite{Ca}. In this case $S$ is a fixed surface in $\mathbf P^3$, up to projective equivalence. 
Namely $S$ is the 4-nodal cubic surface, known also as \it Cayley cubic. \rm  Fixing on $\mathbf P^3$ suitable coordinates $(x_1:x_2:x_3:x_4)$, we can assume that the equation of $S$ is
$$
x_1x_2x_3 + x_2x_3x_4 + x_1x_3x_4 + x_1x_2x_4 = 0,
$$
In particular $S$ contains the edges of the tetrahedron
$$
T = \lbrace x_1x_2x_3x_4 = 0 \rbrace
$$
and $ \Sing \ S = \lbrace v_1 \dots v_4 \rbrace$, where $v_i, i = 1 \dots 4$, is a vertex of $T$.   Let
$$
\sigma: S' \to S
$$
be the blowing up of $ \Sing \ S$. We fix our notations as follows:
\begin{enumerate} \it
\item $E := \sum E_{ij}$, $E_{ij}$ being the strict transform of $\overline {v_iv_j}$, $1 \leq i < j \leq 4$,
\item $F := \sum F_i$, where $F_i := \sigma^{-1}(v_i)$,
\item $H \in \vert \mathcal \sigma^* \mathcal O_S(1) \vert$.
\end{enumerate} \par 
We have $\sigma^*T \sim 2E + 3F \sim 4H$ so that $F \sim 4H - 2E - 2F$. Therefore $\mathcal O_{S'}(F)$ is divisible by 2 in $ \Pic \ S'$ and there exists a double covering
$$
p': \tilde S' \to S'
$$
branched on $F$. In particular we have the commutative diagram
$$
\begin{CD}
{\tilde S'}~@>{p'}>> {S'} \\
@VV{\sigma'}V @VV{\sigma}V \\
{\tilde S}~@>p>> S \\
\end{CD}
$$
where $p$ is branched on $ \Sing \ S$. We introduce now some further preliminaries and then we outline the proof of the rationality of $\mathcal R_4$, as it is given
in \cite{Ca}.  \medskip \par 
A smooth general $C \in \vert 2H \vert$ is the pull-back of a general quadratic section of $S$. In particular $T$ defines  on $C$
the effective divisors of degree two:
$$
e_{ij} = C \cdot E_{ij}
$$
and the line bundle $L := \mathcal O_C(D)$, where $$ D := 2H - E + F \sim \frac 12F. $$
\begin{lemma}~$L$ is a non trivial 2-torsion element of $ \Pic^0(C)$: \end{lemma}
\begin{proof} We have $L^{\otimes 2} \cong \mathcal O_C(F)$. The latter is $\mathcal O_C$ because $F$ is effective and $F \cap C$ is empty.
To show that $L$ is not trivial consider $\tilde C := {p'}^*C$ and observe that $p'/ \tilde C: \tilde C \to C$ is induced by $L$. If $L$ is trivial then
we have $\tilde C = C_1 + C_2$, where $p'/C_i: C_i \to C$ is biregular and $C_1 \cap C_2 = \emptyset$. But then $C_1^2C_2^2 - C_1C_2 =
144 > 0$: against Hodge index theorem.  
\end{proof} \par 
As in the previous section we consider now the involution
$$
t_L:  \Pic^2(C) \to  \ \Pic^2(C),
$$
induced by the tensor product with $L$, and the intersection
$$
Z = W^0_2(C) \cap t_L^* W^0_L(C).
$$ 
The self intersection of $W^0_2(C)$ is six, hence  we can expect that $Z$ consists of 6 points. Assume that $C \in \mid 2H \mid$ is general, then this is true:
\begin{lemma} $Z = \lbrace \mathcal O_C(e_{ij}), \ 1 \leq i < j \leq 4 \rbrace$. Moreover one has
$$
\mathcal O_C(e_{12}-e_{34}) \cong \mathcal O_C(e_{13}-e_{24}) \cong \mathcal O_C(e_{14}-e_{23}) \cong L.
$$
\end{lemma} 
\begin{proof} We can write $2H - E = (2E_{12} + E_{13} + E_{14} + E_{23} + E_{24}) - E = E_{12} - E_{34}$. We know that $L \cong \mathcal O_C(2H-E)$, hence
$L \cong \mathcal O_C(e_{12} - e_{34})$. The same argument works for the other $e_{ij}$'s. It remains to show that $Z$ is finite: since $C$ is not hyperelliptic the
map $u: C \to \mathbf P^2$, defined by $\omega_C \otimes L$, is a morphism, \cite{CD} 0.6.1. Moreover $u$ is not injective at each $f \in Z$, see \cite{CD} 0.6.5. 
But then $u(C)$ is a plane cubic and $C$ is bielliptic. This is impossible, since  $\vert 2H \vert$ dominates $\mathcal M_4$ and $C$ is general. \end{proof} \par 
Now observe that $\vert H \vert$ is the anticanonical system of $S'$. In particular $\vert H \vert$ is invariant by the automorphism group $Aut \ S'$. It follows that
$Aut \ S'$ is isomorphic to the symmetric group $S_4$ of all projectivities  fixing the point $(1:1:1:1)$ and the set $ \Sing \ S$ of the vertices of the tetrahedron $T$. In particular $S_4$ acts on $\vert 2H \vert$
via this isomorphism. \par  
\begin{lemma} Let $C$, $C'$ be general in $\vert 2H \vert$. Then the Prym pairs $(C, \mathcal O_{C}(D))$ and $(C', \mathcal O_{C'}(D))$ are isomorphic
 if and only if $C' = a(C)$ for some $a \in S_4$.
\end{lemma}
\begin{proof} Assume the two pairs are isomorphic. Since $C$ and $C'$ are canonically embedded, there exists a projective automorphism
$a$ such that $a(C) = C'$ and $a^*\mathcal O_C(D) \cong \mathcal O_{C'}$. Then, by the previous lemma, it follows that $a(T) = T$.
Hence $S'' = a(S)$ is a cubic through $T$ and contains $C' \cup  \Sing \ T$. Then $S'' =  S$ for degree reasons and $a \in Aut \ S'$. The converse is obvious.
\end{proof} \par 
\begin{theorem} \ \begin{enumerate} \item $\vert 2H \vert / S_4$ is birational to $\mathcal R_4$,
\item $\vert 2H \vert / S_4$ is rational. 
\end{enumerate}\end{theorem}
\begin{proof} (1) By the previous lemma there exists a generically injective rational map
$
\phi: \vert 2H \vert / S_4 \to \mathcal R_4
$
sending $C$ to the moduli point of $(C, \mathcal O_C(D))$. Since $\mathcal R_4$ and $\vert 2H \vert$ are irreducible of the same dimension, (1) follows.  
\par 
(2) Let $l_1, l_2, l_3$ be independent linear forms on $\mathbf P^2$ and let $l_4 = l_1 + l_2 + l_3$. Then $S'$ is the image of the rational map
$
\phi: \mathbf P^2 \to \mathbf P^3
$
defined by the linear system of cubics passing through the nodes of the quadrilateral $l_1l_2l_3l_4 = 0$, \cite{D}Ê8.2. The strict transform of $\vert 2H \vert$ by 
$\phi$ is the 9-dimensional linear system 
$$
l_1l_2l_3l_4 q + z_1(l_1l_2l_3)^2 + z_2(l_1l_2l_4)^2 + z_3(l_1l_3l_4)^2 + z_4(l_2l_3l_4)^2 = 0,
$$
where $q$ is any quadratic form in $l_1, l_2, l_3$. It is the linear system of all sextics which are singular at these six nodes. $S_4$ acts on it just by permutations of $\lbrace l_1 \ l_2 \ l_3 \ l_4 \rbrace$. In \cite{Ca}, using
this description of $\vert 2H \vert$, the field of invariants of $\vert 2H \vert$ is determined and its rationality is  proved.
\end{proof} \par 
The theorem implies the rationality of $\mathcal R_4$. Note that in this case a \it unique \rm linear system of curves, on a fixed suitable surface, dominates
the moduli space we are considering.  We pass now to some unirationality questions.
\subsection{Unirationality of $\mathcal R_6$ via Enriques surfaces}
The unirationality of $\mathcal R_6$  was first proved  by Donagi, see \cite{Do1}. This proof is related to nets of
quadrics and conic bundles, as we will see later. \par Previously  we want still to play with \'etale or quasi \'etale double coverings of surfaces,  in 
particular with Enriques surfaces and their K3 covers. Therefore we partially reproduce a second proof  given in \cite{Ve3}. This proof relies on Enriques
surfaces. We describe it, with some variations, through various steps. Some of them are just preliminaries on Enriques surfaces. \medskip \par 
\underline{\it Step \ 0}: {\it  Enriques surfaces and $\mathcal R_g$} \par 
The canonical bundle $\omega_S$ of an  Enriques surface $S$ defines a non trivial  \'etale double covering $ p: \tilde S \to S$ such that $\tilde S$ is a K3 surface. \par Let $C \subset S$ be any smooth, connected curve of genus $g \geq 2$. It is well known that then $ L = \omega_S \otimes \mathcal O_C$ is a non trivial 2-torsion element of  $ \Pic^0_2(C)$.  Notice also that $ \ \dim \vert C \vert = g - 1$, cfr. \cite{CD} ch. I. \par  We recall that any moduli space of polarized Enriques surfaces $(S, \mathcal O_S(C))$ is 10 -dimensional. A naif count of parameters suggests that the moduli space of pairs $(S,C)$ dominates $\mathcal R_g$ for $g \leq 6$. In the sequel we  prove that  an irreducible component of it, containing the open set of pairs $(S,C)$ such that $C$ is very ample,  is indeed unirational and dominates $\mathcal R_g$ for $g = 6$. 
 \rm \medskip  \par 
{\underline  {\it Step 1}: \it Enriques and Fano polarizations} \par  At first we consider polarized Enriques surfaces $(S, \mathcal O_S(H))$, where $H$ is an integral curve of genus 4. These are related again to the tetrahedron
$$
T \subset \lbrace x_1x_2x_3x_4 = 0 \rbrace \subset \mathbf P^3.
$$
Let $S'$ be a general sextic passing doubly through the edges of $T$. Then the normalization of $S'$ is a general Enriques surface $S$
and the pull-back of $\mathcal O_{S'}(1)$ on $S$ is a polarization $\mathcal O_S(H)$ of genus 4, see \cite{CD} 4 E. \par 
We will say that $S'$ is an \it Enriques sextic. \rm Writing explicitely its equation, the linear system of Enriques sextics passing doubly through $T$ is
$$
qx_1x_2x_3x_4 + a_1(x_2x_3x_4)^2 + a_2(x_1x_3x_4)^2 + a_3(x_1x_2x_4)^2 + a_4 (x_1x_2x_3)^2 = 0,
$$
where $q$ is a quadratic form  and $a_1, a_2, a_3, a_4$ are constants.
\begin{definition} $\mathbb E$ is the previous linear system of sextic surfaces. \end{definition} \par 
$\mathbb E$ dominates the moduli of polarized Enriques surfaces $(S, \mathcal O_S(H))$, which is therefore unirational. 
\begin{definition} A Fano polarization on an Enriques surface $S$ is a line bundle $\mathcal O_S(C)$ defined by a smooth, integral curve $C$ of genus 6. \end{definition} \par 
We recall that an Enriques surface $S$ is said to
be  \it nodal \rm if $S$ contains an \it effective \rm divisor $R$ such that $R^2 = -2$. A general Enriques surface is not nodal. Nodal Enriques surfaces
contain a copy of $\mathbf P^1$. Let $(S, \mathcal O_S(C))$ be a Fano polarization. For simplicity we will assume that:  
\begin{itemize} \it
\item[$\circ$] $S$ is not nodal, 
\item[$\circ$] $\mathcal O_S(C)$ is very ample.
\end{itemize}
 \par 
A general Enriques surface $S$ has finitely many Fano polarizations $\mathcal O_S(C)$ modulo automorphisms. $\vert C \vert$ defines an embedding
$$
S \subset \mathbf P^5
$$
with the properties summarized in the next theorem, see \cite{CD} 4 H and \cite{CV}.
 \begin{theorem} \ \\ \rm (1) \it
 $S$ contains exactly 20 plane cubics, coming in pairs   $(E_n, E'_n)$ such that $$ \mathcal O_S(E_n - E'_n) \cong \omega_S, \ n = 1 \dots 10. $$
 \rm (2) \it Every trisecant line to $S$ is trisecant to one of them. \end{theorem} \par 
 Furthermore let $ \vert H \vert = \vert E_i + E_j + E_k \vert$ with $1 \leq i < j < k \leq 10$.  Then $\vert H \vert$ defines a generically injective morphism $ f: S \to \mathbf P^3 $ with
 the properties listed below, see \cite{CD} 4 H.   
  \begin{theorem} Let
 $$
 S' := f(S) \subset \mathbf P^3,
 $$
 up to projectivities we have:
 \begin{enumerate} \it
 \item $f: S \to S'$ is finite,
 \item $S'$ is an Enriques sextic and belongs to $\mathbb E$,
 \item $f(E_n)$, $f(E'_n)$ are skew edges of $T$, $n = i, j, k$. 
 \end{enumerate}
 \end{theorem}\par 
 The same properties hold replacing $E_n$ by $E'_n$ for some $n = i, j, k$.  Let
 $ C \in \vert \mathcal O_S(1) \vert $ be a hyperplane section of $S \subset \mathbf P^5$. Considering $f(C)$ we have:
 
\begin{lemma} \rm \ \par (1) \it The only element of $\mathbb E$ containing $f(C)$ is $S'$.  \par
\rm (2) \it $\mathcal O_C(H)$ is non special so that $h^0(\mathcal O_C(H)) = 4$. \par
\rm (3) \it $\mathcal O_C(H)$ is very ample for a general $C$. \end{lemma}
 \begin{proof} (1) Assume $f(C) \subset S''$, where $S'' \in \mathbb E$ and $S'' \neq S'$. Then $f^*S''$   is the curve  $ 4\sum (E_n + E'_n) + C + R \in  \vert 6H \vert$, with
 $R \in \vert 2H - C \vert$ and $R^2 = -2$. This is a contradiction because we assume that $S$ is non nodal. \par
(2)  Since $(H-C)^2 = -2$ and $S$ is not nodal, it follows from Riemann-Roch that $h^i(\mathcal O_S(H-C)) = 0, i = 0, 1, 2$.  Then the non speciality of $\mathcal O_C(H)$ follows from  the standard exact sequence
$$
0 \to \mathcal O_S(H-C) \to \mathcal O_S(H) \to \mathcal O_C(H) \to 0.
$$
(3) $C$ is very ample and $f: S \to S'$ is a finite,  generically injective morphism. Then, counting dimensions, the locus of all $D \in \vert C \vert$ such that $f/D$ is not an embedding is a proper closed subset.
\end{proof} \medskip  \par 
\underline{\it Curves of genus 6 and degree 9 on Enriques sextics} \par 
Keeping the previous notations, we want to study the embedded curve
$$
C' := f(C) \subset \mathbf P^3.
$$
At first we fix some further conventions: we denote by $L_m$ and $L'_m$ the skew 
 edges of $T$ such that $$ f^*L_m = E_m \   , \  f^*L'_m = E'_m,  \  \ m = i, j, k. $$ 
 We also assume that $L_i, L_j, L_k$ are \it coplanar \rm lines. It is clear that
 \begin{proposition} $f^* \mathcal O_{C'}(1) \cong \mathcal O_C(e_i + e_j + e_k)$, so that $(f/C)^*(L_m) = e_m$ and $(f/C)^*(L'_m) = e'_m$. In particular each edge of
 $T$ is  trisecant to $C'$.
 \end{proposition}
  Finally we put $$ e_n = E_n \cdot C \ , \ e'_n = E'_n \cdot C, \ n = 1 \dots 10. $$  
   Since $E_n - E'_n \sim E'_n - E_n$ is a canonical divisor on $S$, the line bundle $$ L := \mathcal O_C(e_n-e'_n) $$ is a non trivial 2-torsion element of
$ \Pic^0(C)$.  The tensor product by $L$ defines a translation
$
t_L:  \ \Pic^3(C) \to  \ \Pic^3(C).
$
For a general $C$ we consider $W^0_3(C) \subset  \Pic^3(C)$ and, as previously,  the intersection
$$
Z := W^0_3(C) \cap t^*_L W^0_3(C).
$$
\begin{proposition} $ Z = \lbrace e_n, \ e'_n, \ n = 1 \dots 10 \rbrace$ and the intersection is transversal.
\end{proposition}
\begin{proof} It is obvious that $e_n, e'_n \in Z$. Since the self intersection of $W^0_3(C)$ is 20, the statement follows if $Z$  is
finite. To prove this recall that $C \subset S \subset \mathbf P^5$ and that $\mathcal O_C(1) \cong \omega_C \otimes L$. It is standard to check that $e \in Z$  if and only if
$L \cong \mathcal O_C(e-e')$ with $e'$ effective  if and only if $e$ is contained in a trisecant line to $C$. But it is shown in \cite{CV} that every trisecant line to a smooth,
non nodal Fano model $S \subset \mathbf P^5$ is contained in the  plane spanned by one cubic $E_n$ or $E'_n$. Hence $e$ is $e_n$ or $e'_n$, for some $n = 1 \dots 10$, and $Z$ is finite.  \end{proof} \par 
\begin{definition} $\mathbb F$ is the family of all curves $C' \subset \mathbf P^3$ such that \begin{enumerate}
\item $C'$ is smooth, connected of degree nine and genus six,
\item $C'$ is contained in a not nodal Enriques sextic $S' = f(S) \in \mathbb E$,
\item $C' = f(C)$ where $\mathcal O_S(C)$ is a Fano polarization.
\item $C f^*L_m = C f^* L'_m = 3$, $m = i, j, k$.
\end{enumerate}
\end{definition} \par 
We point out that condition (2) implies that the Enriques sextic $S' \in \mathbb E$ is \it unique, \rm this is due to lemma 4.14.  We recall also that $f: S \to S'$ is just the normalization map.
\begin{definition}Ê$\tilde {\mathbb E}$ is the family of pairs $(S', \mathcal O_S(C))$ such that \begin{enumerate} \item $S' \in \mathbb E$ is a non nodal Enriques sextic,
\item $f: S \to S'$ is the normalization map, \item $C' := f(C)$ belongs to $\mathbb F$.
\end{enumerate}
\end{definition}
Since $C' \in \mathbb F$ uniquely defines $S'$, we have a morphism
$$
p: \mathbb F \to \tilde {\mathbb E}
$$
sending $C'$ to the pair $(S', \mathcal O_S(C))$. The fibre of $p$ at $(S', \mathcal O_S(C))$ is a non empty open set of the 5-dimensional  linear system $ \vert C \vert $. \par It is standard to construct  the
family $q: \mathcal S \to \tilde {\mathbb E}$ which is the normalization of the universal Enriques sextic $\mathcal S' \subset \tilde {\mathbb E} \times \mathbf P^3$. It is standard as well to construct a line 
bundle $\mathcal C$ over $\tilde {\mathbb E}$ whose restriction to the fibre $S$ of $q$  at $x := (S', \mathcal O_S(C))$ is $\mathcal O_S(C)$.  Then Grauert theorem 
implies that, on a dense open set of $\tilde {\mathbb E}$, $q_* \mathcal C$ is a vector bundle with fibre $H^0(\mathcal O_S(C))$ at $x$. $\mathbf P \mathcal C$ is clearly birational to
$\mathbb F$, hence it follows that:
 \begin{lemma} $\mathbb F$ is birational to $\mathbf P^5 \times \tilde{\mathbb E}$. \end{lemma}
  \par 
On the other hand the assignement $C' \mapsto (C,L)$ defines the moduli map
$$
m: \mathbb F \to \mathcal R_6,
$$ 
where $C$ and $L \cong \mathcal O_C(e_n - e'_n)$ are obtained as above from $C '$. 
 \begin{proposition} $m$ is dominant. \end{proposition}
\begin{proof} Let $C' \in \mathbb F$ and let $(C,L)$ be in  the isomorphism class of $m(C')$. Then , keeping the previous notations, $C'$ is embedded in
$\mathbf P^3$ by a line bundle $\mathcal O_C(a_i + a_j + a_k)$ for some $a_i, a_j, a_k$ in $Z := W^0_3(C) \cap t_L^* W^0_3(C)$. By proposition 4.16 
$Z$ is finite. Moreover we know that $h^0(\mathcal O_{C'}(1)) = 4$ by lemma 4.14. Let $C'' \in \mathbb F$, then $m(C'') = m(C')$  if and only if $C''$ is biregular
to $C'$ and $\mathcal O_{C''}(1) \cong \mathcal O_{C'}(b_i + b_j + b_k)$ for some $b_i, b_j, b_k \in Z$. In particular $C''$ and $C'$ are projectively equivalent  if and only if
$C'' = \alpha(C')$ for some $\alpha \in Aut \ T$, the group of projectivities of $T$. Since $Aut \ T$ is 3-dimensional and $Z$ is finite, it 
follows that the fibre of $m$ over $m(C')$ is 3-dimensional. Hence $m$ is dominant  if and only if $ \ \dim \ \mathbb F \geq  \ \dim \ \mathcal R_6 +  \ \dim \ Aut \ T = 18$. 
To prove this note that the natural projection   $u: \tilde{\mathbb E} \to \mathbb E$ has finite fibres. Indeed the fibre of $u$ at $S'$ is the set of Fano polarizations 
of $S$. Hence $ \ \dim \ \tilde {\mathbb F} =  \dim \ \mathbb F = 13$ and $ \ \dim \ \mathbb F = 18$.
 \end{proof} \par  
\begin{theorem} $\tilde {\mathbb E}$ is irreducible and rational. \end{theorem}  \par 
This theorem is proven in \cite{Ve3}.   Since $\tilde {\mathbb E}$ is rational then $\mathbb F$ is rational too. Since $m: \mathbb F \to \mathcal R_6$ is dominant it follows:
\begin{corollary} $\mathcal R_6$ is unirational. \end{corollary} 
  \medskip  \par 
  \subsection{The unirationality of $\mathcal R_g$ via conic bundles: $g \leq 6$ }
We continue along the path of the unirationality of $\mathcal R_g$ adding more facts and  a few historical remarks. Then we will discuss some well known relations between 
conic bundles and Prym pairs.  Finally we will complete our picture giving a simple proof via conic bundles of the following known result
  \begin{itemize} \it  \item[$\circ$] $\mathcal R_g$ is unirational for $g \leq 6$. \end{itemize} \par  
The proof uses linear systems of nodal conic bundles in  $\mathbf P^2 \times \mathbf P^2$. So it follows, in some sense, the spirit of classical parametrizations of $\mathcal M_g$ via
families of nodal plane curves, we started with in the first part of this exposition. \medskip \par 
\underline{\it A brief  historical overview} \par
The first proof of the unirationality of $\mathcal R_6$, given by Donagi in \cite{Do1}, goes back to 1984. A further purpose of this paper is  to prove the unirationality of $\mathcal A_5$, 
the moduli space of principally polarized abelian varieties. The author considers Fano threefolds, of index one and of the principal series,
$$
V \subset \mathbf P^7
$$
which are a complete intersection of type $(1,1,2)$ in the Grassmannian $G(2,5)$. Among them one has the family of those threefolds having five ordinary double points. The purpose is to show that the intermediate
Jacobian of a minimal desingularization of a general threefold in this family is in fact a general 5-dimensional principally polarized abelian variety. As we will see, this is actually the case. \par
Projecting from one of the nodes of $V$, one obtains a complete intersection of three quadrics in $\mathbf P^6$. Actually this net of quadrics is quite special.
The discriminant curve of the net is a plane curve
$
\Gamma \cup P,
$
where $P$ is a line and $\Gamma$ is a 4-nodal sextic. It turns out that the normalization $C$ of $\Gamma$ is endowed with a non split \'etale 2:1 cover
$$
\pi: \tilde C \to C.
$$
$\Gamma$ parametrizes a family of singular quadrics $Q_x, x \in \Gamma$,  of the net.  $\tilde C$ is the normalization of the family of rulings of maximal linear subspaces of the quadrics $Q_x$,
$x \in \Gamma$. An important fact is that every such a $Q_x$ has rank 6, even when $x$ is singular for $\Gamma \cup P$. This implies that $\pi$ is \'etale. \par
Let $L$ be the non trivial line bundle defining $\pi$. Then $(C,L)$ is a Prym pair. The family of 5-nodal Fano threefolds $V$ is easily seen to be rational. Hence, for proving in this way
the unirationality of $\mathcal R_6$ and of $\mathcal A_5$, the crucial point is to show that the previous Prym pais $(C,L)$ is general in moduli. This step is done in \cite{Do1} 
via an appropriate algorithm.  
We will see in a moment a simple variation of this method. \par Before we want to complete our picture with a brief report on the unirationality of $\mathcal R_5$.  
The first complete proof of the unirationality of $\mathcal R_5$ is quite new. This result is due to Izadi, Lo Giudice and Sankaran, \cite{ILS}. An independent
proof follows from \cite{Ve4}. \par All the subject, as well as some of the proofs, was strongly influenced by the study of rational parametrizations for the
moduli spaces $\mathcal A_g$ of principally polarized abelian varieties of dimesion $g \leq 6$. In \cite{C} Clemens was the first to produce a rational
parametrization of the moduli space $\mathcal A_4$. He studies a special family of threefods, namely quartic double solids
$$
f: V \to \mathbf P^3,
$$
branched on a quartic surface $B$ with 6 nodes. The method is the same used by Donagi for $\mathcal A_5$: the intermediate Jacobian of the blowing up
of $ \Sing \ V$ is 4-dimensional. Then the main point is to show that the period map
$$
p: \mathcal V \to \mathcal A_4
$$
is dominant. Also in this case Prym pairs are present in the construction, though not considered in the proof. Projecting from a node of $B$ one can
realize a birational model of $V$ as a singular conic bundle over $\mathbf P^2$.  In this case the discriminant $\Gamma  \subset \mathbf P^2$ is a sextic 
with 5 nodes. Its normalization $C$ is endowed with an \'etale double cover $\pi: \tilde C \to C$, parametrizing the irreducible components of the singular fibres of this conic bundle.
\medskip \par 
\underline{ \it Families of nodal conic bundles on $\mathbf P^2$ and $\mathcal R_g$} \par
We want to see, more in general, some useful families of nodal conic bundles. At first we define the conic bundles on $\mathbf P^2$ we want. Consider a rank 3 vector bundle $E$ 
over $\mathbf P^2$ and the natural projection $p: \mathbf PE \to \mathbf P^2$. Let
$$ V \subset \mathbf PE$$ be a hypersurface, we introduce the following
\begin{definition} $V$ is an admissible conic bundle of $\mathbf PE$ if:   
\begin{enumerate} \it 
\item  $p/V: V \to \mathbf P^2$ is flat,
\item  each fibre of $p/V$ is a conic of rank $\geq 2$,
\item $ \Sing \ V$ is a finite set of $\delta$ nodes,
\item  the union of the singular fibres of $p$ is an integral surface.
\end{enumerate} 
\end{definition} \par 
By definition the discriminant curve of $V$ is
$$
C' := \lbrace x \in \mathbf P^2 \ / \ rk \ p^*(x) = 2 \rbrace.
$$
To give a $\delta$-nodal conic bundle $V \subset \mathbf PE$ is equivalent to give 
$$
q: E \otimes E \to \mathcal O_{\mathbf P^2}(m),
$$
where $q$ is a quadratic form on $E$ defining the exact sequence
$$
0 \to E \to E^*(m) \to A \to 0.
$$
Furthermore conditions (2) and (4) imply that
 \begin{itemize} \it
\item[$\circ$] $A$ is a line bundle on a $C'$,
\item[$\circ$] $C'$ is an integral $\delta$-nodal curve,
\end{itemize} \par 
It is easy to see that the map $p/ \Sing \ V:  \Sing \ V \to  \Sing \ C'$ is bijective. Moreover $C'$ is endowed with the \'etale double covering
$$
\pi': \tilde C' \to C',
$$
where
$$
\tilde C' = \lbrace l \in \mathbf PE^*_x \ / \ l \subset p^*(x), \ x \in C' \rbrace
$$
is the family of the irreducible components of the singular fibres of $p/V$. $\tilde C'$ is integral by condition (4). 
We denote the normalization map of $C$ as
$$
\nu: C \to C'.
$$
Clearly $\pi'$ is defined by a non trivial 2-torsion element
$$
L' \in  \Pic^0 (C').
$$
Moreover  we have the standard exact sequence
$$
\begin{CD}
0 @>>> {{\mathbf C^*}^{\delta}} @>>> { \Pic^0 (C')} @>{\nu^*}>> { \Pic^0(C)} @>>> 0. \\
\end{CD}
$$
Hence $L'$ defines a line bundle
$$
L := \nu^* L'
$$
such that $L^{\otimes 2} \cong \mathcal O_{\tilde C}$. 
\begin{lemma} $L$ is non trivial. \end{lemma}
\begin{proof} Consider the cartesian square
$$
\begin{CD}
{\tilde C}~@>{\pi}>> C \\
@VV{\tilde \nu}V @VV{\nu}V \\
{\tilde C'} @>{\pi'}>> {C'}, \\
\end{CD}
$$
where $\tilde \nu$ is the normalization map of $\tilde C'$ and $\pi$ is defined by $L$. Since $\tilde C'$ is integral, $\tilde C$ is integral too.
Hence $\pi$ is non split and $L$ is not trivial.
\end{proof} \par  
\begin{definition} $(C, L)$ is the Prym pair of the conic bundle $V$. \end{definition} \par 
We will assume that $C$ has genus $g$. From now on we denote as
$$
\vert V \vert_{\delta}
$$
the family of admissible conic bundles $W \in \vert V \vert$ having exactly $\delta$ nodes.
\begin{definition} The natural map is the morphism $$ m: \vert V \vert_{\delta} \to \mathcal R_g $$ sending an admissible
conic bundle  $V' \in  \vert V \vert$ to the moduli of its Prym pair.
\end{definition} \par 
It is interesting to study $m$ in the simplest cases of $E$. For instance if $E$ is the direct sum of three line bundles.
Let $\mathcal O_{\mathbf PE}(1)$ be the tautological bundle. The linear system $\vert V \vert$ we want  to use is 
$\vert \mathcal O_{\mathbf PE}(2) \vert$.  The case $$ E = \mathcal O_{\mathbf P^2}(1)^{\oplus 3} $$ is already 
quite interesting.  Indeed it provides a quick proof of the unirationality of $\mathcal R_g$, $g \leq 6$. In this case we  have
$$
\vert V  \vert = \mathcal O_{\mathbf P^2 \times \mathbf P^2}(2,2) \vert.
$$
Fixing coordinates $(x,y) := (x_1:x_2:x_3) \times (y_1:y_2:y_3)$, it follows that
$$
\vert V \vert = \lbrace \sum_{1 \leq i < j \leq 3} a_{ij} y_iy_j = 0 \rbrace,
$$
where $M := (a_{ij})$ is a $3 \times 3$ symmetric matrix of quadratic forms in $x$. Let $V$ be standard with discriminant curve $C'$ and let
$L' \in  \Pic^0 (C')$ be the non trivial line bundle considered above. Then $M$ defines the exact sequence
$$
\begin{CD}
0 @>>> {\mathcal O_{\mathbf P^2}^3(-3)}~@>M>> {\mathcal O_{\mathbf P^2}(-1)^3 }~@>>> {L'} @>>> 0. \\
\end{CD}
$$ 
Note that $C'$ is an integral sextic with $\delta$ nodes and that $V$ provides a sextic equation of $C'$ as the determinant of a symmetric matrix of 
quadratic forms. It is well known that to give such a  determinant is equivalent to give a non trivial element of $ \Pic^0_2 (C')$ as $L'$, cfr. \cite{Be1} lemma 6.8 and also \cite{DI} 6.1.
\begin{lemma} The natural map $\vert V \vert_{\delta} \to \mathcal R_g$ is dominant for $g \leq 6$.
\end{lemma}
\begin{proof} Let $(C,L)$ be a general Prym pair, with $C$ of genus $g \leq 6$. Then $C$ is birational to a plane sextic $C'$ with $\delta = 10 - g$ nodes.
Let $L' \in {\nu^*}^{-1}(L)$, where $\nu: C \to C'$ is the normalization map. Then $L'$ defines an exact sequence as above, hence a matrix $M$ as above.
Consider the conic bundle $V$, defined by the equation $x M y = 0$ in $\mathbf P^2 \times \mathbf P^2$. Since $L'$ is a line bundle, it easily follows that 
$V$ is admissible and in $\vert V \vert_{\delta}$. Hence $m_g(V)$ is the moduli point of $(C,L)$ and the statement follows. \end{proof} \par 
Like in the case of nodal plane curves we can now consider the family of $\delta$-nodal conic bundles $V$. Since $ \dim \ \vert V \vert = 35$, we can assume
that the points of $ \Sing \ V$ are in general position for $4 \leq \delta \leq 7$. Then, as in the case of nodal plane curves,  $\vert V \vert_{\delta}$ is rational.
This implies that:
\smallskip

\begin{theorem}~$\mathcal R_g$ is unirational for $g \leq 6$. \end{theorem} \par 
\begin {remark} \rm For $g = 6$ then $ \Sing \ V$ consists of  4 general points. Since all these sets are projectively equivalent in $\mathbf P^2 \times \mathbf P^2$ we can fix one
of them, say $Z = \lbrace o_1, o_2, o_3, o_4 \rbrace$. Let $\mathcal I_Z$ be the ideal sheaf of $Z$. We point out that then \it $\mathcal R_6$ is dominated by a linear system, \rm namely by
$$
\vert \mathcal I_Z^{2}(2,2) \vert.
$$
\end{remark}


\addcontentsline{toc}{subsection}{References}
\begin{thebibliography}{EMSF}

\bibitem[AJ]{AJ} D. Abramovich, T. Jarvis, {\em{Moduli of twisted spin curves}}, Proc. American Math. Soc., \textbf{131} (2003), 685-699


\bibitem[Ar]{Ar} E. Arbarello, {\em{Alcune osservazioni sui moduli delle curve appartenenti a una data superficie algebrica}}, Lincei Rend. Sc. fis. mat. e nat. \textbf{59} 725-732 (1975) 

\bibitem[AC1]{AC1} E. Arbarello, M. Cornalba, {\em{Footnotes to a paper of Beniamino Segre}}, Math. Annalen \textbf{256} 341-362 (1981)

\bibitem[AC2]{AC2} E. Arbarello, M. Cornalba, {\em{Su una congettura di Petri}}, Comm. Math. Helvetici \textbf{56} 1-38 (1981) 

\bibitem[AC3]{AC3} E. Arbarello, M. Cornalba {\em{A few remarks about the variety of irreducible plane curves of given degree and genus}}, Ann. Scient. Ec. Norm. Sup. 16 (1983), 467-488


\bibitem[ACGH]{ACGH} E. Arbarello, M. Cornalba, P.A. Griffiths, J. Harris, {\em{Geometry of Algebraic Curves, vol. I }}, GMW \textbf{267},  Springer New York (1985)

\bibitem[ACG]{ACG} E. Arbarello, M. Cornalba, P.A. Griffiths, {\em{Geometry of Algebraic Curves, vol. II}}, GMW \textbf{268},  Springer New York (2011)

\bibitem[AS]{AS} E. Arbarello, E. Sernesi,  {\em{The equation of a plane curve}} Duke Math. J. \textbf{46} 469-485 (1979)


\bibitem[BDPP]{BDPP} S. Boucksom, J.P. Demailly, M. Paun and T. Peternell, {\em{The pseudo-effective cone of a compact K\"ahler manifold and varieties of negative Kodaira dimension}}, preprint arXiv AG-0405285 (2004)


\bibitem[Be2]{Be2} A. Beauville, {\em{Vari\'et\'es de Prym et jacobiennes interm\'ediares}}, Ann. \`Ecole Norm. Sup. (4), \textbf{10} 149-196 (1977)

\bibitem[Be1]{Be1} A. Beauville, { \em{Determinantal hypersurfaces}} Michigan Math. J. Special volume in honor of W. Fulton  \textbf{48} 39-64 (2000)

\bibitem[BF]{BF} G. Bini, C. Fontanari, { \em{Moduli of Curves and Spin Structures via Algebraic Geometry}}, Transactions AMS \textbf{358} 3207-3217 (2006)

\bibitem[BFV]{BFV} G. Bini, C. Fontanari, F. Viviani  {\em{On the birational geometry of the universal Picard variety}}, preprint arXiv AG-1006.072 to appear in IMRN (2010)

\bibitem[Bo]{Bo} C. B\"ohning, {\em {The rationality of the moduli space of curves of genus 3 after P. Katsylo }}, in Cohomological and geometric approaches to rationality problems, Progr. Math. \textbf{282}, 17-53 BirkhŠuser Boston
(2010)

\bibitem[BC]{BC} I. Bauer, F. Catanese {\em  {The Rationality of certain moduli spaces of curves of genus 3}} in Cohomological and geometric approaches to rationality problems,
Progr. Math. \text{282},  1-16, BirkhŠuser Boston (2010)

\bibitem[BV]{BV} A. Bruno and A. Verra, {\em{$M_{15}$ is rationally connected}}, in Projective varieties with unexpected properties, 51-65, Walter de Gruyter New York (2005).

\bibitem[BaV]{BaV} I. Bauer, A. Verra, { \em { The rationality of the moduli space of genus four curves endowed with an order three subgroup of their Jacobian }}  Michigan Math. J. \textbf{59} 483-504 (2010)

\bibitem[C1]{C1} G. Castelnuovo, {\em{Ricerche generali sopra i sistemi lineari di curve piane}} Mem. R. Accad. delle Scienze di Torino, \textbf{42} 3-42 (1892)

\bibitem[CC]{CC}  L. Chiantini, C. Ciliberto { \em{On the Severi varieties of surfaces in $\mathbf P^3$}}, J. Alg. Geometry, \textbf{8} 67-83 (1999)

\bibitem[CC1]{CC1}  L. Chiantini, C. Ciliberto { \em{On weakly defective varieties}}, Transactions AMS, \textbf{354} 151-178 (2002)

\bibitem[CCC]{CCC} L. Caporaso, C. Casagrande and M. Cornalba, {\em{Moduli of roots of line bundles on curves}}, Transactions AMS \textbf{359} (2007), 3733--3768.

\bibitem[Ca]{Ca} F. Catanese, {\em{On certain moduli spaces related to curves of genus 4}}, in Algebraic Geometry, LNM  \textbf{1008} 30-50 Springer New-York (1983)

\bibitem[Ca1]{Ca1} F. Catanese, {\em {Moduli of algebraic surfaces}} in Theory of moduli, LNM \textbf{1337} 1-83 Springer New-York (1988)

\bibitem[Ch]{Ch}  X. Chen, { \em{Rational curves on K3 surfaces}}, Comm. Algebra \textbf{31} J. Alg. Geom., \textbf{8} 245Ð278 (1999)

\bibitem[Ci1]{Ci1} C. Ciliberto, {\em{Geometric Aspects of Polynomial Interpolation}},  in More Variables and of Waring's Problem  PM \textbf{} Birkh\"auser Boston (2001)


\bibitem[CD]{CD} F. Cossec, I. Dolgachev, {\em{Enriques Surfaces I}} Progress in Mathematics \textbf{76}, Birk\"auser Boston (1989) 


\bibitem[C]{C} M. Cornalba, {\em{Moduli of curves and theta-characteristics}}, in Lectures on Riemann surfaces  560-589, World Scientific Teaneck New Jersey (1989)

\bibitem[CR1]{CR1} M. C. Chang, Z. Ran, {\em{Unirationality of the moduli space of curves of genus 11, 13 (and 12)}}, Inventiones Math. \textbf{76} 41-54  (1984)


\bibitem[CR2]{CR2} M. C. Chang, Z. Ran, {\em{On the slope and Kodaira dimension of $M_g$ for small $g$}}, J. Differential Geom.  \textbf{34} 267-274  (1991)


\bibitem[CS]{CS} L.Chiantini, E.Sernesi,  {\em{Nodal curves on surfaces of general type}} Math.Annalen \textbf{307} 41-56 (1997) 

\bibitem[CV]{CV} A. Conte, A. Verra, {\em{Reye constructions for nodal Enriques surfaces}},  Transactions American Math. Soc. \text{336} 79-100 (1993)

\bibitem[DH]{DH} S. Diaz, J. Harris, {\em{Geometry of the Severi varieties}}, Transactions AMS, \textbf{309}  1-34 (1988)

\bibitem[D]{D} I. Dolgachev, {\em{Topics in classical algebraic geometry I}}, to appear in Cambridge UP, available at http://www.math.lsa.umich.edu/\~ \ idolga/.

\bibitem[D1]{D1} I. Dolgachev, {\em{ Rationality of $\mathcal R_2$ and $\mathcal R_3$}} Pure Appl. Math. Q. \textbf{4} 501-508 (2008)

\bibitem[DI]{DI} O. Debarre, A. Iliev, {\em{On nodal prime Fano threefold of degree 10}} preprint (2010)

\bibitem[Do1]{Do1} R. Donagi, {\em{The unirationality of $\mathcal A_5$}}, Annals of Math. \textbf{119} 269-307 (1984)

\bibitem[Do2]{Do} R. Donagi, {\em{The fibre of the Prym map}}, in Journees de Geometrie Algebriques  D'Orsay, Asterisque  \textbf{218} 145-175 SMF Paris  (1993).

\bibitem[DS]{DS} R. Donagi, R. Smith, {\em{The fibre of the Prym map}} Acta Math. \textbf{145} 26-102 (1981)

\bibitem[EH]{EH} D. Eisenbud, J. Harris, {\em {The Kodaira dimension of the moduli space of curves of genus $\geq 23$}} Invent. Math. \textbf{90} 359-387  (1987)



\bibitem[F1]{F1} G. Farkas, {\em{The birational trype of the moduli space of even spin curves}},  Advances in Mathematics, \textbf{223},  (2010), 433-443

\bibitem[F2]{F2} G. Farkas, {\em {Aspects of the birational geometry of $M_g$}}  in 'Geometry of Riemann surfaces and their moduli spaces' Surveys in Diff. Geom. \textbf{14} 57-111 (2010)
 
\bibitem[F3]{F3} G. Farkas, {\em{Koszul divisors on moduli spaces of curves}}, American J. of Math.  \textbf{131}  819-867(2009) 

\bibitem[F4]{F4} G. Farkas, {\em{Syzygies of curves and the effective cone of $\overline {\mathcal M}_g$}}, Duke Math. J. \textbf{135} (2006), 53-98.

\bibitem[F5]{F5} G. Farkas, {\em{Brill-Noether geometry on moduli space of spin curves}}, in 'Classification of algebraic varieties'  EMS Congress Reports 259-276 (2011)



\bibitem[F6]{F6} G. Farkas {\em{The geometry of the moduli space of curves of genus 23}} Math. Annalen  \textbf{318} 43-65 (2000)

\bibitem[FKPS]{FKPS} F. Flamini, A. Knutsen, G. Pacienza, E. Sernesi, {\em~{Nodal curves with general moduli on K3 surfaces}}, Comm. in Algebra, \textbf{36} 3955-3971 (2008)

\bibitem[FP]{FP} G. Farkas and M. Popa, {\em{Effective divisors on $\mathcal M_g$, curves on K3 surfaces and the Slope Conjecture}}, J. of Algebraic Geometry \textbf{14 } 241-267 (2005)

\bibitem[FV1]{FV1} G. Farkas, A. Verra,  {\em{The classification of universal Jacobians over the moduli space of curves}} arXiv AG-1005.5354 to appear in Commentarii Math. Helvetici (2010)

\bibitem[FV2]{FV2} G. Farkas, A. Verra, {\em{The geometry of the moduli space of odd spin curves}} preprint arXiv AG-1004.0278 (2010)

\bibitem[FV3]{FV} G. Farkas, A. Verra, {\em{Moduli of theta characteristics via Nikulin surfaces}} Math. Annalen to appear (2012)


 \bibitem[Fu1]{Fu1} W. Fulton, {\em{On nodal curves}}, in 'Algebraic Geometry - Open Problems'  Springer LNM \textbf{997}, 542-575 (1983)

\bibitem[Fu2]{Fu2} W. Fulton, {\em{Hurwitz schemes and irreducibility of moduli of algebraic curves}}, Annals of Math. \textbf{90}, 542-575 (1969)

\bibitem[Ge]{Ge} F. Geiss, {\em {The unirationality of Hurwitz spaces of 6-gonal curves of small genus}} preprint arXiv AG-1109.3603 (2011)

\bibitem[GL]{GL} M. Green,  R. Lazarsfeld, {\em{Special divisors on curves on a K3 surface}}, Inventiones Math. \textbf{89}  357-370 (1987)

\bibitem[GL1]{GL1} M. Green, R. Lazarsfeld {\em{ Some results on the syzygies of finite sets and algebraic curves}},  Compositio Math. \textbf{67} 301-314 (1988) 

\bibitem[H1]{H1} J. Harris, {\em{On the Kodaira dimension of the moduli space of curves II: The even genus case}}, Inventiones  Math. \textbf{75} 437-466 (1984)

\bibitem[H2]{H2} J. Harris, {\em{On the Severi problem}}, Inventiones  Math. \textbf{84} 445-461 (1984).

\bibitem[HM]{HM} J. Harris,  D. Mumford, {\em{On the Kodaira dimension of the moduli  space of curves}}, Inventiones Math. \textbf{67} 23-88 (1982), 

\bibitem[HMo]{HMo} J. Harris, I. Morrison, {\em{Slopes of effective divisors on the moduli space of curves}}, Inventiones Math.  \textbf{99 } 321-355 (1990)


\bibitem[Hu]{Hu} R. Hudson, {\em{Kummer's Quartic Surface}}, Cambridge UP Cambridge (1905) 

\bibitem[ILS]{ILS} E. Izadi, A. Lo Giudice, G. Sankaran {\em{The moduli space of \'etale double covers of genus 5 curves is unirational}},  Pacific J. of Math. \textbf{239}  39-52 (2009)

\bibitem[K]{K} J. Kollar, {\em {Which are the simplest algebraic varieties?}} Bull. Amer. Math. Soc. \textbf{38} 409-433 (2001).

\bibitem[K1] {K1} P. Katsylo, {\em{Rationality of fields of invariants of reducible representations of the group SL2}}. Vestnik Moskov. Univ. Ser. I Mat. Mekh.  \textbf{5} 77-79 (1984)

\bibitem[K2]{K2} P. Katsylo,  {\em{On the unramified 2-covers of the curves of genus 3}}, in Algebraic Geometry and its Applications,  Aspects of Mathematics \textbf{25} 61-65 Vieweg Berlin (1994)

\bibitem[LO]{LO} H. Lange, A. Ortega, {\em{Prym varieties of cyclic coverings}} Geometriae Dedicata \textbf{150} 391-403 (2011)




\bibitem[MM]{MM} S. Mori, S. Mukai, {\em{ The uniruledness of the moduli space of curves of genus 11}}, in \it Algebraic Geometry (Tokyo/Kyoto 1082) \rm LNM \textbf{1016} 334-353 Springer, Berlin (1983)


\bibitem[Mu1]{Mu1} S. Mukai, {\em{Curves, K3 surfaces and Fano manifolds of genus $\leq 10$}}, in Algebraic geometry and commutative algebra I  367-377 Kinokunya Tokyo (1988)

\bibitem[Mu2]{Mu2}  S. Mukai,{\em {Curves and K3 surfaces of genus eleven}}, in 'Moduli of vector bundles', Lecture Notes in Pure and Applied Math. \textbf{179}, Dekker 189-197 (1996) 

\bibitem[Mu3]{Mu3} S. Mukai, {\em {Non-abelian Brill-Noether theory and Fano 3-folds}} Sugaku Expositions \textbf{14} 125-153 (2001)

\bibitem[Mu4]{Mu4} S. Mukai, {\em{Curves and Grassmannians }}, in 'Algebraic Geometry and related topics'  19-40 International Press, Boston (1992)  





\bibitem[M7]{M7} D. Mumford, {\em{Theta-characteristics on algebraic curves}}, Ann. Ecole Norm. Sup (4) \textbf{4} (1971),181-192.

\bibitem[M8] {M8}D. Mumford, {\em{Prym Varieties I, in Contributions to Analysis}}(New York, 325-350 (1974)
 
\bibitem[N]{N} V. Nikulin, {\em {Finite groups of automorphisms of Kaehlerian K3 surfaces}},  Trans. Moscow Math.Soc. \textbf{38} 71Ð135 (1980)


\bibitem[Re]{Re} S. Recillas, { \em{ Jacobians of curves with a $g^1_4$ are Prym varieties of trigonal curves}}, Bol. Soc. Mat. Mexicana \textbf{19} 9Ð13 (1974)

\bibitem[Sch]{Sch} F.O. Schreyer, {\em {Computer Aided Unirationality Proofs of Moduli Spaces} } in Handbook of Moduli, International Press, Boston (2011) 

\bibitem[SB]{SB} N. Shepherd-Barron, {\em{Invariant theory for $S_5$ and the rationality of $\mathcal M_6$}}, Compositio Math. \textbf{70} 13-25 (1989)

\bibitem[SD]{SD} B. Saint-Donat, {\em {Projective models of K3 surfaces}}, Amer. J. Math \textbf{96} 602-639 (1974) 

\bibitem[S]{S} F. Severi, {\em{Sulla classificazione delle curve algebriche e sul teorema di esistenza di Riemann}}, Atti R. Acc. dei Lincei, serie V \textbf{24} 877-888 (1915)

\bibitem[S1]{S1} F. Severi, {\em{Vorlesungen \"uber Algebraische Geometrie}}, Teubner, Leipzig (1921)

\bibitem[Se]{Sg}  C. Segre, { \em {Sulla varieta' cubica con dieci punti doppi dello spazio a quattro dimensioni}} , Atti della R. Acc. delle Scienze di Torino, \textbf{22}, 791-801 (1886-87). 

\bibitem[Se1]{Se1} B. Segre, {\em{Sui moduli delle curve algebriche}}, Ann. di Matematica \textbf{4}  (1930), 71-102.

\bibitem[Se2] {Se2} B. Segre,  {\em{Alcune questioni su insiemi finiti di punti in geometria algebrica}}, Rend. Sem. Mat. Univ. Torino \textbf{20} 67-85 (1960/61)

\bibitem[Se3]{Se3} B. Segre, { \em{ Sui moduli delle curve poligonali, e sopra un complemento al teorema di esistenza di Riemann}} Math. Annalen \textbf{100} 537-551 (1928)

\bibitem[Ser1]{Ser1} E. Sernesi, {\em{L'unirazionalit\`a della variet\`a dei moduli delle curve di genere 12}}, Ann. Scuola Normale Sup. Pisa, \textbf{8} 405-439 (1981)

\bibitem[Ser2]{Ser2} E. Sernesi, {\em{Deformation of Algebraic Schemes}}  GMW \textbf{334} Springer, Berlin (2006)


\bibitem[Ser3]{Ser3} E. Sernesi, {\em{Moduli of rational fibrations}}, preprint arXiv AG-0702865 (2007)

\bibitem[T]{T} A. Terracini, {\em{Su due problemi, concernenti la determinazione di alcune classi di superficie considerati da G. Scorza e da F. Palatini}}, Atti Soc. Nat. e Mat. Modena, \textbf{6} (1922) 

\bibitem[Tan1]{Tan1}  A. Tannenbaum,  {\em{ Families of algebraic curves with nodes}}, Compositio Math. \textbf{41}, 107-126 (1980)

\bibitem[Tan2]{Tan2}  A. Tannenbaum, { \em{Families of curves with nodes on K3 surfaces}}  Math. Annalen \textbf{260} 239-253 (1982)

\bibitem[VdG]{VdG} G. Van der Geer, {\em{On the geometry of a Siegel modular threefold}},  Math. Annalen  \textbf{260} 317-350 (1982)

\bibitem[vGS]{vGS} B. van Geemen, A. Sarti, {\em{ Nikulin involutions on K3 surfaces}} Math. Zeitschrift \textbf{255} 731-753 (2007)

\bibitem[Ve1]{Ve1} A. Verra, {\em{The unirationality of the moduli space of curves of genus 14 or lower}}, Compositio Math.  \textbf{141} (2005), 1425-1444.

\bibitem[Ve2]{Ve2} A. Verra, {\em{ The fibre of the Prym map in genus three}} Math. Annalen \textbf{276} 433-448 (1983)

\bibitem[Ve3]{Ve3} A. Verra, {\em{A short proof of the unirationality of $\mathcal A_5$}}, Indagationes Math. \textbf{46} 339-355 (1984)

\bibitem[Ve4]{Ve4} A. Verra, {\em{On the universal principally polarized abelian variety of dimension 4}}, Curves and Abelian Varieties,  Contemporary Math.  \textbf{465} 253-274 (2008)

\bibitem[We]{We} G. Welters, { \em{Curves of twice the minimal class on principally polarized abelian varieties}} Nederl. Akad. Wetensch. Indag. Math. \textbf{49} 87-109 (1987)

\end{thebibliography}
\end{document}